\documentclass[leqno,a4paper]{article}
\usepackage{amsmath,amsthm,amssymb}
\usepackage[utf8]{inputenc}
\usepackage[T1]{fontenc}

  \usepackage{hyperref}
 \usepackage[svgnames,x11names]{xcolor}

 \hypersetup{
   colorlinks, linkcolor={DarkSlateBlue},
   citecolor={DarkSlateBlue}, urlcolor={DarkSlateBlue}
  }

\usepackage{enumerate}
\usepackage{bbm}
\usepackage{todonotes}
\usepackage{mathrsfs}
\usepackage{mathabx}
\usepackage{nicefrac}
\usepackage{tikz}
\usetikzlibrary{matrix}

\usepackage{url}

\usepackage[sort,numbers]{natbib}

\providecommand{\noopsort}[1]{}

  \usepackage{srcltx}
  \usepackage{graphicx}

\input xy
\xyoption{all}

  \usepackage{color}
  \usepackage{ulem}

	\def\CC{{\ifmmode{\mathbbm{C}}\else{$\mathbbm{C}$}\fi}}
	\def\EE{{\ifmmode{\mathbbm{E}}\else{$\mathbbm{E}$}\fi}}
	\def\FF{{\ifmmode{\mathbbm{F}}\else{$\mathbbm{F}$}\fi}}
	\def\HH{{\ifmmode{\mathbbm{H}}\else{$\mathbbm{H}$}\fi}}
	\def\KK{{\ifmmode{\mathbbm{K}}\else{$\mathbbm{K}$}\fi}}
	\def\NN{{\ifmmode{\mathbbm{N}}\else{$\mathbbm{N}$}\fi}}
	\def\QQ{{\ifmmode{\mathbbm{Q}}\else{$\mathbbm{Q}$}\fi}}
	\def\RR{{\ifmmode{\mathbbm{R}}\else{$\mathbbm{R}$}\fi}}
	\def\TT{{\ifmmode{\mathbbm{T}}\else{$\mathbbm{T}$}\fi}}
	\def\UU{{\ifmmode{\mathbbm{U}}\else{$\mathbbm{U}$}\fi}}
	\def\ZZ{{\ifmmode{\mathbbm{Z}}\else{$\mathbbm{Z}$}\fi}}

	\newcommand{\tend}[2]{\xrightarrow[#1\to#2]{}}
	
	\newcommand{\Id}{{\rm Id}}

	\newcommand{\ind}[1]{\mathbbmss{1}_{#1}}

    \newcommand{\bfu}{\boldsymbol{u}}

\newcommand{\diam}{\mathop{\mathrm{diam}}}
\newcommand{\proj}{\mathrm{proj}}
\newcommand{\af}{\ca_{\cf}}

\def\qed{\unskip\quad \hbox{\vrule\vbox
to 6pt {\hrule width 4pt\vfill\hrule}\vrule} }

\newcommand{\bez}{\nopagebreak\hspace*{\fill}
 \nolinebreak$\qed$\vspace{5mm}\par}

\newtheorem{Th}{Theorem}[section]
\newtheorem{Prop}[Th]{Proposition}
\newtheorem{Lemma}[Th]{Lemma}

\newtheorem{Cor}[Th]{Corollary}
\theoremstyle{definition}
\newtheorem{Remark}[Th]{Remark}
\newtheorem{Def}{Definition}[section]
\newtheorem{Example}[Th]{Example}

\newcommand{\beq}{\begin{equation}}
\newcommand{\eeq}{\end{equation}}

\def\scalar(#1,#2){(#1\mid#2)}

\newcommand{\eps}{\varepsilon}

\newcommand{\raz}{\mathbbm{1}}

\newcommand{\ci}{\mathcal{I}}

\newcommand{\ca}{{\cal A}}
\newcommand{\cb}{{\cal B}}
\newcommand{\cc}{{\cal C}}

\newcommand{\cf}{{\cal F}}

\newcommand{\xbm}{(X,\mathcal{B}_X,\mu)}
\newcommand{\zdr}{(Z,\mathcal{B}_Z,\rho)}
\newcommand{\zdk}{(Z,\mathcal{B}_Z,\kappa)}
\newcommand{\ycn}{(Y,\mathcal{B}_Y,\nu)}
\newcommand{\ot}{\otimes}
\newcommand{\ov}{\overline}
\newcommand{\la}{\lambda}

\newcommand{\bs}{\mathbb{S}}

\newcommand{\Q}{\mathbb{Q}}
\newcommand{\R}{{\mathbb{R}}}
\newcommand{\T}{{\mathbb{T}}}
\newcommand{\C}{{\mathbb{C}}}
\newcommand{\Z}{{\mathbb{Z}}}
\newcommand{\N}{{\mathbb{N}}}
\newcommand{\D}{{\mathbb{D}}}
\newcommand{\PP}{{\mathbb P}}
\newcommand{\vep}{\varepsilon}

\newcommand{\mob}{\boldsymbol{\mu}}
\newcommand{\lio}{\boldsymbol{\lambda}}

\newcommand{\bfv}{\boldsymbol{v}}

\newcommand{\zdn}{(Z,\mathcal{D},\nu)}

\newcommand{\cfec}{\cf_{\rm ec}}

\renewcommand{\P}{\mathscr{P}}
\renewcommand{\Q}{\mathscr{Q}}

\newcommand{\Xb}{\bar{X}}
\newcommand{\xb}{\bar{x}}
\newcommand{\mub}{\bar{\mu}}

\begin{document}
\title{On orthogonality to uniquely ergodic systems}
\author{M.\ G\'orska, M.\ Lema\'nczyk, T.\ de la Rue}

\maketitle

\begin{abstract}
We solve Boshernitzan's problem of characterization (in terms of so called Furstenberg systems) of bounded sequences that are orthogonal to all uniquely ergodic systems. As a step toward this solution, we provide  a characterization of automorphisms which are disjoint from all ergodic ones as those whose a.a.\ ergodic components form a family of pairwise disjoint automorphisms. Some variations of Boshernitzan's problem involving characteristic classes are considered. As an application, we characterize sequences  orthogonal to all uniquely ergodic systems whose (unique) invariant measure yields a discrete spectrum automorphism as those satisfying an averaged Chowla property.
\end{abstract}

\tableofcontents

\section{Introduction}
\subsection{General background and main notations} Throughout $\N:=\{1,2,\ldots\}$ stands for the set of natural numbers

We recall that a \emph{standard Borel space} is a measurable space whose sigma-algebra consists of the Borel sets for some Polish metric.  If $X$ is such a space then its sigma-algebra is denoted by $\cb_X$.

If $X$ is a standard Borel space, then $M(X)$ is the space of all Borel probability measures on $X$, equipped with the sigma-algebra generated by all evaluation maps
\[\mu \mapsto \mu(B) \quad (B \in \mathcal{B}_X).\]
With this sigma-algebra, $M(X)$ is again Borel standard. If $Y$ is another measurable space, then a \emph{probability kernel} from $Y$ to $X$ is a measurable map from $Y$ to $M(X)$.

We consider measurable invertible transformations $T$ acting on a standard Borel space $(X,\cb_X)$. We denote by $\ci_T\subset\cb_X$ the sub-sigma-algebra of $T$-invariant Borel sets:
\[ \ci_T\ :=\ \{A\in\cb_X:T^{-1}A=A\}. \]
We denote by $M(X,T)\subset M(X)$ the subspace of $T$-invariant probability measures. A measure $\mu\in M(X,T)$ is said to be \emph{ergodic} (for $T$) if $\mu$ only gives values 0 or 1 to sets in $\ci_T$. Let $M^e(X,T)\subset M(X,T)$ stand for the subspace of ergodic $T$-invariant measures. Both $M(X,T)$ and $M^e(X,T)$ are Borel subsets of $M(X)$.
When $T$ preserves a probability measure $\mu$ on $(X,\cb_X)$, we call $T$ an {\em automorphism} of the standard Borel probability space $(X,\mathcal{B}_X,\mu)$, and $(X, \mathcal{B}_X,\mu,T)$ is called a \emph{measure-preserving system}. (We often omit $\cb_X$ and write only $(X,\mu,T)$.) For a given $\mu\in M(X)$, the group of automorphisms of $(X,\mathcal{B}_X,\mu)$ (in which two transformations are identified if they agree outside a $\mu$-negligible set) is denoted by ${\rm Aut}(X,\cb_X,\mu)$. It is a Polish group when considered with the smallest topology (called {\em strong}) making all the maps $T\mapsto f\circ T$, $f\in L^2(X,\mu)$, continuous.

By a {\em topological system} $(X,T)$, we mean the action of a homeomorphism $T$ on a compact metric space $X$.
Given a topological system $(X,T)$, $M(X)$ endowed with the weak$^\ast$-topology is compact, and both $M(X,T)$ and $M^e(X,T)$ are non-empty subspaces of $M(X)$ with $M(X,T)$ being closed and $M^e(X,T)$ being $G_\delta$. Then, each $\mu\in M(X,T)$ yields a measure-preserving system
$(X, \mathcal{B}_X,\mu,T)$.

If for $i=1,2$,  $(X_i,\mathcal{B}_{X_i},\mu_i,T_i)$ is a measure-preserving system, we can consider $J(T_1,T_2)$ the set of {\em joinings}, i.e.\ of all $T_1\times T_2$-invariant measures on $X_1\times X_2$ with marginals $\mu_1$ and $\mu_2$, respectively. Then, the two automorphisms are called (Furstenberg) {\em disjoint} (we write $T_1\perp T_2$) if $J(T_1,T_2)=\{\mu_1\ot\mu_2\}$. It is not hard to see that if $T_1\perp T_2$ then at least one of the two automorphisms has to be ergodic. For a single measure-preserving system $(X,\mu,T)$, we write $J_2(T)$ instead of $J(T,T)$.

\subsection{Presentation of the main problem and motivations}
In the present paper, we will study the following problem of {\em orthogonality}: Given a class $\mathscr{C}$ of topological systems and a (bounded) sequence $\bfu:\N\to\D:=\{z\in\C\colon |z|\leq1\}$ of zero mean, i.e.\ $M(\bfu):=\lim_{N\to\infty}\frac1N\sum_{n\leq N}\bfu(n)=0$, we say that $\bfu$ is orthogonal to $\mathscr{C}$, $\bfu\perp \mathscr{C}$, if
\beq\label{sarnak10}
\lim_{N\to\infty}\frac1N\sum_{n\leq N}f(T^nx)\bfu(n)=0\eeq
for each $(X,T)\in\mathscr{C}$, $f\in C(X)$ and all $x\in X$ (we write then $\bfu\perp(X,T)$). The main motivation to study this problem is the celebrated Sarnak's conjecture\footnote{Sarnak's conjecture is in turn motivated by the Chowla conjecture from 1965 which predicts that the autocorrelations of the Liouville function vanish, i.e.\ the Liouville function is a generic point (see Section~\ref{s:liftinglemma}) for the Bernoulli measure $B(1/2,1/2)$ for the full shift $\{-1,1\}^{\Z}$; the Chowla conjecture implies Sarnak's conjecture \cite{Sa}, \cite{Ta0}, \cite{Fe-Ku-Le}.}  \cite{Sa} in which $\bfu$ is the M\"obius (or Liouville) function and $\mathscr{C}=\mathscr{C}_{\rm ZE}$ the class of (topological) zero entropy systems.
We can also study the problem of orthogonality without the zero-mean assumption for $\bfu$. Namely, by $\bfu\perp(X,T)$, we mean uncorrelation:
\beq\label{f:defort}\frac1N\sum_{n\leq N}f(T^nx)\bfu(n)=\frac1N\sum_{n\leq N}f(T^nx)\cdot \frac1N\sum_{n\leq N}\bfu(n)+o(1)\eeq
(when $N\to\infty$) for all $f\in C(X)$ and $x\in X$.  Note that, if $(X,T)$ is uniquely ergodic, i.e.\ if it has a unique $T$-invariant probability measure (by unicity, such a measure has to be ergodic),  the above requirement is equivalent to $\lim_{N\to\infty}\frac1N\sum_{n\leq N}f(T^nx)\bfu(n)=0$ for all $f\in C(X)$ of zero mean (with respect to the unique invariant measure) and $x\in X$. Moreover, if $M(\bfu)=0$, we come back to~\eqref{sarnak10}, although we will see later that the natural (stronger than $M(\bfu)=0$) assumption on $\bfu$ is that the first Gowers-Host-Kra norm vanishes: $\|\bfu\|_{u^1}=0$.

As done by Tao \cite{Ta}, we can change the Ces\`aro averaging in \eqref{sarnak10}, by the logarithmic averages (the logarithmic mean of $\bfu$ is zero whenever its usual mean vanishes) and study the corresponding orthogonality problem:
\beq\label{sarnaklog}
\lim_{N\to\infty}\frac1{\log N}\sum_{n\leq N}\frac1nf(T^nx)\bfu(n)=0\eeq
which, in particular, leads to the logarithmic Sarnak's conjecture.
While even the logarithmic Sarnak's conjecture stays open, a remarkable theorem by Fran\-zi\-ki\-na\-kis-Host \cite{Fr-Ho} asserts  that the Liouville function $\lio$ ($\lio(n)=-1$ if the number of prime divisors (counted with multiplicity) is odd, and it is~1 otherwise)  and many other multiplicative functions are logarithmically orthogonal to all uniquely ergodic zero entropy systems. A natural question arises
to characterize those sequences which are orthogonal (or logarithmically orthogonal) to the class UE of all uniquely ergodic systems. In \cite{Co-Do-Se}, this particular problem is attributed to M.\ Boshernitzan, and from now on we will call it {\em Boshernitzan's problem}.
Conze, Downarowicz and Serafin showed in \cite{Co-Do-Se} that Boshernitzan's problem is non-trivial by noticing that sequences ``produced'' by systems which are disjoint in the Furstenberg sense from all ergodic automorphisms are orthogonal to the UE class. This path of reasoning, relating the fundamental problem of Furstenberg disjointness of measure-preserving systems in ergodic theory \cite{Fu} with the problem of orthogonality to sequences is not new and it was applied earlier, when studying Sarnak's conjecture, in particular, it was successfully applied in \cite{Fr-Ho} and many other papers (see the survey \cite{Fe-Ku-Le}). Let us quickly recall the essence of this approach.

Given $\bfu:\N\to\D$ ($\D$ stands for the unit disc; we freely treat $\bfu$ as a two-sided sequence, e.g.\ by setting $\bfu(-n)=\bfu(n)$, $\bfu(0)=1$), we consider $X_{\bfu}:=\ov{\{S^k\bfu\colon k\in\Z\}}$ the subshift generated by $\bfu$, where $S:\D^{\Z}\to\D^{\Z}$, $Sy(n)=y(n+1)$, for all $n\in\Z$, stands for the left shift. Let
$$
V(\bfu):=\Big\{\kappa\in M(X_{\bfu})\colon \Big(\exists N_k\uparrow\infty\Big)\;\kappa=\lim_{k\to \infty}\frac1{N_k}\sum_{n\leq N_k}\delta_{S^n\bfu}\Big\}\subset M(X_{\bfu},S)$$
denote the set of $\bfu$-{\em visible} measures. ($V^{\rm log}(\bfu)\subset M(X_{\bfu},S)$ is defined analogously; in general, the sets of visible and logarithmically visible measures can be disjoint, see, e.g., \cite{Go-Le-Ru}.)
By a {\em Furstenberg system} of $\bfu$, we mean the measure-preserving system $(X_{\bfu},\kappa,S)$, for each $\kappa\in V(\bfu)$.
Suppose now that we understand (in the sense of ergodic properties) all Furstenberg systems of $\bfu$. Then, if we take any topological system $(X,T)$ for which for each $\nu\in M(X,T)$, the measure-preserving systems $(X,\nu,T)$ is Furstenberg disjoint to all $(X_{\bfu},\kappa,S)$ for all $\kappa\in V(\bfu)$, then it is not hard to see that
$$
\bfu\perp (X,T).$$
In fact, to achieve this goal, as proved in \cite{Ka-Ku-Le-Ru}, we need a condition appealing {\bf only} to Furstenberg systems of $\bfu$, namely, we need to see that the function $\pi_0:X_{\bfu}\to \D$, $\pi_0((x_k)):=x_0$ is orthogonal to the $L^2$-space of a certain factor $\ca(\kappa)$ of $(X_{\bfu},\kappa,S)$ -- the factor $\ca(\kappa)$, called {\em characteristic}, will depend on the class $\mathscr{C}$ to which $(X,T)$ belongs:
\beq\label{veechin}
\pi_0\perp L^2(\ca(\kappa))\text{ for all }\kappa\in V(\bfu).\eeq
This is a particular instance of the Veech condition (for the definition of the Veech condition in the general characteristic class setup, see Section~\ref{s:aoma}). For Sarnak's conjecture (when $\mathscr{C}=\mathscr{C}_{\rm ZE})$ the relevant conjecture\footnote{Veech's conjecture was proved recently in \cite{Ka-Ku-Le-Ru}.} was formulated by Veech \cite{Ve} and stated that,  for each Furstenberg system $\kappa$ of $\lio$, $\pi_0$ is orthogonal to the $L^2$-space of the Pinsker factor $\Pi(\kappa)$ (i.e.\ the largest factor with  zero Kolmogorov-Sinai entropy).

All the above can be repeated in the logarithmic case, and in \cite{Fr} it is in fact proved that either a Furstenberg system of $\lio$ is ergodic (then the Chowla conjecture holds), or it must be disjoint from all ergodic systems with zero entropy. The latter naturally raises a question (of independent interest in ergodic theory) how to describe the class Erg$^\perp$ of automorphisms disjoint from all ergodic automorphisms.\footnote{The reader can easily check that although the disjointness theory in ergodic theory is well developed, see e.g. the monograph \cite{Gl}, but it usually considers the disjointness question between two ergodic automorphisms. Indeed, in this case, the problem is reduced to show that each ergodic joining is product measure.}
We now pass to a description of the results of the paper.

\subsection{Ergodic theory results}\label{s:etr}

Let $T$ be an automorphism of $(X,\cb_X,\mu)$. We are interested in the disjointness condition of $T$ with the class of all ergodic automorphisms (which we denote by $T\in {\rm Erg}^\perp$ or $T\perp {\rm Erg}$).
 Therefore, we focus on the case where the system is \emph{non-ergodic}, and consider its ergodic decomposition:
 \beq\label{ed10}
  \mu\ =\ \int_X \mu_x d\mu(x),
 \eeq
that is, the disintegration of $\mu$ over the sub-sigma-algebra $\ci_T$ of $T$-invariant subsets. Here, $x\mapsto\mu_x$ is a measurable map from $X$ to $M^e(X,T)$, and the sigma-algebra it generates coincides modulo $\mu$ with $\ci_T$.
Then (up to a $\mu$-negligible set), we can write $X$ as a disjoint union
\[
 X\ =\ \bigsqcup_{n\in\NN\cup\{\infty\}}X_n,
\]
where, for $n\in\NN$, $X_n$ is the subset of $x\in X$ such that $\mu_x$ is concentrated on $n$ points with equal mass $1/n$ (and those $n$ points are cyclicly permutted by $T$), and $X_\infty$ stands for the subset of $x\in X$ such that $\mu_x$ is non-atomic. We set
\[
 \Xb\ :=\ \{\mu_x:x\in X\}\subset M^e(X,T),
\]
equipped with the pushforward image $\mub$ of $\mu$. We can view $\Xb$ as an abstract space which is introduced to embody the ``identity'' part of $T$: $(X,\mu,T)$ appears as a relatively ergodic extension of the identity system $(\Xb,\mub,\Id_{\Xb})$. We also set, for $n\in\NN\cup\{\infty\}$,
\[
 \Xb_n\ :=\ \{\mu_x:x\in X_n\}\subset \Xb,
\]
so that
\[
 \Xb\ =\ \bigsqcup_{n\in\NN\cup\{\infty\}}\Xb_n.
\]
For $n\in\NN$, we introduce the finite ergodic system $\bigl(\{1,\ldots,n\},\nu_n,R_n\big)$, where $\nu_n$ is the uniform probability distribution over $\{1,\ldots,n\}$ and $R_nj:=j+1\mod n$. We also introduce an abstract standard Borel probability space $(Y,\cb_Y,\nu)$, where $\nu$ is non-atomic.
Then the system $(X,\mu,T)$ can be represented, up to a measure-theoretic isomorphism, as follows:
\begin{itemize}
 \item For $n\in\NN$, $X_n=\Xb_n\times \{1,\ldots,n\}$, and for $x=(\xb,j)\in X_n$, $Tx = (x,R_nj)$. Moreover,
 \[
    \mu|_{X_n}\ =\ \mub|_{\Xb_n}\otimes\nu_n.
 \]
 \item $X_\infty = \Xb_\infty\times Y$, and for $x=(\xb,y)\in X_\infty$, $Tx = (\xb,T_{\xb}y)$, where the automorphisms $T_{\xb}\in {\rm Aut}(Y,\nu)$, $\xb\in\Xb$, are ergodic. Here, the map
\beq\label{ergdec18}\Xb_\infty\ni\xb\mapsto T_{\xb}\in {\rm Aut}(Y,\nu)\text{ is Borel}.\eeq
Moreover,
 \[
    \mu|_{X_\infty}\ =\ \mub|_{\Xb_\infty}\otimes\nu.
 \]
\end{itemize}
To standardize our notations, we will also set $T_{\xb}:=R_n$ for $\xb\in\Xb_n$. This allows us to rewrite the ergodic decomposition~\eqref{ed10} as
\beq\label{ed11}
\mu=\int_{\Xb}\mu_{\xb}\,d\mub(\xb),\eeq
where the conditional measures $\mub_{\xb}$ are now either $\delta_{\xb}\ot\nu_n$ if $\xb\in\Xb_n$ or $\delta_{\xb}\ot\nu$ whenever $\xb\in\Xb_\infty$.

\vspace{2ex}

\noindent
{\bf Theorem~A.} {\em
The following conditions are equivalent:\\
(i) $T\in{\rm Erg}^\perp$.\\
(ii) The extension $T\to \Id_{\Xb}$ is confined (i.e.\ the only member $\lambda\in J_2(T)$ such that $\lambda|_{\Xb\times \Xb}=\mub\ot\mub$ is equal to $\mu\ot\mu$, see Subsection~\ref{ss:conf}).\\
(iii) $T_{\xb}\perp T_{\xb'}$ for $\mub\ot\mub$-a.a.\ $(\xb,\xb')\in \Xb\times \Xb$.}\footnote{The implication (iii) $\Rightarrow$ (i), using the approach from \cite{Ka-Ku-Le-Ru}, can be thought of to be more or less implicit in \cite{Co-Do-Se}.}

(For the proof of Theorem~A, see Theorem~\ref{t:duzoroz} and its proof.)

\vspace{2ex}

We refer to Section~\ref{sec:cc} for the definition of a characteristic class $\cf$ of measure-preserving systems, and the associated largest $\cf$-factor $\af(X,\mu,T)$ in any system $(X,\mu,T)$. Recall that the class ZE of zero-entropy systems is the largest proper characteristic class \cite{Ka-Ku-Le-Ru}, and that the largest ZE-factor $\ca_{\rm ZE}(X,\mu,T)$ is the Pinsker factor $\Pi(T)$.

\vspace{2ex}

\noindent
{\bf Theorem~B.} {\em \\ (i)  Let $\cf$ be a characteristic class. Then $T\perp \cf\cap{\rm Erg}$
if and only if the extension $T|_{\af}\to \Id_{\Xb}$ is confined ($\af $ is the largest $\cf$-factor of $T$).\\
(ii) $T\perp {\rm ZE}\cap {\rm Erg}$ if and only if
$\Pi(T_{\xb})\perp \Pi(T_{\xb'})$ for $\mub\ot\mub$-a.a.\ $(\xb,\xb')\in \Xb\times \Xb$.
}

(For the proof of Theorem~B, see Corollary~\ref{c:charconf}, Corollary~\ref{c:duzorozPinsker} and their proofs.)
\vspace{2ex}

Let $J_2^{\rm RelErg}(T)$ stand for the subset of those members $\la$ of $J_2(T)$ whose marginal on $\Xb\times\Xb$ is $\mub\otimes\mub$, and which is relatively ergodic over the factor $(\Xb\times \Xb,\mub\ot\mub,\Id_{\Xb\times\Xb})$, i.e. $\ci_{T\times T} = \ci_{T}\otimes\ci_{T}$ mod $\lambda$.
The next theorem refers to the weakly ergodic subspace $F_{\rm we}(T)\subset L^2(\mu)$, whose precise definition is given in  Subsection~\ref{ss:wep}. For $f\in L^2(\mu)$, being orthogonal to $F_{\rm we}(T)$ means the absence of correlation of $f$ with any function $g$ defined on $Z$ in any joining of $T$ with an ergodic system $(Z,\rho,R)$. We recall that $\ci_{T\times T}$ stands for the sub-sigma-algebra of $T\times T$-invariant Borel subsets of $X\times X$.

\vspace{2ex}

\noindent
{\bf Theorem~C.} {\em Assume that $f\in L^2(X,\mu)$, and $\EE_{\mu}[f\,|\,\ci_T]=0$. The following conditions are equivalent:\\
(i) $f\perp F_{\rm we}(T)$.\\
(ii) $\EE_{\la}[f\ot\ov{f}\,|\,\Xb\times \Xb]=0$ mod $\mu\ot\mu$ for all $\la\in J_2^{\rm RelErg}(T)$. \\
(iii) $\EE_{\la}[f\ot\ov{f}\,|\,\ci_{T\times T}]=0$ mod $\la$ for all $\la\in J_2(T)$, $\la|_{\Xb\times \Xb}=\mub\ot\mub$.}

(For the proof of Theorem~C, see Theorem~\ref{t:OrMaIm} and Corollary~\ref{c:relvN2} and their proofs.)

\subsection{Solution of Boshernitzan's problem}

Let $\bfu:\N\to\D$.

\vspace{1ex}

\noindent
{\bf Proposition~D.} {\em
$\bfu\perp\,{\rm UE}$ if and only if for each Furstenberg system $\kappa\in V(\bfu)$, $\pi_0\perp F_{{\rm we},0}(X_{\bfu},\kappa,S)$.
}

(For the proof, see Proposition~\ref{p:oues} and its proof.)

\vspace{2ex}

In Proposition~\ref{prop:selfjoiningsofFS}, we describe combinatorially all self-joinings of Furstenberg systems of $\bfu$. Hence, Proposition~D, Theorem~C together with the von Neumann theorem describe the full solution of Boshernitzan's problem in terms of Furstenberg systems of $\bfu$, see Remark~\ref{r:Bosher}.   Proposition~D has its characteristic class version counterpart ($\af=\af(\kappa)$ below stands for the largest characteristic $\cf$-factor of a Furstenberg system $(X_{\bfu},\kappa,S)$ of $\bfu$):

\vspace{2ex}

\noindent
{\bf Corollary~E.} {\em For any characteristic class $\cf$, the following conditions are equivalent:\\
(a) $\bfu \perp \cf\cap {\rm UE}$.\\
(b) For each Furstenberg system $\kappa$ of $\bfu$, for each $\la\in J^{\rm RelErg}_2(S|_{\af},\kappa|_{\af})$ such that $\la|_{\ci_S\ot\ci_S}=\kappa|_{\ci_S}\ot\kappa|_{\ci_S}$, we have
$$
\EE_\lambda\Bigl[\EE_\kappa[\pi_0\,|\,\af]\ot \EE_\kappa[\pi_0\,|\,\af]\,|\,\ci_S\ot\ci_S\Bigr]=
\EE_\kappa[\pi_0\,|\,\ci_S]\ot
\EE_\kappa[\pi_0\,|\,\ci_S].
$$
(c) For each Furstenberg system $\kappa$ of $\bfu$, for each $\la\in J_2(X_{\bfu},\kappa,S)$ such that $\la|_{\ci_S\ot\ci_S}=\kappa|_{\ci_S}\ot\kappa|_{\ci_S}$, we have
$$
\EE_\lambda\Bigl[\EE_\kappa[\pi_0\,|\,\af]\ot \EE_\kappa[\pi_0\,|\,\af]\,|\,\ci_{S\times S}\Bigr]=
\EE_\kappa[\pi_0\,|\,\ci_S]\ot
\EE_\kappa[\pi_0\,|\,\ci_S].
$$
}

(For the proof, see Corollary~\ref{c:enfin} and its proof.)

\vspace{2ex}

Given a characteristic class $\cf$, we denote by $\mathscr{C}_{\cf}$ the class of topological systems for which all visible invariant measures yield automorphisms in $\cf$. (An invariant measure $\mu$ in $(X,T)$ is {\em visible} if there exists $x\in X$ which is quasi-generic for $\mu$, i.e.\ $\mu\in V(x)$.) The main result of \cite{Ka-Ku-Le-Ru} stated that the Veech condition (see Section~\ref{s:aoma}) is {\bf equivalent} to the Sarnak condition: $\bfu\perp \mathscr{C}_{\cf}$, {\bf provided that $\cf$ is a so called ec-class} (see Section~\ref{sec:cc}). It was left as an open problem whether the equivalence of these two conditions holds for any characteristic class. It turns out that this is not the case as the following result shows:

\vspace{2ex}

\noindent
{\bf Proposition~F.} {\em There exist a characteristic class $\cf$ and $\bfu:\N\to \{-1,1\}$ which is generic for a measure $\kappa$ on $\{-1,1\}^{\Z}$ such that $\bfu\perp \mathscr{C}_{\cf}$ but the Veech condition fails.}

(For the proof, see Section~\ref{s:aoma}.)

\subsection{Applications}
We refer to Subsection~\ref{ss:ghk} and Subsection~\ref{s:acp} for the definitions of Gowers-Host-Kra (GHK) norms $\|\cdot\|_{u^s}$ and of  an averaged Chowla property\footnote{An averaged Chowla property (even in a quantitative version) has been proved for the Liouville function (and many other multiplicative functions) by Matom\"aki-Radziwi\l\l-Tao \cite{Ma-Ra-Ta}.}. Denote by DISP the (characteristic) class of automorphisms having discrete spectrum.
Our main application is the following result:

\vspace{2ex}

\noindent
{\bf Corollary~G.}  {\rm Let $\bfu:\N\to\D$, $\|\bfu\|_{u^1}=0$.\footnote{This condition is equivalent to so called ``zero mean property on typical short interval'', cf.\ Proposition~\ref{p:ju1}.} The following conditions are equivalent:\\
(i) $\bfu$ satisfies an averaged Chowla property.\\
(ii) $\bfu\perp {\rm UE}\cap{\rm DISP}$.\\
(iii) $\bfu\perp \mathscr{C}_{\rm DISP}$.}

(For the proof, see Theorem~\ref{t:dz2} and its proof.)

\vspace{2ex}
The equivalence between (ii) and (iii) seems to be exceptional for the class DISP, we do not expect it to hold for a general characteristic class. It obviously does not hold for the class of all automorphisms but more interestingly, it does not hold for the class ZE. Indeed, for all sequences $\bfu$ of the form~\eqref{waznyciag}, we have $\|\bfu\|_{u^1}=0$ by Remark~\ref{r:ujeden},
and they satisfy $\bfu\perp {\rm UE}\cap{\rm ZE}$. However, by Remark~\ref{r:ujedenA}, they  do not satisfy (iii), as their own topological entropy is zero. Denote by CE the class of topological systems whose set of ergodic invariant measures is countable. In light of what was said, the following result looks somewhat surprising.

\vspace{2ex}

\noindent
{\bf Corollary H.}  {\rm  Let $\cf$ be an arbitrary characteristic class.   Let $\bfu:\N\to\D$, $\|\bfu\|_{u^1}=0$. Then $\bfu\perp {\rm UE}\cap\cf$ if and only if $\bfu\perp{\rm CE}\cap \cf$.

}

(For the proof, see Proposition~\ref{p:thi3} and its proof. The class ${\rm CE}\cap\cf$ is defined unambiguously due to Proposition~\ref{p:countableCC}.)

\vspace{2ex}

It follows that in the theorem of Frantzikinakis-Host \cite{Fr-Ho}, the final statement of orthogonality (in the logarithmic sense) of the Liouville function to all zero entropy systems from CE, is in fact equivalent to the orthogonality to all zero entropy uniquely ergodic systems.
Recall that $\bfu:\N\to\D$ is called {\em mean slowly varying} if \beq\label{msvf}\lim_{N\to\infty}\frac1N\sum_{n\leq N}|\bfu(n)-\bfu(n+1)|=0.\eeq

\vspace{2ex}

\noindent
{\bf Corollary I.}  {\rm  The mean slowly varying functions are multipliers of UE$^\perp$. That is, if $\bfu:\N\to\D$ is mean slowly varying and $\bfv\perp {\rm UE}$ then $\bfu\cdot\bfv\perp {\rm UE}$, where $(\bfu\cdot\bfv)(n)=\bfu(n)\cdot\bfv(n)$ for $n\in\N$.

(For the proof, see Proposition~\ref{p:t+m} and its proof.)

\vspace{2ex}

Here, it seems that 
the class of mean slowly varying function should precisely be the class of multipliers of UE$^\perp$.

\vspace{2ex}

The next fact is for multiplicative functions (see Section~\ref{s:pretensious} for some basic facts on such functions).

\vspace{2ex}

\noindent
{\bf Corollary~J.}  {\rm The only pretentious multiplicative functions $\bfu:\N\to\D$ orthogonal to UE are Archimedean characters $n\mapsto n^{it}$.}

(For the proof, see  Corollary~\ref{c:PretMult}.)

\vspace{2ex}

It is reasonable to expect that the result holds in the class of all multiplicative functions.

\subsection{Auxiliary results} In the last three sections we prove some new ergodic results which are used to obtain the aforementioned theorems, but which are also of an independent interest.  Namely,  in Section~\ref{s:osiem}, we provide a combinatorial description of all self-joinings of all Furstenberg systems of a bounded arithmetic function $\bfu$. In Section~\ref{s:RelErg}, the existence of a so-called \emph{relative ergodic decomposition}, with respect to a sub-sigma-algebra of $\ci_T$ is established (the proof has been written by Tim Austin). Finally, in Section~\ref{s:PpEC}, we show that the ``trace'' of the Pinsker factor of an automorphism on almost all ergodic component is equal to the Pinsker factor of the  ergodic component. This result seems to be a part of folklore, but we could not find its proof in the literature.

\section{Preparatory material}
\subsection{Lifting lemma} \label{s:liftinglemma}
Given a topological system $(Y,S)$ and $\kappa\in M(Y)$, we say that a point $u\in Y$ is {\em quasi-generic along a sequence} $(N_k)$ for $\kappa$ if
 $\frac1{N_m}\sum_{n\leq N_m}g(S^nu)\to \int_Yg\,d\nu$ for all $g\in C(Y)$.
 If the convergence takes place for the whole sequence of natural numbers, then $u$ is called {\em generic} (for $\kappa$).
We will need the following lemma about lifting of quasi-generic point (along a subsequence) to quasi-generic sequences:

\begin{Th}[\cite{Ka-Ku-Le-Ru}]\label{t:lifting}
 \label{l:kklr} Assume that $(Y,S)$ and $(X,T)$ are topological systems. Let $\nu\in M(X,T)$, $u\in Y$ be quasi-generic along an increasing sequence $(N_m)$ for $\kappa\in M(Y,S)$ and $\rho\in J(\kappa,\nu)$. Then there exist a sequence $(x_n)\subset X$ and a subsequence $(N_{m_\ell})$ such that $(S^nu,x_n)$ is quasi-generic along $(N_{m_\ell})$ for $\rho$, i.e.\ $\lim_{\ell\to\infty}\frac1{N_{m_\ell}}\sum_{n\leq N_{m_\ell}}G(S^nu,x_n)=\int_{Y\times X}G\,d\rho$ for each $G\in C(Y\times X)$, and the set $\{n\geq1:\: x_{n}\neq Tx_{n-1}\}$ is of the form $\{b_1<b_2<\cdots\}$ with $b_{k+1}-b_k\to \infty$ when $k\to\infty$.\end{Th}

\begin{Remark} As proved in \cite{Ka-Ku-Le-Ru}, this result is also true for the logarithmic averages.\end{Remark}

\begin{Remark} First lifting lemmas of that kind can be found in  \cite{Ka}, \cite{Co-Do-Se}.\end{Remark}

\subsection{Orbital uniquely ergodic models}
We will also need the following:
\begin{Th}[\cite{Ab-Le-Ru}]\label{t:orbmod} Assume that $(X,T)$ is uniquely ergodic and let $(x_n)\subset X$ satisfy $\{n\geq0:\: x_{n+1}\neq Tx_n\}$ is of the form $(b_k)$ with $b_{k+1}-b_k\to \infty$ when $k\to\infty$. Consider $\underline{x}:=(x_n)\in X^{\N}$, the latter space with the one-sided shift $S$. Set $Y:=\overline{\{S^n\underline{x}:\:n\geq 0\}}\subset X^{\N}$. Then the subshift $(Y,S)$ is uniquely ergodic and (measure-theoretically) isomorphic to the original system.\end{Th}

\subsection{Joinings and Markov operators} Given  measure-preserving systems $(Z,\cb_{Z},\nu,R)$,  $(Z',\cb_{Z'},\nu',R')$ and a joining $\rho\in J(R',R)$, we let $\Phi_\rho$ denote the corresponding operator $\Phi_\rho:L^2(Z',\nu')\to L^2(Z,\nu)$ defined by
\beq\label{marko1} \int_Z \Phi_\rho(f')\ov{f}\,d\nu= \int_{Z'\times Z}f'(z')\ov{f(z)}\,d\rho(z',z)\eeq
for all $f'\in L^2(Z',\nu')$ and $f\in L^2(Z,\nu)$.
Then $\Phi_\rho$ is Markov: $\Phi_\rho\raz=\Phi_\rho^\ast\raz=\raz$ and $\Phi_\rho(g')\geq0$ whenever $g'\geq0$. Moreover, $\Phi_\rho$ is intertwining $R$ and $R'$: $\Phi_\rho\circ R'=R\circ \Phi_\rho$. And vice versa: if $\Phi:L^2(Z',\nu')\to L^2(Z,\nu)$ is a Markov operator intertwining $R$ and $R'$, then $\Phi=\Phi_\rho$ for a unique joining $\rho\in J(R',R)$, see e.g.\ \cite{Gl}, Chapter~6.

\subsection{Weakly ergodic part of a measure-preserving system}\label{ss:wep} Given a measure-preserving system $(X,\cb_{X},\mu,T)$ we call $f\in L^2(X,\cb_{X},\mu)$ {\em weakly ergodic} if there exists an ergodic $(X',\cb_{X'},\mu',T')$ and a joining $\rho\in J(T',T)$ such that $f\in {\rm Im}(\Phi_\rho)$, where $\Phi_\rho:L^2(X',\mu')\to L^2(X,\mu)$ is the Markov operator corresponding to $\rho$.  The closed subspace $F_{\rm we}(T)$ spanned by the weakly ergodic elements is called the {\em weakly ergodic} part of $L^2(X,\mu)$ for $T$. Note that
\beq\label{calkazero}
\mbox{If $f\perp F_{\rm we}(T)$ then  $f\in L^2_0(X,\mu)$, i.e.  $\int_Xf\,d\mu=0$.}\eeq
Moreover,
\beq\label{doda28}
\mbox{$f\perp F_{\rm we}(T)$ iff $\Phi_{\rho'}(f)=0$}\eeq
for all ergodic $T'$ and $\rho'\in J(T,T')$.

In what follows when considering the problem of orthogonality, unless stated otherwise, we will only consider the zero mean functions.
In order to treat the non-zero mean case, we consider the space $F_{{\rm we},0}(T)$ generated
by ${\rm Im}(\Phi_\rho|_{L^2_0(X',\mu')})$ for all ergodic $T'$.

\begin{Remark} If a system $(X,\cb_{X},\mu,T)$ is non-ergodic then $F_{\rm we}(T)$ is never dense in it. Indeed, take any non-trivial factor $\ca\subset\cb_{X}$ of $T$ which belongs to
${\rm Erg}^\perp$ (we can take for $\ca$ the sigma-algebra of invariant sets). If $g\in L^2_0(\ca)$ then $\int \Phi_\rho(f')\ov{g}\,d\mu=\int f'\ot\ov{g}\,d\rho=0$, the latter because $T'\perp T|_{\ca}$.\end{Remark}

\subsection{Koopman operator and spectral measures}
Given a measure-preserving system $(X,\cb_{X},\mu,T)$, recall that we have the corresponding Koopman operator (denoted also by $T$) acting on $L^2\zdn$ by
$$ f\mapsto Tf:=f\circ T.$$
This operator is unitary. Given $f\in L^2(\mu)$, we denote by $\sigma_f$  (or $\sigma_{f,T}$ if an ambiguity can arise) the {\em spectral measure} of $f$ which is a finite positive (Borel) measure on $\bs^1$ determined by
$$
\widehat{\sigma}(n):=\int_{\bs^1}z^n\,d\sigma_f(z)=\int_{X}T^n(f)\cdot \ov{f}\,d\mu\text{ for all }n\in\Z.$$
Then, $f$ is an eigenfunction corresponding to an eigenvalue $e^{2\pi i\alpha}$ of $T$ if and only if its spectral measure is atomic concentrated at $e^{2\pi i\alpha}$. $T$ is said to have discrete spectrum if the space $L^2\zdn$ is spanned by the eigenfunctions of $T$. If $T$ is ergodic then its eigenvalues form a (countable) subgroup of $\bs^1$.

\subsection{Furstenberg systems of arithmetic functions}
Given $\bfu:\N\to\D$, by $V(\bfu)$, we denote the set of shift-invariant measures on $(X_{\bfu},\cb(X_{\bfu}))$ for which $\bfu\in X_{\bfu}$ is quasi-generic. Recall that $(X_{\bfu},\cb(X_{\bfu}),\kappa,S)$ is a {\em Furstenberg system} of $\bfu$.
Let $\pi_0(y)=y_0$ for $y\in X_{\bfu}$; then $\pi_n=\pi_0\circ S^n$, $n\in\Z$, is a stationary process (once, we fix  $\kappa\in V(\bfu)$).

\subsection{Confined extensions}\label{ss:conf} Confined extensions (see the definition below) are systematically studied in \cite{Be} (we recall some known results for completeness). In the present paper, we consider them only in the context of extensions of identity. Let $R\in{\rm Aut}\zdk$ and $T\in{\rm Aut}\xbm$.

\begin{Lemma}\label{l:self}If $\Phi_\rho:L^2(Z,\kappa)\to L^2(X,\mu)$ and $\pi_0\in L^2(X,\mu)$ then
$$
\pi_0\perp {\rm Im}\,\Phi_\rho\text{ if and only if } \pi_0\perp {\rm Im} (\Phi_\rho\circ \Phi^\ast_\rho).$$
\end{Lemma}
\begin{proof}
Assume that $ \pi_0\perp {\rm Im} (\Phi_\rho\circ \Phi^\ast_\rho)$. Then for each $g\in L^2(Z,\kappa)$, we have $g=g_1+g_2$, where
$g_1\in\ov{{\rm Im}\,\Phi^\ast_\rho}$ and $g_2\perp {\rm Im}\,\Phi^\ast_\rho$. Hence
$$
\int \pi_0\Phi_\rho(g)\,d\mu=\int \pi_0\Phi_\rho(g_1)\,d\mu+\int \pi_0\Phi_\rho(g_2)\,d\mu=$$$$
\int \pi_0\Phi_\rho(g_1)\,d\mu+\int \Phi^\ast_\rho(\pi_0)g_2\,d\kappa=
\int \pi_0\Phi_\rho(g_1)\,d\mu$$
and since $g_1$ can be approximated by $\Phi_\rho^\ast(h)$ with $h\in L^2(X,\mu)$, the claim follows.\end{proof}

Below, we restrict to $L^2_0$-spaces. By Lemma~\ref{l:self}, we obtain the following:

\begin{Lemma} [del Junco-Rudolph, \cite{Ju-Ru}, proof of Prop.\ 5.3] \label{l:dJR} Under the above assumption, $\Phi_\rho=0$ if and only if $\Phi_\rho\circ \Phi^\ast_\rho=0$.\end{Lemma}

Assume now that we have an extension $\widetilde{R}\in{\rm Aut}(\widetilde{Z},\widetilde{\mathcal{D}},\widetilde{\kappa})$ of $R$. If $\widetilde{\la}$ is any self-joining of $\widetilde{\kappa}$ and $\lambda$ is its restriction to $Z\times Z$ then we easily check that
\beq\label{obc}
\Phi_\lambda= \proj_{L^2(Z)}\circ \Phi_{\widetilde{\lambda}}|_{L^2(Z)}.\eeq

\begin{Lemma}\label{l:conf1} Assume that $\widetilde\rho$ is a joining of $\widetilde{R}$ and $T$ (and $\rho$ stands for its restriction to $Z\times X$). Then $$\Phi_\rho^\ast\circ \Phi_\rho=\proj_{L^2(Z)}\circ \Phi^\ast_{\widetilde{\rho}}\circ \Phi_{\widetilde{\rho}}|_{L^2(Z)}.$$\end{Lemma}
\begin{proof} Similarly as in~\eqref{obc}, we have $\Phi_\rho=\Phi_{\widetilde\rho}|_{L^2(Z)}$. Then, for $f,g\in L^2(Z)$,
$$
\langle\Phi^\ast_\rho\circ \Phi_\rho f,g \rangle=
\langle\Phi_\rho f,\Phi_\rho g\rangle=$$$$
\langle \Phi_{\widetilde\rho} f, \Phi_{\widetilde\rho} g\rangle=
\langle \Phi^\ast_{\widetilde\rho}\circ \Phi_{\widetilde\rho}f,g\rangle=
\langle \proj_{L^2(Z)}\circ\Phi^\ast_{\widetilde\rho}\circ \Phi_{\widetilde\rho}f,g \rangle.$$
\end{proof}

Following \cite{Be}, an extension  $\widetilde{Z}\to Z$ is called {\em confined} if for each self-joining $\widetilde\lambda\in J_2(\widetilde Z)$, we have
$$
\widetilde \lambda|_{Z\times Z}=\kappa\ot\kappa\;\Rightarrow \widetilde\lambda=\widetilde\kappa\ot\widetilde\kappa.$$

We have now the following lifting disjointness result (first proved in \cite{Be}).

\begin{Prop} \label{p:conf1}
Assume that $\widetilde{Z}\to Z$ is a confined extension and let $Z\perp X$. Then $\widetilde Z\perp X$.\end{Prop}
\begin{proof} Define $\widetilde{\la}\in J_2(\widetilde{R})$ so that $\Phi_{\widetilde{\la}}=\Phi^\ast_{\widetilde{\rho}}\circ \Phi_{\widetilde{\rho}}$. Then by this and \eqref{obc},
$$
\Phi_{\la}=\proj_{L^2(Z)}\circ \Phi^\ast_{\widetilde{\rho}}\circ \Phi_{\widetilde{\rho}}|_{L^2(Z)}.$$
But now, by Lemma~\ref{l:conf1}, we obtain that $\Phi_{\la}=\Phi_\rho^\ast\circ\Phi_\rho$. Since $\Phi_\rho=0$, also $\Phi_{\la}=0$ and therefore $\Phi_{\widetilde{\la}}=0$ by our assumption. The result follows from Lemma~\ref{l:dJR}.
\end{proof}

It follows easily that each confined extension of an identity is disjoint from all ergodic automorphisms. In fact, in Theorem~\ref{t:duzoroz}, we will show that an automorphism $R$ is disjoint from all ergodic automorphisms if and only if it is a confined extension of an identity.

\begin{Lemma}\label{l:ilej}
Assume that $S\in{\rm Aut}\ycn$. Then the set
$$
S^{\not\perp}:=\{R\in {\rm Aut}\zdk:\: R\not\perp S\}$$
cannot contain an uncountable  family of pairwise disjoint automorphisms.\end{Lemma}
\begin{proof} Assume that we have $S^{\not\perp}\ni R_i\in {\rm Aut}\zdk$, $i\in I$ (uncountable) are pairwise disjoint. Now, suppose that $\Phi_i:L^2\zdk\to L^2\ycn$ is a non-trivial intertwining Markov operator. Take $f,g\in L^2_0\zdk$, then, for each $i\neq j$, we have
$$
\langle \Phi_if,\Phi_jg\rangle=\langle \Phi_j^\ast\Phi_if,g\rangle=0$$
as $\Phi^\ast_j\Phi_i=0$ is an intertwining operator of $R_i$ and $R_j$, so on $L_0^2$ it is zero. So the images via $\Phi_i$'s are pairwise orthogonal, and we obtain  a contradiction by separability.
\end{proof}

\subsection{Joinings of a non-ergodic automorphism with an ergodic system}\label{s:jerg}
Assume that $T\in{\rm Aut}(X,\mathcal{B}_X,\mu)$.
In the notation introduced at the beginning of Section~\ref{s:etr}, ${T}:(\xb,u)\mapsto\bigl(\xb,T_{\xb}(u)\bigr)$ acts on the space $X=\bigsqcup_{n\geq1}\Xb_n\times\{1,\ldots,n\}\sqcup \Xb_\infty\times Y$, where $\mu|_{\Xb_n\times\{1,\ldots,n\}}=\mub|_{\Xb_n}\ot\nu_n$ and
$\mu|_{\Xb_\infty\times Y}=\mub|_{\Xb_\infty}\ot\nu$.
Let us fix $R$ an {\bf ergodic} automorphism of $\zdk$.
Let $\rho\in J({T},R)$. Then we have factors:
$$
(X\times Z,\rho)\to (X,\mu)\to ({\Xb},\mub),$$
so let us disintegrate $\rho$ over $({\Xb},\mub)$:
\beq\label{rozk2}
\rho=\int_{{\Xb}}\rho_{{\xb}}\,d\mub({\xb}).\eeq
Since the action of $T$ on $\Xb$ is the identity, note that
\beq\label{rozk3}
\mbox{the measures $\rho_{{\xb}}$ are ${T}\times R$-invariant.}\eeq
Moreover, for $A\in \mathcal{B}_X$ and $B\in \mathcal{B}_Z$, we have
$$
\rho(A\times B)=\int_{\Xb} \rho_{{\xb}}(A\times B)\,d\mub({\xb}),$$
so
$$
\mu(A)=\rho(A\times Z)=\int_{\Xb} \rho_{{\xb}}(A\times Z)\,d\mub({\xb})$$
and
$$
\kappa(B)=\rho(X\times B)=\int_{\Xb} \rho_{{\xb}}(X\times B)\,d\mub({\xb}),$$
so by that (and in view of~\eqref{rozk3}):
\beq\label{rozk4}
\rho_{{\xb}}\in J((T_{\xb},\nu),(R,\kappa))\eeq
(for a.a.\ $\xb\in\Xb_\infty$; replace $\nu$ by $\nu_n$ when we consider $\xb\in\Xb_n$, $n\in\N$) by the uniqueness of disintegration of $\mu$ (over $\mub$) and the ergodicity of $\kappa$.

And vice versa: if \eqref{rozk4} holds {\bf and} the map
${\xb}\mapsto \rho_{{\xb}}$ is measurable then the formula \eqref{rozk2} defines a $\rho\in J({T},R)$. This observation is quite meaningful and, given $\Phi_{\rho}$, allows us to produce more Markov operators  $\Phi_{\rho_{D}}$ with $\rho_D\in J({T},R)$, indexed by measurable subsets $D\subset {\Xb}$, by:
\beq\label{tofibers}
\begin{array}{c}
(\rho_{D})_{{\xb}}=\rho_{{\xb}},\text{ for }{\xb}\in D\text{ and}\\
(\rho_{D})_{{\xb}}=\delta_{\xb}\ot(\nu\ot\kappa)\text{ or }\delta_{\xb}\ot(\nu_n\ot\kappa) \text{ otherwise}\end{array}\eeq
(depending on whether $\xb\in \Xb_\infty$ or $\xb\in \Xb_n$).
In what follows, in most cases we will only consider real-valued functions. Note that Markov operators send real-valued functions to real-valued functions. Moreover, the mean of a function is zero if and only if the mean of its real part and of imaginary part are zero.

\begin{Lemma}
\label{l:tofibers1}
Assume that $\rho \in J(R,{T})$ with $R$ ergodic. Let $f\in L^2(X,\mu)$ and $g\in L^2_0\zdk$ be real-valued functions and assume that
$f\perp F_{{\rm we},0}({T})$ (in particular,
$\int f\Phi_\rho(g)\,d\mu=0$).
Then, for a.a.\ ${\xb}\in {\Xb_\infty}$, we have
\[
 \int_Y f(\xb,\cdot)\Phi_{\rho_{{\xb}}}(g)\,d\nu=0,
\]
and for every $n$, and a.a.\ ${\xb}\in {\Xb_n}$, we have
\[
 \int_{\{1,\ldots,n\}} f(\xb,\cdot)\Phi_{\rho_{{\xb}}}(g)\,d\nu_n=0.
\]
\end{Lemma}

\begin{proof} We consider the case $\Xb_\infty$ (the proof is the same for other cases). Suppose that for some $\vep_0>0$,
\begin{multline*}
D:=\Big\{{\xb}\in {\Xb_\infty}:\; \int_Y f(\xb,\cdot)\Phi_{\rho_{{\xb}}}(g)\,d\nu\geq \vep_0\Big\}=\\
\Big\{{\xb}\in {\Xb_\infty}\colon \int_{Y\times Z} f(\xb,\cdot)\ot g\,d\rho_{{\xb}}\geq \vep_0\Big\}
\end{multline*}
has positive $\mub$-measure. Note that, since $\int g\,d\kappa=0$, we have
$\int \Phi_{\rho_{{\xb}}}(g)\,d\nu=\int g\,d\kappa=0$. In particular, if for $\xb\in\Xb$, $\rho_{{\xb}}=\delta_{\xb}\ot(\nu\ot\kappa)$, then $\Phi_{\rho_{{\xb}}}(g)=\int \Phi_{\rho_{{\xb}}}(g)\,d\kappa=0$. By our assumption, $f\perp \Phi_{\rho_{D}}(g)$. But then,
\begin{multline*}
\int f \Phi_{\rho_{D}}(g)\,d\mu=\int f\ot g\,d\rho_{D}=
\int\Big(\int  f\ot g\,d(\rho_{D})_{{\xb}}\Big)d\mub({\xb})=\\
\int_{D}\Big( \int f({\xb},y)\Phi_{\rho_{{\xb}}}(g)(y)\,d\nu(y)\Big)d\mub({\xb})\ \geq\ \vep_0\mub(D)\ >\ 0,
\end{multline*}
a contradiction.\end{proof}

In the above sense, being orthogonal to Markov images of ergodic systems is shifted to examine Markov images of ergodic systems on the (ergodic) fibers.

\subsection{Relative ergodic decomposition}

We consider an invertible, bi-measurable transformation $T$ of a standard Borel space $(X,\cb_X)$. Recall that $\ci_T$ is the sub-sigma-algebra of $T$-invariant Borel sets
Recall also that $M(X)$ stands for the space of Borel, probability measures on $X$, and $M(X,T)\subset M(X)$ for the subset of $T$-invariant measures. Let $(Y,\ca)$ be another measurable space (not necessarily standard). Let $\varphi:X\to Y$ be a measurable map which satisfies the property
\beq\label{martim1}
\varphi^{-1}(\ca)\subset \ci_T.\eeq
We also fix a $T$-invariant probability measure $\mu\in M(X,T)$.

\begin{Prop} \label{p:tim1}
There exists a probability kernel $[0,1]\ni t\mapsto \nu_t\in M(X)$ such that:\\
(a) for all $t\in[0,1]$, $\nu_t\in M(X,T)$,\\
(b) we have $$
\int_0^1\nu_t\,dt=\mu$$
and\\
(c) each $\nu_t$ satisfies $\varphi_\ast(\nu_t)=\varphi_\ast(\mu)$ and
$$
\ci_T=\varphi^{-1}(\ca)\text{ mod }\nu_t\; (cf.\ \eqref{martim1}).$$

Moreover, if $\varphi^{-1}(\ca)=\ci_T\text{ mod }\mu$, then for each decomposition satisfying (a), (b) and (c), we have $\nu_t=\mu$ for all $t$.
\end{Prop}

We postpone the proof of Proposition~\ref{p:tim1} to Section~\ref{s:RelErg}.

\subsection{Characteristic classes}
\label{sec:cc}
A class $\cf$ of measure-preserving systems (implicitly closed under isomorphism) is called {\em characteristic} if it is closed under factors and countable joinings (the latter implies that it is closed under inverse limits). An \emph{$\cf$-factor} of $(Z,\cb_{Z},\kappa,R)$ is an $R$-invariant sub-sigma-algebra $\ca$ of $\cb_{Z}$ such that the action of $R$ restricted to $\ca$ defines a system in $\cf$. One of the reasons to consider characteristic classes is the following (for the proof, see e.g.\ \cite{Ru}):

\begin{Th} \label{t:largest} If $\mathcal{F}$ is a characteristic class and $R\in {\rm Aut}\zdk$ then the largest $\cf$-factor $\af=\af(Z,\cb_{Z},\kappa,R)\subset\cb_{Z}$ of $(Z,\cb_{Z},\kappa,R)$ always exists. Moreover, any joining $\rho$ of $R$ with some $(X,\mu,T)\in\cf$ is the relatively independent extension of the relevant restriction of the joining to an element of $J(R|_{\af},T)$. Equivalently, for all $f\in L^2(\kappa)$ and all $g\in L^2(\mu)$,
$$
 \int_{Z\times X}  f(z)\, g(x) \, d\rho = \int_{Z\times X} \EE_\rho \bigl[ f\,|\,\af \bigr](z)\, g(x) \, d\rho(z,x)
$$
\end{Th}

\begin{Cor} \label{c:rozchar} Assume that $\mathcal{F}$ is a characteristic class and $R\in {\rm Aut}\zdk$. Let $S\in \mathcal{F}$. Then  $S\perp R$ if and only if $S\perp R|_{\af}$.\end{Cor}

Of course the class ZE of zero entropy automorphisms is a characteristic class.

\begin{Prop}[\cite{Ka-Ku-Le-Ru}] If $\cf$ is a \emph{proper} characteristic class (that is, if $\cf$ does not include \emph{all} measure-preserving systems), then $\cf\subset$ZE.\end{Prop}

To any given characteristic class $\cf$, we can associate another class $\cfec$, consisting of measure-preserving systems whose almost all ergodic components are in $\cf$. Then $\cfec$ is also a characteristic class, which is called an \emph{ec-class}. It is explained in~\cite{Ka-Ku-Le-Ru} that ec-classes are the right framework for characterizing orthogonality of a bounded arithmetic function to the topological systems whose invariant measures determine measure-preserving systems belonging to a fixed characteristic class. However the ec-issue disappears in Boshernitzan's problem, and the following proposition provides some explanation for this.

\begin{Prop}\label{ec99}Let $\cf$ be a characteristic class. For each automorphism $T$ of $\xbm$ and $f\in L^2\xbm$, we have
$$
\EE_\mu\bigl[f\,|\,\ca_{\cf}(T)\bigr]-\EE_\mu\bigl[f\,|\,\ca_{\cf_{\rm ec}}(T)\bigr] \in F_{\rm we}(T)^\perp.$$
\end{Prop}
\begin{proof}
Take any ergodic automorphism $R$ of $\zdk$ and let $g\in L^2\zdk$. Assume that $\rho\in J(T,R)$. We apply now consecutively: Theorem~\ref{t:largest} for $\cf$ (twice), the ergodicity of $R$ (the $\cf$-factor and the $\cf_{\rm ec}$-factor of $R$ are the same), and Theorem~\ref{t:largest} for $\ca_{\cf_{\rm ec}}$ (twice), 
to obtain the following:
\begin{align*}
\int\EE_\mu\bigl[f\,|\,\ca_{\cf}(T)\bigr]\ot g\,d\rho
&=\int\EE_\mu\bigl[f\,|\,\ca_{\cf}(T)\bigr]\ot \EE_\kappa\bigl[g\,|\,\ca_{\cf}(R)\bigr]\,d\rho \\
&=\int f\ot \EE_\kappa\bigl[g\,|\,\ca_{\cf}(R)\bigl]\,d\rho\\
&=\int f\ot \EE_\kappa\bigl[g\,|\,\ca_{\cf_{\rm ec}}(R)\bigl]\,d\rho\\
&=\int\EE_\mu\bigl[f\,|\,\ca_{\cf_{\rm ec}})(T)\bigr]\ot \EE_\kappa\bigl[g\,|\,\ca_{\cf_{\rm ec}}(R)\bigr]\,d\rho\\
&=\int\EE_\mu\bigl[f\,|\,\ca_{\cf_{\rm ec}}(T)\bigr]\ot g\,d\rho.
\end{align*}
It follows that $\EE_\mu\bigl[f\,|\,\ca_{\cf_{\rm ec}}(T)\bigr]-\EE_\mu\bigl[f\,|\,\ca_{\cf}(T)\bigr]\perp {\rm Im}(\Phi_\rho)$ and since $R$ is arbitrary ergodic, the claim follows.
\end{proof}

It has been already noticed in \cite{Ka-Ku-Le-Ru} that if $\cf$ is a characteristic class and $T$ is automorphism on $\xbm$ such that a.a.\ its ergodic components are in $\cf$ then $T$ {\bf need not} belong to $\cf$.
However, the following general property holds.

\begin{Prop}\label{p:countableCC} Let $\cf$ be a characteristic class and let $T$ be an automorphism of $\xbm$. If $\mu$ has a countable ergodic decomposition with respect to $T$, then $(X,\cb,\mu,T)$ is in $\cf$ if and only if $(X,\cb,\mu_i,T)$ is in $\cf$ for every ergodic component $\mu_i$.\end{Prop}
We postpone the proof of this proposition to Section~\ref{s:cmec}.

\subsection{Measurable selectors}
We recall:
\begin{Th} [Kallman \cite{Kall}] \label{kallman} Suppose that $X$ is a standard Borel space and $Y$ is a separable topological space metrizable by a complete metric. Suppose that $A\subset X\times Y$ is Borel and that for each $x\in X$, $A_x:=\{y\in Y:\: (x,y)\in A\}$ is a countable union of compact sets ($\sigma$-compactness of $A_x$). Let $f:X\times Y\to X$ be projection. Then $f(A)$ is Borel and there is a Borel section $f(A)\to A$ of the map $f|_A$.\end{Th}

We also recall the following useful result (see Theorem 4.5.2 in \cite{Sr}):

\begin{Th}\label{t:graph}
Assume that $X,Y$ are Polish spaces and let $A\subset X$ be Borel. Let $f:A\to Y$. Then $f$ is Borel if and only if ${\rm graph}(f)$ is a Borel subset of $A\times Y$.\end{Th}

For example, by~\eqref{ergdec18},
the set $\{(x,T_x):\:x\in X\}$ is a Borel subset of $X\times {\rm Aut}(Y,\nu)$.

\subsection{Disjointness in ${\rm Aut}(X,\mu)$}\label{s:adeljunco}
We recall that the set ${\rm Aut}(X,\mu)$ of automorphisms of the standard Borel probability space $(X,\mu)$ is a Polish group with the strong topology. Furthermore, the space $M(X\times X)$  of Borel probability measures on $X\times X$ is a Polish space with the weak$^\ast$-topology. Note that  $C_2(\mu)$ which is the set of couplings of $\mu$ with itself, i.e.\ the subset of $M(X\times X)$ consisting of measures with both marginals $\mu$, is a \textbf{compact} subset of $M(X\times X)$. Indeed, this is a consequence of Prokhorov's theorem, as $C_2(\mu)$ is both closed and tight in $M(X\times X)$, the tightness deriving from that of $\{\mu\}$ in $M(X)$.

We consider the following pseudo-metrics ``responsible'' for the topologies on ${\rm Aut}(X,\mu)$ and $C_2(\mu)$, respectively:
\[ d(T,T';Q):=\sum_{q\in Q}\mu(Tq\triangle T'q) \]
and
\beq\label{metrykaA}  d(\rho,\rho';Q):=\sum_{q_1,q_2\in Q}|\rho(q_1\times q_2)-\rho'(q_1\times q_2)|,\eeq
where $Q$ is a finite measurable partition of $X$.

We aim at proving the following:

\begin{Prop}\label{p:andres} The set $\{(S,T):\: S,T\in{\rm Aut}(X,\mu), S\perp T\}$ is a Borel subset of ${\rm Aut}(X,\mu)\times {\rm Aut}(X,\mu)$ (in fact, it is a $G_\delta$-set).\end{Prop}

We will adapt del Junco's proof of Theorem~1 \cite{Ju} to our needs.
Given two finite (measurable) partitions $P,Q$ of $X$ and $\delta,\vep>0$, we let
$$
\mathscr{O}(P,\vep,Q,\delta)$$
denote the set of pairs $(S,T)$ of automorphisms of $(X,\mu)$ satisfying: if $\rho\in C_2(\mu)$ then (notation $(\cdot)_\ast(\rho)$ stands for the image of the measure $\rho$ under the map ``$\cdot$'')
$$
d(\rho,(S\times T)_\ast\rho;Q)<\delta\;\Longrightarrow\; d(\rho,\mu\ot\mu;P)<\vep.$$
Obviously, if $0<\delta'<\delta$ then
$$
\mathscr{O}(P,\vep,Q,\delta) \subset \mathscr{O}(P,\vep,Q,\delta').$$
In fact, we have the following:

\begin{Lemma}\label{l:andres1}
$\mathscr{O}(P,\vep,Q,\delta)\subset {\rm Int}(\mathscr{O}(P,\vep,Q,\delta/3)).$
\end{Lemma}
\begin{proof} Assume that $(S,T)\in \mathscr{O}(P,\vep,Q,\delta)$.  Suppose that (the pseudo-metrics in the product space giving the relevant topology are given by the relevant ``max'')
$$d((S,T),(S',T');Q)<\delta/(3|Q|).$$
We want to show that $(S',T')\in \mathscr{O}(P,\vep,Q,\delta/3)$.
So take $\rho\in C_2(\mu)$ and assume that $d(\rho,(S'\times T')_\ast\rho;Q)<\delta/3$. Then
$$
d((S'\times T')_\ast\rho,(S\times T)_\ast\rho;Q)\leq d((S'\times T')_\ast\rho,(S\times T')_\ast\rho;Q)+$$
$$
d((S\times T')_\ast\rho,(S\times T)_\ast\rho;Q)= $$
$$
\sum_{q_1,q_2\in Q}|\rho(S'q_1\times T'q_2)-\rho(Sq_1\times T'q_2)|+
\sum_{q_1,q_2\in Q}|\rho(Sq_1\times T'q_2)-\rho(Sq_1\times Tq_2)|\leq
$$$$
\sum_{q_1,q_2\in Q}\rho((S'q_1\triangle Sq_1)\times X)+
\sum_{q_1,q_2\in Q}\rho(X\times (T'q_2\triangle Tq_2)\leq$$
$$
\sum_{q_1,q_2\in Q}\mu((S'q_1\triangle Sq_1)+
\sum_{q_1,q_2\in Q}\mu(T'q_2\triangle Tq_2)\leq 2|Q| \delta/(3|Q|)=\frac23\delta.
$$
It follows that
$$
d(\rho, (S\times T)_\ast(\rho);Q)\leq \frac23\delta+\delta/3=\delta$$
and since $(S,T)\in \mathscr{O}(P,\vep,Q,\delta)$, $d(\rho,\mu\ot\mu;P)<\vep$.
\end{proof}

We have obtained that the set $$\mathscr{O}(P,\vep,Q):=\bigcup_{\delta>0} \mathscr{O}(P,\vep,Q,\delta)$$
is open, so the set
$$
\mathscr{O}:=\bigcap_{m,n\in\N}\bigcup_{\ell\in\N}\mathscr{O}(P_m,\frac1n,P_\ell)$$
is $G_\delta$,
where $P_n$ is going to the partition into points and it consists of clopen sets (we assume additionally that $X$ is a zero-dimensional space).
\begin{Lemma}\label{l:andres2} If $(S,T)\notin \mathscr{O}$ then $S\not\perp T$.\end{Lemma}
\begin{proof}
By assumption, there exist $m,n\geq1$ such that for all $\ell\geq1$,
$(S,T)\notin \mathscr{O}(P_m,\frac1n,P_\ell,\frac1\ell)$. Hence, for some $\rho_\ell\in C_2(\mu)$, we have
\beq\label{and1}
d(\rho_\ell,(S\times T)_\ast\rho_\ell;P_\ell)<\frac1\ell,\eeq
\beq\label{and2}
d(\rho_\ell,\mu\ot\mu;P_m)\geq\frac1n.\eeq
Without loss of generality, we can assume that $\rho_\ell\to\rho$. Then $\rho\in C_2(\mu)$ and, by \eqref{and1}, $(S\times T)_\ast\rho=\rho$, so $\rho$ is a joining. While \eqref{and2} tells us that $\rho\neq\mu\ot\mu$, so $T$ and $S$ are not disjoint.\end{proof}

\noindent
{\bf Proof of Proposition \ref{p:andres}} All we need to prove is that $\mathscr{O}=\{(S,T):\: S\perp T\}$. Assume that $(S,T)\in\mathscr{O}$. Then, for each $m,n\geq1$ there are $\ell\geq1$ and $\delta>0$ such that $(S,T)\in \mathscr{O}( P_m,\frac1n,P_\ell,\delta)$. But if $\rho\in C_2(\mu)$ is a joining then this measure is $S\times T$-invariant, so $d(\rho,(S\times T)_\ast\rho;P_\ell)=0<\delta$. Hence, $d(\rho,\mu\ot\mu; P_m)<\frac1n$ and therefore $\rho=\mu\ot\mu$. The result follows now from Lemma~\ref{l:andres2}.

\begin{Remark}\label{r:doroz}
In fact, del Junco proved that  the set of automorphisms disjoint with a fixed one is Borel.\end{Remark}

\begin{Remark}\label{r:dorozMT} We can also study disjointness (and various problems of measurability of natural subsets) when automorphisms are defined on different spaces. For example, if $(X',\mathcal{B}_{X'},\mu')$ is another standard Borel probability space (we assume that $X'$ is a compact metric space), then the set
$$
\{(T,T')\in {\rm Aut}(X,\mu)\times {\rm Aut}(X',\mu')\colon T\perp T'\}$$
is a Borel subset (in fact, it is $G_\delta$). To obtain this result, we repeat the proof of Proposition~\ref{p:andres} with the pseudo-metrics $d$ on $C_2(\mu)$ in~\eqref{metrykaA} replaced with the pseudo-metrics $d$ on the space $C(\mu,\mu')$ of couplings between $\mu$ and $\mu'$ given by
$$
d(\rho_1,\rho_2;Q,Q')=\sum_{q\in Q,q'\in Q'}|\rho_1(q\times q')-\rho_2(q\times q')|$$
where $Q$ and $Q'$ are finite (measurable) partitions of $X$ and $X'$, respectively.\end{Remark}

\Large
\begin{center} PART I\end{center}
\normalsize

\section{Automorphisms disjoint from all ergodic systems}\label{s:adiserg}

In all this section,  we assume that ${T}:(\xb,u)\mapsto\bigl(\xb,T_{\xb}(u)\bigr)$ is an automorphism
acting on the space $X=\bigsqcup_{n\geq1}\Xb_n\times\{1,\ldots,n\}\sqcup \Xb_\infty\times Y$, where $\mu|_{\Xb_n\times\{1,\ldots,n\}}=\mub|_{\Xb_n}\ot\nu_n$,
$\mu|_{\Xb_\infty\times Y}=\mub|_{\Xb_\infty}\ot\nu$, and the automorphisms $T_{\xb}$ are ergodic, cf.\ Section~\ref{s:etr}.

We aim at proving the following theorem:

\begin{Th}\label{t:duzoroz} The following conditions are equivalent:\\
(i)
${T}\in {\rm Erg}^\perp$.\\
(ii) The extension ${T}\to \Id_{\Xb}$ is confined.\\
(iii) For $\mub\ot\mub$-a.a.\ $(\xb,\xb')\in {\Xb}\times {\Xb}$, we have $T_{\xb}\perp T_{\xb'}$.\end{Th}

\begin{Cor}\label{c:duzoroz}
If ${T}\in{\rm Erg}^\perp$ then $h({T})=0$.
\end{Cor}
\begin{proof} Ergodic positive entropy transformations are not disjoint because of Sinai's theorem.\end{proof}

\begin{Cor} If there exists $S\in {\rm Erg}$ such that $\{T_{\xb}\colon \xb\in {\Xb}\}\subset S^{\not\perp}$ then ${T}\not\perp {\rm Erg}$.\end{Cor}
\begin{proof} If $\mub$ has an atom, say $\mub(\{\xb_0\})>0$, then also $\mub\ot\mub(\{(\xb_0,\xb_0)\})>0$ and (iii) of Theorem~\ref{t:duzoroz} applies. Otherwise, use Lemma~\ref{l:ilej}.\end{proof}

\subsection{Proof of (ii) $\Rightarrow$ (i)} The proof follows from Proposition~\ref{p:conf1} since classically $\Id\in{\rm Erg}^\perp$.

\subsection{Proof of  (iii) $\Rightarrow$ (ii)}
Assume that $\rho\in J_2({T})$, $\rho|_{{\Xb}\times {\Xb}}=\mub\ot\mub$. Then, by disintegrating over $\Xb\times\Xb$,
\[
\rho=\int_{\Xb\times\Xb}\rho_{\xb,\xb'}\,d\mub(\xb)d\mub(\xb'),
\]
where, for $(\xb,\xb')\in \Xb_m\times \Xb_n$,  $\rho_{\xb,\xb'}$ is $T_{\xb}\times T_{\xb'}$-invariant, defined on $\{(\xb,\xb')\}\times\{1,\ldots,m\}\times\{1,\ldots,n\}$ if $m,n\in\N$,
 $\{(\xb,\xb')\}\times\{1,\ldots,m\}\times Y$ if $m\in\N$, $n=\infty$ and  $\{(\xb,\xb')\}\times Y\times Y$ if $m=n=\infty$.
Furthermore,
$\rho=\sum_{m,n\in\N\cup\{\infty\}}\rho|_{X_m\times X_n}$,
where
$$
\rho|_{X_m\times X_n}=\int_{\Xb_m\times\Xb_n}\rho_{\xb,\xb'}\,d(\mub|_{\Xb_m})(\xb)d(\mub|_{\Xb_n})(\xb').$$
Consider the case $m=n=\infty$ (the reasoning is similar in all remaining cases). For $A\in\mathcal{B}_{\Xb_\infty}$ and $B\in\mathcal{B}_Y$, we have
\beq\label{selfdis1}
\begin{array}{c}
(\mub|_{\Xb_\infty})\ot\nu(A\times B)=(\rho|_{X_\infty\times X_\infty})(A\times B\times {\Xb}_\infty\times Y)=\\
\int_{A\times {\Xb}_\infty}
\rho_{\xb,\xb'}(B\times Y)\,d(\mub|_{\Xb_\infty})(\xb)d(\mub|_{\Xb_\infty})(\xb').\end{array}\eeq
By substituting $A={\Xb}_\infty$, we first obtain that the measure
$$B\mapsto \int\rho_{\xb,\xb'}(B\times Y)\,d(\mub|_{\Xb_\infty})(\xb)d(\mub|_{\Xb_\infty})(\xb')$$ is $\nu$, the integrand measures are $T_{\xb}$-invariant and since $\nu$ is $T_{\xb}$-ergodic, we must have $\rho_{\xb,\xb'}\in J(T_{\xb},T_{\xb'})$.
But $T_{\xb}\perp T_{\xb'}$ for $\mub\ot\mub$-a.a.\ $(\xb,\xb')\in {\Xb}\times {\Xb}$, so $\rho_{\xb,\xb'}=\nu\ot\nu$ and $\rho|_{X_\infty\times X_\infty}$ is the product measure.\bez

\begin{Remark} Note that under the assumption $(iii)$, we must have $\mub(\Xb_n)=0$ for all $n\in\N\setminus\{1\}$. Indeed, for $2\leq n<\infty$, we have $T_{\xb}=R_n$ for all $\xb\in \Xb_n$.\end{Remark}

\begin{Remark} \label{f:selfdis} Note also that the reasoning in the above proof can be reversed. Indeed, returning to \eqref{selfdis1}, we obtain that if $$\Xb_\infty\times\Xb_\infty\ni(\xb,\xb')\mapsto \rho_{\xb,\xb'}\in J(T_{\xb},T_{\xb'})\text{ is measurable}$$ then $\rho\in J_2(T|_{X_\infty})$ as
$$
\int_{A\times {\Xb_\infty}}
\rho_{\xb,\xb'}(B\times Y)\,d(\mub|_{\Xb_\infty})(\xb)d(\mub|_{\Xb_\infty})(\xb')=$$$$\int_{A\times {\Xb_\infty}}
\nu(B)\,d(\mub|_{\Xb_\infty})(\xb)d(\mub|_{\Xb_\infty})(\xb')=(\mub|_{\Xb_\infty})\ot\nu(A\times B).$$
A similar argument works for the remaining cases of elements in $J(T|_{X_m},T|_{X_n})$.
\end{Remark}

\subsection{Kallman theorem and joinings}\label{s:technikalia}

From~\eqref{ergdec18}, Proposition~\ref{p:andres}, Remark~\ref{r:dorozMT} and Remark~\ref{r:doroz}, we immediately get the following:
\begin{Cor}\label{c:easywn}
 (i) The set $\{(\xb,\xb')\in {\Xb}\times {\Xb}\colon T_{\xb}\perp T_{\xb'}\}$ is Borel.\\
(ii) If $S_0\in {\rm Aut}(Y,\nu)$ (or $S_0\in {\rm Aut}(\{1,\ldots,m\},\nu_m)$, $m\geq1$) then the set $\{\xb\in {\Xb}\colon T_{\xb}\perp S_0\}$ is Borel.
\end{Cor}

Now, we can prove the following result as an application of Kallman theorem.

\begin{Lemma}\label{l:rhoz2} Fix $S_0\in {\rm Aut}(Y,\nu)$ (or $S_0\in {\rm Aut}(\{1,\ldots,m\},\nu_m)$, $m\geq1$) and suppose that 
$\mub(\{\xb\in {\Xb}\colon T_{\xb}\not\perp S_0\})>0$. Then ${T}\not\perp S_0$.
\end{Lemma}
\begin{proof}
Let (by some abuse of notation)
$$C_2(Y,\nu):=\{J\colon L^2(Y,\nu)\to L^2(Y,\nu)\colon J\text{ is Markov}\}$$
be the space of Markov operators corresponding to couplings of $\nu$ with itself. Then
$C_2(Y,\nu)$ is a compact metric space (with the weak operator topology). 

Let $\Pi$ denote the Markov operator corresponding to product measure $\nu\ot\nu$, and consider the set
\beq\label{rhoz3}
\{(\xb,T_{\xb},J)\colon \xb\in {\Xb}, \Pi\neq J\in C_2(Y,\nu), JT_{\xb}=S_0J\}\eeq
which is a Borel subset of
$\bigl({\Xb}\times {\rm Aut}(Y,\nu)\bigr)\times C_2(Y,\nu)$
as the intersection of 
$$
\{(\xb,T_{\xb})\colon \xb\in {\Xb}\}\times (C_2(Y,\nu)\setminus\{\Pi\})$$
and
$${\Xb}\times\{(R,J)\in {\rm Aut}(Y,\nu)\times C_2(Y,\nu)\colon JR=S_0J\}.$$
Let $f$ denote the projection of the set~\eqref{rhoz3} on the first two coordinates. Note that this image is precisely $\{(\xb,T_{\xb})\colon T_{\xb}\not\perp S_0\}$. For each such $(\xb,T_{\xb})$ the fiber consists of all Markov operators $J\in C_2(Y,\nu)\setminus\{\Pi\}$ such that $JT_{\xb}=S_0J$, so it is a compact set up to  the one-element set $\{\Pi\}$, and therefore it is $\sigma$-compact. We can now use Theorem~\ref{kallman}, to obtain a Borel map
$$
\{\xb\in {\Xb}\colon T_{\xb}\not\perp S_0\}\ni \xb \to J_{\xb}\in C_2(Y,\nu)\setminus\{\Pi\}$$
satisfying $J_{\xb}T_{\xb}=S_0J_{\xb}$ which is precisely the fact that $J_{\xb}$ is a joining of $T_{\xb}$ and $S_0$.
Since we assume that $\mub(\{\xb\in {\Xb}\colon T_{\xb}\not\perp S_0\})>0$, we obtain a nontrivial joining between ${T}$ and $S_0$ as in Section~\ref{s:jerg}, see \eqref{tofibers}.
\end{proof}

\subsection{Proof of (i) $\Rightarrow$ (iii)}
(By contraposition.) Suppose that $\mub\ot\mub(\{(\xb,\xb'):\:T_{\xb}\not\perp T_{\xb'}\})>0$.\footnote{The set we are considering is measurable by the same token as (i) in Corollary~\ref{c:easywn}.}. By Fubini's theorem, there is $\xb_0\in {\Xb}$ such that the set $D$  of $\xb'$ satisfying $T_{\xb_0}\not\perp T_{\xb'}$ has positive $\mub$-measure. By Lemma~\ref{l:rhoz2}, it follows that ${T}\not\perp T_{\xb_0}$, hence ${T}\not\in{\rm Erg}^\perp$.

\subsection{Application of the relative ergodic decomposition}

Let us consider now the result stated in Proposition~\ref{p:tim1} in the context of self-joinings of ${T}$ acting on $\xbm$ by
$$
{T}(\xb,u)=(\xb,T_{\xb}(u)),$$
where $T_{\xb}$ are ergodic (as explained in Section~\ref{s:etr}). Let $\la\in J_2({T})$, $\la|_{{\Xb}\times {\Xb}}=\mub\ot\mub$.
The disintegration of $\lambda$ with respect to the projection on ${\Xb}\times {\Xb}$ writes
\begin{equation}
 \label{eq:XxX}\la=\int_{{\Xb}\times {\Xb}}\la_{\xb,\xb'}\,d\mub(\xb)d\mub(\xb'),
\end{equation}
where $\la_{\xb,\xb'}\in J(T_{\xb},T_{x '})$ for $\mub\ot\mub$-a.a.\ $(\xb,\xb')\in {\Xb}\times {\Xb}$.
We can apply Proposition~\ref{p:tim1} to $({T}\times{T},\lambda)$, when $\phi:X\times X\to {\Xb}\times {\Xb}$ is the projection onto ${\Xb}\times {\Xb}$. We obtain that
$$
\la=\int_0^1\la_t\,dt,$$
where, from (a), each $\lambda_t$ is ${T}\times{T}$-invariant, and from (c) the marginal on ${\Xb}\times {\Xb}$ of each $\lambda_t$ is the same as the marginal of $\lambda$, which is $\mub\otimes\mub$, and moreover the system $\bigl(X\times X,\lambda_t,{T}\times{T}\bigr)$ is a relatively ergodic extension of $({\Xb}\times {\Xb},\mub\otimes\mub,\Id)$. It follows that the ergodic decomposition of $\lambda_t$ takes the form
$$\la_t=\int_{{\Xb}\times {\Xb}}\la_{t,\xb,\xb'}\,d\mub(\xb)d\mub(\xb'),$$
where the measures $\la_{t,\xb,\xb'}$ are ergodic (for $T_{\xb}\times T_{\xb'})$.
Integrating with respect to $t$, and comparing with~\eqref{eq:XxX}, we get that for $\mub\otimes\mub$-almost all $(\xb,\xb')$, we have
$$ \lambda_{\xb,\xb'} = \int_0^1 \la_{t,\xb,\xb'}\, dt. $$
Projecting on each $Y$-coordinate the above relation, and remembering that the automorphisms $T_{\xb}$ are ergodic, we get that for Lebesgue-almost all $t\in[0,1 )$ (depending on $(\xb,\xb')$), $\la_{t,\xb,\xb'}$ is an ergodic joining of $T_{\xb}$ and $T_{\xb'}$.

Finally, defining
\begin{multline}
\label{RelErg}
J_2^{\rm RelErg}({T}):= \{\la\in J_2({T})\colon\la|_{{\Xb}\times {\Xb}}=\mub\ot\mub, \\
\la_{\xb,\xb'}\in J^e(T_{\xb},T_{\xb'})\text{ for }\mub\ot\mub\text{-a.a. }(\xb,\xb')\in {\Xb}\times {\Xb}\},
\end{multline}
we obtain the following:

\begin{Cor}\label{c:martim1} Each self-joining $\la$ of ${T}$ projecting into $\mub\ot\mub$ (on ${\Xb}\times {\Xb}$) is of the form $\la=\int_0^1\la_t\,dt$, where $\la_t\in J_2^{\rm RelErg}({T})$ for all $t\in [0,1]$.\end{Cor}

\begin{Remark}
\label{rm:nonempty}
 This proves in particular that $J_2^{\rm RelErg}({T})$ is not empty, since the corollary applies to $\lambda=\mu\otimes \mu$.
\end{Remark}

\subsection{Examples}\label{e:przyk} Let $M=G/\Gamma$, where $G=H_3(\R)$ and $\Gamma=H_3(\Z)$ ($\Gamma$ is an example of a lattice in the Heisenberg group; $M$ is an example of a nil-manifold). On $M$ we consider the (normalized) measure $\kappa$ which is the image of Haar measure on $G$. Given $g\in G$, let $T_g:M\to M$: $T_g(x\Gamma)=gx\Gamma$. Denote by $\pi:M\to\T^2$ the map
$$
\left[\begin{array}{ccc}
1 & a & c\\
0 & 1 & b\\
0 & 0 & 1\end{array}\right]/\Gamma\mapsto (a~{\rm mod}~1,b~{\rm mod}~1)\in\T^2.$$
Note that $\pi$ is well defined. Moreover, since
$$
\left[\begin{array}{ccc}
1 & a & c\\
0 & 1 & b\\
0 & 0 & 1\end{array}\right]\cdot \left[\begin{array}{ccc}
1 & a' & c'\\
0 & 1 & b'\\
0 & 0 & 1\end{array}\right]=\left[\begin{array}{ccc}
1 & a+a' & c'+ab'+c\\
0 & 1 & b+b'\\
0 & 0 & 1\end{array}\right],$$
the fibers of $\pi$ are the circle $\T$. Moreover, the rotation $T_g$ on $M$ is sent to the rotation $R_{(a,b)}$ on $\T^2$. By a theorem of Auslander-Green-Hahn \cite{Au-Gr-Ha}, Chapter~V, if $1,a,b$ are rationally independent (i.e.\ when $R_{(a,b)}$ is ergodic) $T_g$ is ergodic and it is a compact group (circle) extension of the rotation  (from the measure-theoretic point of view $M$ is $\T^2\times\T$ and in these ``coordinates'' $\kappa$ is Leb$_{\T^2}\otimes$Leb$_{\T}$) such that in the orthocomplement of $L^2(\T^2,{\rm Leb}_{\T^2})$ the spectrum is countable Lebesgue.
It follows that when $T_g$ and $T_{g'}$ are ergodic then
\beq\label{heis1}
T_g\perp T_{g'}\Leftrightarrow ka+\ell b=k'a'+\ell'b'\neq0\eeq
for some $k,\ell,k'\ell'\in\Z$ (i.e. $R_{(a,b)}$ and $R_{(a',b')}$ have a non-trivial common factor).

Denote by $q$ the natural quotient map $q:G\to G/\Gamma$ and set $N_{k,\ell,m}:=\{(a,b)\in\T^2\colon k+\ell a+mb=0\}$. Let
$$A:=(\pi\circ q)^{-1}\big(\bigcup_{k,\ell,m}(\T^2\setminus N_{k,\ell,m})\big).$$
Then $A\subset G$ is a subset of full Haar$_G$-measure. Now, given $g=(a,b,c)\in G$, we can find only a subset of $(a',b')\in\T^2$ of zero Leb$_{\T^2}$ measure to see that $R_{(a,b)}\not\perp R_{(a',b')}$. It follows that the condition $T_g\perp T_{g'}$ is satisfied  for Haar$_{G}\otimes$Haar$_{G}$-a.a.\ $g,g')\in G\times G$.

It follows that, by considering {\bf any} probability measure $\mub$ on $G$ equivalent to Haar$_{G}$, the automorphism:
$$
{T}(g,x\Gamma):=(g,gx\Gamma),\;{T}:G\times M\to G\times M,$$
preserves the measure $\mub\ot \kappa$, and $(G,\mub)$ is the space of ergodic components. By Theorem~\ref{t:duzoroz}, ${T}\perp$~Erg.

The above reasoning has its natural generalization for any connected nilpotent $d$-step group $G$ and the corresponding rotations on nil-manifolds. Indeed, as before, we consider the quotient map
$G/\Gamma\to G/([G,G]\Gamma)$, where the latter space is the torus $\T^d$ and a nil-rotation $T_g$ is sent to a rotation on $\T^d$. Moreover, by Auslander-Green-Hahn theorem \cite{Au-Gr-Ha}, Chapter V, $T_g$ is ergodic iff  the corresponding rotation on $\T^d$ is ergodic. As the set of ergodic rotations on $\T^d$ is of full Lebesgue measure, also the set of ergodic nil-rotations on $G/\Gamma$ is of full Haar measure. We now, apply the same reasoning to the product space $(G\times G)/(\Gamma\times \Gamma)$ to obtain that the set of ergodic nil-rotations $T_{g,g'}$ is of full (product) Haar measure. As $T_{g,g'}=T_g\times T_{g'}$, we use now  the following result: assuming $T_g$, $T_{g'}$ ergodic,
$$ T_g\perp T_{g'}\text{ iff the Cartesian product }T_g\times T_{g'}\text{ is ergodic}
$$
(see Corollary~1.5 in \cite{Mo-Ri-Ro}) to conclude that on the set of $(g,g')$ of full (product) Haar measure, we have $T_g\perp T_{g'}$. Hence the  corresponding ${T}$ (see~\eqref{heis1}) is disjoint from all ergodic systems.

\begin{Remark} Automorphisms of that type considered for connected and simply connected nilpotent Lie groups are fundamental to prove Sarnak's conjecture: they satisfy Sarnak's conjecture (we can consider other measures $\mub$ on $G$ than those equivalent to Haar measure, including those with compact support; but we should consider only those for which $T_g\perp T_{g'}$ for $\mub\ot\mub$-a.a. $(g,g')$; then ${T}$ will be considered on the compact space ${\rm supp}(\mub)\times M$ and it will be a homeomorphism of zero entropy) but whether the strong MOMO property holds for them is a big open problem, see \cite{Ka-Ku-Le-Ru}, \cite{Ka-Le-Ri-Te}, \cite{Ta}. We recall that for a fixed $\bfu:\N\to\D$, a topological system $(X,T)$ is said to satisfy the strong MOMO property relative to $\bfu$ (or simply strong $\bfu$-MOMO property) if
\[ \lim_{K\to\infty}\frac1{b_K}\sum_{k\leq K}\Big\|\sum_{b_k\leq n<b_{k+1}}\bfu(n)\cdot f\circ T^n\Big\|_{C(X)}=0 \]
for each choice of $(b_k)$ so that $b_{k+1}-b_k\to\infty$ and all $f\in C(X)$. Clearly, the strong $\bfu$-MOMO implies~\eqref{sarnak10}. If $\bfu$ is M\"obius, we speak about the strong MOMO property.\end{Remark}

Another class of examples of non-ergodic automorphisms  arises when we consider (measurable) flows $\mathcal{S}=(S_t)_{t\in\R}$ on $\ycn$ and we set
\beq\label{newex}
T(t,y)=(t,S_ty)\text{ on }([0,1]\times Y, {\rm Leb}_{[0,1]}\ot \nu).\eeq
Now, using Theorem~\ref{t:duzoroz}, \cite{Fr-Le} (Theorem~8.4 therein) and \cite{Ka-Le-Ul}  (Theorem~1.1 therein), we obtain the following:
\begin{Cor} \label{c:newexamples}The automorphism \eqref{newex} is in ${\rm Erg}^\perp$ whenever:\\
(i)  the flow $\mathcal{S}$ has singular spectrum;\\
(ii) $\mathcal{S}$ is a non-trivial smooth time change of a horocycle flow (each such flow has Lebesgue spectrum).\end{Cor}
Note that our classical example $(x,y)\mapsto (x,y+x)$ on $\T^2$ is a particular case of (i) (use as $\mathcal{S}$ the linear flow on $\T$) and that the papers \cite{Fr-Le}, \cite{Ka-Le-Ul} provide many  examples of smooth flows on surfaces for which the automorphisms \eqref{newex} are in ${\rm Erg}^\perp$.

\section{Orthogonality to ergodic Markov images}\label{s:oemi}
\subsection{Preparatory observations and some motivations}\label{ss:ghk}
Before we formulate the main result of this section, we need some preparatory observations.
We recall that we consider ${T}:(\xb,u)\mapsto(\xb,T_{\xb}u)$ acting on $(X,\mu)$ with the fiber automorphisms $T_{\xb}$ ergodic. It follows that the extension
\beq\label{now1}
{T}\to {T}|_{\Xb}=\Id_{\Xb}\text{ is relatively ergodic,}\eeq
i.e.\ the sigma-algebra generated by the projection on $\Xb$ coincides with $\ci_{{T}}$ mod $\mu$.
Let us consider now $\la\in J_2({T})$, $\la|_{\Xb\times \Xb}=\mub\ot\mub$.
If $f\in L^2(X,\mu)$ and $g\in L^2(\Xb,\mub)$ then for each $N\geq1$,
$$\int f(x)g(\xb')\,d\la|_{X\times {\Xb}}(x,\xb')=\frac1N\int\sum_{n\leq N}f\circ{T}^n(x)g(\xb')\,d\la|_{X\times {\Xb}}(x,\xb').$$
Letting $N\to\infty$, and using the von Neumann theorem together with the fact that $\Xb$ corresponds to $\ci_{{T}}$, we obtain that
\begin{equation*}
\begin{split}
 \int f(x)g(\xb')\,d\la|_{X\times {\Xb}}(x,\xb') = \int \EE(f\,|\,\Xb)(\xb)g(\xb')\,d\mub(\xb)d\mub(\xb') \\
  = \int f\,d\mu\int g\,d\mub.
\end{split}
\end{equation*}
Hence, we have
\beq\label{now3}
\la|_{X\times {\Xb}}=\mu\ot \mub,
\eeq
see also Section~\ref{s:howtorecognize}.
If now $f\in L_0^2(X,\mu)$
and if $\widetilde{f}:=\EE_{\mu}[f\,|\,{\Xb}]$, then \eqref{now3} implies that
\begin{multline}\label{now4a}
\EE_{\la}[(f-\widetilde{f})\ot (f-\widetilde{f})\,|\,\Xb\times \Xb]=\EE_{\la}[f\ot f\,|\,\Xb\times \Xb]\\-\EE_{\mu}[f\,|\,{\Xb}]\ot \EE_{\mu}[f\,|\,{\Xb}]
\end{multline}

as
\beq\label{qqq}\EE_\lambda[f\ot \widetilde{f}\,|\,\Xb\times \Xb]=(1\ot\widetilde{f})\,
\EE_\lambda\bigl[(f\ot1)\,|\,\Xb\times \Xb\bigr]=\widetilde{f}\ot\widetilde{f}\eeq
the latter equality due to~\eqref{now3}.
Moreover, noticing that $\widetilde{f}\perp F_{\rm we}({T})$, we obtain
\beq\label{now4}
\mbox{$f\perp F_{\rm we}({T})$ if and only if
$f-\widetilde{f}\perp F_{\rm we}({T})$.}\eeq
Therefore, in what follows, we will constantly assume that $f$ satisfies
\beq\label{now5}
\EE[f\,|\,{\Xb}]=0.\eeq

In order to understand how natural this assumption is,
recall first the notion of Gowers-Host-Kra norms (GHK norms, for short) of $f$:
\beq\label{ghk1}\|f\|^2_{u^1}:=
\lim_{H\to\infty}\frac1H\sum_{h\leq H}\int f\circ {T}^h\cdot\ov{f}\,d\mu,\eeq
whence,
\beq\label{ghk2}\|f\|^2_{u^1}=\Big|\int f\,d\mu\Big|^2\text{ if }{T}\text{ is ergodic}\eeq
and $\|f\|^{2^{s+1}}_{u^{s+1}}:=\lim_{H\to\infty}\frac1H\sum_{h\leq H}\|f\circ {T}^h\cdot\ov{f}\|_{u^s}^{2^s}$ for $s\geq1$.
Note that, by the von Neumann theorem, $\|f\|_{u^1}=\int|\EE[f\,|\,{\Xb}]|^2\,d\mub$.
Remember that we aim at a study of (bounded) arithmetic functions $\bfu$ through the associated Furstenberg systems $\kappa\in V(\bfu)$ and we consider $f=\pi_0$, cf.~\eqref{veechin}. In other words, given a sequence $(N_k)$ which yields a Furstenberg system $\kappa\in V(\bfu)$, we can define $\|\bfu\|_{u^s}$ as $\|\pi_{0}\|_{u^s}$ for the system $(X_{\bfu},\kappa,S)$.
In particular, we say that $\|\bfu\|_{u^s}=0$ if $\|\pi_0\|_{u^s}=0$ for all $\kappa\in V(\bfu)$.
Hence, in the problem of classifying arithmetic functions $\bfu$ which are orthogonal to uniquely ergodic systems, the assumption~\eqref{now5} is now explained by the following classical result (see also Section~3.4.1 in \cite{Fe-Ku-Le}):
\begin{Prop}\label{p:ju1}
 The following conditions are equivalent:\\
(a) $\|\bfu\|_{u^1}=0$.\\
(b) For each $(N_k)$ defining a Furstenberg system of $\bfu$, we have
$$\limsup_{H\to\infty}\lim_{k\to\infty}\frac1{N_k}\sum_{n\leq N}\Big|\frac1H\sum_{h\leq H}\bfu(n+h)\Big|^2=0$$
(i.e.\ the property of ``zero mean on typical short interval'').\\
(c) $\EE_\kappa[\pi_0\,|\,\ci_S]=0$ for each $\kappa\in V(\bfu)$.\\
(d) $\lim_{H\to\infty}\frac1H\sum_{h\leq H}\pi_0\circ S^h=0$ in $L^2(X_{\bfu},\kappa)$ for each $\kappa\in V(\bfu)$.
\end{Prop}

\begin{Remark}\label{r:ju2} (For some basics about multiplicative, bounded by~1 functions and their multiplicative distance $D$, see Section~\ref{s:pretensious}.)
A break-through theorem by Matom\"aki and Radziwi\l\l \  \cite{Ma-Ra} established zero value for the first GHK norm for many classical multiplicative (unpretensious) functions like the Liouville or M\"obius functions. But there are also pretentious multiplicative functions $\bfu$ for which $\|\bfu\|_{u^1}=0$. In fact, all multiplicative functions pretending being non-principal Dirichlet character $\chi$ enjoy this property. Indeed, as shown in \cite{Fr-Le-Ru}, Furstenberg systems of all such arithmetic functions are ergodic,  so (in view of~\eqref{ghk2}) our claim follows once we know that they have mean equal to zero. To see the latter claim, note that for the non-principal Dirichlet character $\chi$, we  have
$D(\chi,n^{it})=+\infty$ for all $t\in\R$, and $D(\bfu,\chi)<+\infty$. It follows that $D(\bfu,n^{it})=+\infty$  as  the multiplicative ``distance'' $D$ satisfies the triangle inequality. The claim now follows from  the Hal\'as theorem.\end{Remark}

\begin{Remark}\label{r:dodatniau1} Note that if all Furstenberg systems of $\bfu$ are identities, then $\|\bfu\|_{u^1}>0$ unless the Besicovitch norm vanishes $\|\bfu\|_{B}=0$ (see Section~\ref{s:pretensious}). Indeed, if in the definition of $u^1$ of $f$, the automorphism $T$ is an identity then $\|f\|_{u^1}=\|f\|_{L^2}$. Hence, by an observation in \cite{Go-Le-Ru}, all non-trivial slowly varying functions have positive GHK first norm. In particular, for all Archimedean characters, we have $\|n^{it}\|_{u^1}>0$.\end{Remark}

Let us consider some examples of orthogonality to ergodic Markov images.
\begin{enumerate}
\item[A.]  If ${T}$ is ergodic, then  all $f\in L^2({X})$ are in the Markov image of $\Phi_{\Delta_\mu}$, where $\Delta_\mu$ stands for the diagonal joining on $(X,\mu)$.
\item[B.] If for $\mub\ot\mub$-a.a.\ $\xb,\xb'\in {\Xb}$, we have $T_{\xb}\perp T_{\xb'}$, then, by Theorem~\ref{t:duzoroz}, all zero mean functions are orthogonal to all ergodic Markov images.
\end{enumerate}

\begin{Prop}
\label{p:Mprod}
Assume that for all $\xb\in {\Xb}$, $T_{\xb}=S$ (in other words, ${T}=\Id_{\Xb}\times S$ with $S$ ergodic). 
Then for all $f\in L^2(\mub\ot\nu)$, $f\perp F_{{\rm we},0}({T})$ if and only if $f$ is measurable with respect to $\xb$.
\end{Prop}
\begin{proof} Suppose that $f\perp F_{{\rm we},0}({T})$.  Then, by considering the joining $\rho$ of ${T}$ with $S$ in which we join diagonally $S$ with itself and the $\Xb$ part is independent of $Y$, and using Lemma~\ref{l:tofibers1}, for every $g\in L_0^2(Y,\nu)$, we have
\beq\label{Mprod1}\int_{Y}f(\xb,y)g(y)\,d\nu(y)= 0\eeq
for $\mub$-a.a.\ $\xb\in {\Xb}$. This $\mub$-a.a.\ condition persists if we consider a countable linearly dense set of $g\in L_0^2(Y)$.
It follows that for $\mub$-a.a. $\xb\in {\Xb}$, $y\mapsto f(\xb,y)$ is $\nu$-almost everywhere equal to $\int_Y f(\xb,y)\,d\nu(y)$, \emph{i.e.} $f$ is $\mub\ot\nu$-almost surely equal to a function of $\xb$.

Conversely, if $f$ is $\xb$-measurable, the fact that $f\perp F_{{\rm we},0}({T})$ is a consequence of the disjointness between identities and ergodic systems.
\end{proof}

\subsection{Characterization result}
We aim at proving the following result.

\begin{Th}\label{t:OrMaIm} Assume that $f\in L^2_0(X,\mu)$ satisfies~\eqref{now5}. Then
$f\perp F_{\rm we}({T})$ if and only if for all $\la\in J^{\rm RelErg}_2({T})$, we have
$$ \EE_{\la}[f\ot f\,|\,\Xb\times \Xb](\xb,\xb')=0$$
for $\mub\ot\mub$-a.a.\ $(\xb,\xb')\in \Xb\times \Xb$.
\end{Th}

\begin{itemize}
\item Note that $\la\in J_2^{\rm RelErg}({T})$ if and only if $\la|_{{\Xb}\times {\Xb}}=\mub\ot\mub$ and $\ci_{{T}\times{T}}$ is the sigma-algebra generated by the projection to $\Xb\times \Xb$ modulo $\lambda$.
\item Under the assumption on $\la$, $\EE_{\la}[f\ot f\,|\,\Xb\times \Xb](\xb,\xb')=0$ for $\mub\ot\mub$-a.a.\ $(\xb,\xb')\in \Xb\times \Xb$ if and only if $\lim_{N\to\infty}\frac1N\sum_{n\leq N}(f\otimes f)\circ ({T}\times {T})^n=0$ (by the von Neumann theorem).
\item Note also that (assuming the validity of Theorem~\ref{t:OrMaIm}), due to Corollary~\ref{c:martim1}: $f\perp F_{\rm we}({T})$ if and only if for all $\la\in J_2({T})$, we have
$\EE_{\la}[f\ot f\,|\,\Xb\times \Xb](\xb,\xb')=0$ for $\mub\ot\mub$-a.a.\ $(\xb,\xb')\in \Xb\times \Xb$.
\item For another equivalent condition (important for applications), see Corollary~\ref{c:relvN2}.
\end{itemize}

\begin{Remark}\label{r:module} It follows directly from Theorem~\ref{t:OrMaIm} that if $f\in L^2_0(X,\mu)$ satisfies \eqref{now5} and $f\perp F_{\rm we}({T})$ then, for each $g\in L^\infty(\Xb,\mub)$, we have $gf\perp F_{\rm we}({T})$.\end{Remark}

\begin{Remark} \label{ft:thm4.5} If we do not want to assume \eqref{now5} then the assertion $\EE_{\la}[f\ot f\,|\,\Xb\times \Xb](\xb,\xb')=0$, in view of \eqref{now4a}, should be replaced by
$$(\ast)\qquad\EE_{\la}[f\ot f\,|\,\Xb\times \Xb](\xb,\xb')=\EE_{\mu}(f\,|\,\Xb)(\xb) \, \EE_{\mu}(f\,|\,\Xb)(\xb')$$
for $\mub\ot\mub$-a.a.\ $(\xb,\xb')\in \Xb\times \Xb$. For example, if we take $\bfu$ such that all its Furstenberg systems are identity, then the only self-joining that we have to consider is the product measure, and the assumption $(\ast)$ holds. We will see later (see Proposition~\ref{p:oues}) that in this case $\bfu$ will be orthogonal to all u.e.\ systems even though the function $\pi_0$ does not satisfy~\eqref{now5}.
\end{Remark}

We need the following:

\begin{Lemma}\label{l:ormaim1}
Assume that $\rho\in J({T},R)$ with $R$ ergodic. Set, by some abuse of notation  $\la:=\rho^\ast\circ \rho\in J_2({T})$
(i.e.\ $\lambda$ is the only joining for which $\Phi_{\la}=\Phi_\rho^\ast\circ\Phi_\rho$). Then $\la|_{\Xb\times \Xb}=\mub\ot\mub$.\end{Lemma}
\begin{proof} Since $R$ is ergodic,  $\rho|_{\Xb\times Z}=\mub\ot\kappa$. That is, $\Phi_\rho|_{L_0^2(\Xb,\mub)}=0$. By the definition of $\la$,
$$
\Phi_{\la}|_{L_0^2(\Xb)}=\Phi^\ast_{\rho}\circ\Phi_{\rho}|_{L_0^2(\Xb)}=0,$$
whence $\la|_{\Xb\times \Xb}=\mub\ot\mub$.\end{proof}

Under the assumptions of the lemma, we have $\rho=\int_{\Xb}\rho_{\xb}\,d\mub(\xb)$, where $\rho_{\xb}\in J(T_{\xb},R)$ for $\mub$-a.a.\ $\xb\in {\Xb}$.
Moreover, by disintegrating
$$
\la=\int_{\Xb\times \Xb}\la_{\xb,\xb'}\,d\mub(\xb)d\mub(\xb'),$$
we obtain that
\beq\label{zloz}
\la_{\xb,\xb'}=\rho^\ast_{\xb'}\circ \rho_{\xb}.\eeq
Indeed,
$$
\Phi_{\la}=\Phi^\ast_{\rho}\circ\Phi_{\rho}= \int_{\Xb}\Phi^\ast_{\rho_{\xb}}\,d\mub(\xb)
\circ \int_{\Xb}\Phi_{\rho_{\xb'}}\,d\mub(\xb')=
\int_{\Xb\times \Xb}\Phi^\ast_{\rho_{\xb}}\circ \Phi_{\rho_{\xb'}}\,d\mub(\xb)d\mub(\xb').$$

\subsection{Proof of Theorem~\ref{t:OrMaIm}. Sufficiency}

Take any ergodic $R$ on $\zdk$ and let $\rho\in J(T,R)$.
 Set $\la=\rho^\ast\circ\rho\in J_2({T})$. In view of Lemma~\ref{l:ormaim1}, $\la|_{\Xb\times \Xb}=\mub\ot\mub$. Now, by Corollary~\ref{c:martim1},  $\la=\int_0^1\la_t\,dt$ with $\la_t\in J_2^{\rm RelErg}({T})$. We then consider $f\in L^2(X,\mu)$ such that (in particular) $\EE_{\la_t}[f\ot f|\,\Xb\times\Xb]=0$ $\mub\ot\mub$-a.e., whence $\EE_{\la}[f\ot f\,|\,\Xb\times \Xb]=0$ $\mub\ot\mub$-a.e. Hence, by~\eqref{zloz}, for $\xb\in\Xb_\infty$ (if $\xb\in \Xb_n$ we replace $\nu$ with $\nu_n$ in the calculation below)
\begin{multline*}
0=\EE_{\la}[f\ot f\,|\,\Xb\times \Xb](\xb,\xb')=\int f\ot f\, d\la_{\xb,\xb'}=\\
\int\Phi_{\la_{\xb,\xb'}}(f)
f|_{\{\xb\}\times Y}\,d\nu=\int_{Z}\Phi_{\rho_{\xb}}(f)\cdot \Phi_{\rho_{\xb'}}(f)\,d\kappa
\end{multline*}
for $\mub\ot\mub$-a.a.\ $(\xb,\xb')\in \Xb\times \Xb$. In view of \eqref{doda28}, all we need to show is that $\Phi_\rho(f)=0$, where $\Phi_\rho:L^2(X,\mu)\to L^2(Z,\kappa)$.
A point $\xb\in {\Xb}$ is called ``good'' if
$$
\mub(\{\xb'\in {\Xb}:\:\Phi_{\rho_{\xb}}(f)\perp \Phi_{\rho_{\xb'}}(f)\})=1.$$
Since the set
 of $(\xb,\xb')$ such that $\Phi_{\rho_{\xb}}(f)\perp \Phi_{\rho_{\xb'}}(f)$ is Borel, by Fubini's theorem the set $G$ of ``good'' points has full measure. Set
 $$
 C:=\{\xb\in {\Xb}:\: \Phi_{\rho_{\xb}}(f)\neq0\}.$$
 Suppose that $\mub(C)>0$. Then $\mub(G\cap C)=\mub(C)>0$. In the set $C\cap G$ select a maximal subset $\{\xb_n:\: n\geq1\}$, so that $\Phi_{\rho_{\xb_n}}(f)\perp \Phi_{\rho_{\xb_m}}(f)$ for $n\neq m$ (this set is countable because $L^2(Z,\kappa)$ is separable, and we can find a maximal set by Zorn's lemma). Let
 $$
 B_n:=\{\xb\in {\Xb}:\:\Phi_{\rho_{\xb}}(f)\perp \Phi_{\rho_{\xb_n}}(f)\}.$$
 Then $\mub(B_n)=1$ and therefore
 $$
 \mub(\bigcap_{n\geq1}B_n\cap G\cap C)>0$$
 and adding any $\xb\in \bigcap_{n\geq1}B_n\cap G\cap C$ makes a larger family than $\{\xb_n:\:n\geq1\}$, a contradiction.

 It follows that the set $C$ has measure zero, and this yields $\Phi_\rho(f)=0$.

\subsection{Proof of Theorem~\ref{t:OrMaIm}. Necessity}

(By contraposition.)
Suppose that there exists $\la\in J_2^{\rm RelErg}({T})$ for which the claim does not hold. Without loss of generality, we assume that $f$ is real, and that, for $(\xb,\xb')$ in a set of positive $\mub\otimes\mub$-measure, we have
\[
 \EE_{\la}[f\ot f\,|\,\Xb\times \Xb](\xb,\xb') = \int f\ot f\,d\la_{\xb,\xb'} > 0.
\]
Using Fubini's theorem, there exists a point ${\xb_0}$ such that, for $\xb$ from a set $A\subset \Xb$ of positive measure,
\[
 \int f\ot f\,d\la_{\xb,\xb_0}\ >\ 0.
\]
We set $R=T_{{\xb_0}}$ (which is ergodic, we also assume that $\xb_0\in \Xb_\infty$), and we define $\rho\in J({T},R)$ by setting (cf.\ Section~\ref{s:jerg})
$\rho:=\int \rho_{\xb}\,d\mub(\xb)$, where $\rho_{\xb}=\la_{\xb,{\xb_0}}$ for $\xb\in A$, and $\rho_{\xb}$ equals the product measure otherwise. We easily obtain, using \eqref{now5} for the last equality
\begin{multline*}
 \int \Phi_\rho(f)\cdot f\,d\mu=
\int f\ot f\,d\rho= \\
\int_{\Xb}\Big(\int f\ot f\,d\rho_{\xb}\Big)\,d\mub(\xb)=
\int_{A}\Big(\int f\ot f\,d\rho_{\xb}\Big)\,d\mub(\xb)
\ >\ 0.
\end{multline*}

\subsection{A  condition for being orthogonal to ergodic Markov images}
We need the following lemma (relativizing the case of ergodic decomposition).
\begin{Lemma}\label{l:relvN}
Let $R$ be an automorphism of $\zdr$ and let $\ca\subset\ci_R$ be a factor. Assume that $\rho=\int_0^1\rho_t\,dQ(t)$, where $\rho_t$ are $R$-invariant and $\ca=\ci_R\mod \rho_t$  for $Q$-a.a.\ $t\in[0,1]$. Let $f\in L^\infty(Z,\rho)$ and $\EE_{\rho_t}[f|\ca]=0$ for $Q$-a.a.\ $t\in[0,1]$.
Then $\EE_\rho[f\,|\,\ci_R]=0$.
\end{Lemma}
\begin{proof} Denote $F_{N,t}:=\int_Z|\frac1N\sum_{n\leq N}f\circ R^n|^2\,d\rho_t$. Then, for $Q$-a.a.\ $t\in[0,1]$, we have $\lim_{N\to\infty}F_{N,t}=0$ by the von Neumann theorem (in $L^2(Z,\rho_t)$) and our assumptions. By the Lebesgue dominated theorem: $\lim_{N\to\infty}\int_0^1F_{N,t}\,dQ(t)=0$. That is,
$$
0=\lim_{N\to\infty}\int_0^1
\Big(\int_Z\Big|\frac1N\sum_{n\leq N}f\circ R^n\Big|^2\,d\rho_t\Big)dQ(t)=
\lim_{N\to\infty}\int_Z\Big|\frac1N
\sum_{n\leq N}f\circ R^n\Big|^2\,d\rho.$$
It follows again by the von Neumann theorem that $\EE_\rho[f\,|\,\ci_R]=0$.\end{proof}
Returning to our context, by Corollary~\ref{c:martim1}, Theorem~\ref{t:OrMaIm} and Lemma~\ref{l:relvN}, we obtain the following:
\begin{Cor}\label{c:relvN2}
Assume that $f\in L^2_0(X,\mu)$ satisfies~\eqref{now5}. Then
$f\perp F_{\rm we}({T})$ if and only if for all $\la\in J_2({T})$ satisfying $\la|_{\Xb\times \Xb}=\mub\ot\mub$, we have
$\EE_{\la}[f\ot f\,|\,\ci_{{T}\times{T}}]=0$.\end{Cor}

\begin{Remark}\label{r:Bosher} The above corollary yields the full solution of Boshernitzan's problem (see next section). To test whether $f\perp F_{\rm we}({T})$ and assuming that $\EE(f|\Xb)=0$, we only need (and it is sufficient) to show that $$(\ast)\;\;\;\;\lim_{N\to\infty}
\int_{{X}\times{X}}\Big|\frac1N\sum_{n\leq N}(f\ot f)\circ ({T}\times{T})^n\,d\la\Big|^2=0$$ for all self-joinings $\la$ of ${T}$ for which $\la|_{\Xb\times \Xb}=\mub\ot\mub$. When we consider Furstenberg systems of $\bfu$ (satisfying the condition: ``zero mean on typical short interval''), then $f=\pi_0$ and the condition $(\ast)$ is expressed combinatorially (i.e.\ expressed in terms of properties of $\bfu$), see Proposition~\ref{p:ju1} and next sections.\end{Remark}

\begin{Remark}\label{r:Bosher1}
If ${T}\perp {\rm Erg}$ then of course each $f$  satisfying $\EE(f|\Xb)=0$ is orthogonal to all ergodic images, but we can also see that the condition $(\ast)$ is trivially satisfied because each $\la$ as above will simply be product measure (as the extension ${T}\to{\Id}|_{\Xb}$ is confined).\end{Remark}
\begin{Remark}
Note that if $\lambda|_{\Xb\times \Xb}=\mub\ot\mub$ then (unless ${T}$ is ergodic) $\lambda$ cannot be the
graph measure ${\mu}_{S}$ for an element ${S}$ of the centralizer $C({T})$ of ${T}$. Indeed, if $S\in C({T})$ then ${S}$ must preserve the sigma-algebra of invariant sets. We obtain
$$
{\mu}_{{S}}|_{\Xb\times \Xb}=\mu_{\bar S},$$ where ${\bar S}:={S}|_{\Xb}$. But $\mu_{\bar S}\neq\mub\ot\mub$ unless $\Xb$ reduces to one point.\end{Remark}

\section{Orthogonality to uniquely ergodic systems}
We recall that (cf.~\eqref{f:defort}) the meaning of orthogonality of $\bfu$ to UE is that $$\lim_{N\to\infty}\frac1N\sum_{n\leq N}f(T^nx)\bfu(n)=0$$ for each uniquely ergodic topological system $(X,T)$, $f\in C(X)$ of zero mean (with respect to the unique invariant measure) and all $x\in X$.

\subsection{Orthogonality and ergodic Markov images}
\begin{Prop} \label{p:oues}
Let $\bfu:\N\to\D$. Then $\bfu\perp\,{\rm UE}$ if and only if for each Furstenberg system $\kappa\in V(\bfu)$, $\pi_0\perp F_{{\rm we},0}(X_{\bfu},\kappa,S)$.
\end{Prop}

\begin{proof} $\Rightarrow$ (by contradiction) Suppose that for some $\kappa\in V(\bfu)$ there is an ergodic system $(Z',\nu',R')$ and a joining $\rho'\in J(\nu',\kappa)$ for which $\pi_0\not\perp {\rm Im}(\Phi_{\rho'}|_{L^2_0(Z',\nu')})$. Using the Jewett-Krieger theorem, we obtain that there exists a uniquely ergodic system $(Z,\nu,R)$ and a joining $\rho\in J(\nu,\kappa)$ such that $\pi_0\not\perp {\rm Im}(\Phi_\rho|_{L^2_0(Z,\nu)})$ (in $L^2(\kappa)$).  It follows that there exists $f\in C(Z)$ of zero mean such that
$\int_{X_{\bfu}} \Phi_\rho(f) \pi_0\,d\kappa \neq0$, in other words
$\int_{Z\times X_{\bfu}}f\ot \pi_0\,d\rho\neq0$. Let $(N_m)$ satisfy $\frac1{N_m}\sum_{n\leq N_m}\delta_{S^n\bfu}\to\kappa$. In view of Theorem~\ref{t:lifting}, there exist a sequence $(z_n)\subset Z$ and a subsequence $(N_{m_\ell})$ such that
$$
\frac1{N_{m_\ell}}\sum_{n\leq N_{m_\ell}}\delta_{(z_n, S^n\bfu)}\to \rho$$
and the set $\{n\geq1:\: z_{n}\neq Rz_{n-1}\}$ is contained in a subset of $\NN$ of the form $\{b_1<b_2<\ldots\}$, where
$b_{k+1}-b_k\to\infty$. Adding if necessary some more $b_k$'s to this set, we may also assume that
$\lim_{k\to\infty} \frac{b_{k+1}}{b_k} = 1$. In this way, defining $K_\ell:=\max\{k:b_k\le N_{m_\ell}\}$, we get $\lim_{\ell\to\infty} \frac{b_{K_\ell}}{N_{m_\ell}} = 1$, and it follows that
\begin{multline*}
0\neq \int f\ot \pi_0\,d\rho=\lim_{\ell\to\infty} \frac1{N_{m_\ell}}\sum_{n\leq N_{m_\ell}}f(z_n)\bfu(n)\ = \\
\lim_{\ell\to\infty}\frac1{b_{K_\ell}}\sum_{k<K_\ell}\Big(\sum_{b_k\leq n<b_{k+1}}f(R^{n-b_k}z_{b_k})\bfu(n)\Big).
\end{multline*}
However, the latter limit is~$0$ because of Theorem~\ref{t:orbmod}  and our assumption of orthogonality on $\bfu$.

$\Leftarrow$ (by contradiction) Let $(X,T)$ be uniquely ergodic. Suppose that for some subsequence $(N_k)$    we have
$\frac1{N_k}\sum_{n\leq N_k}f(T^nx)\bfu(n):=c\neq0$ (for some zero mean $f\in C(X)$ and $x\in X$) and $\frac1{N_k}\sum_{n\leq N_k}\delta_{(T^nx,S^n\bfu)}\to\rho$. Then $\rho$ is a joining, $\rho\in J(\nu,\kappa)$, where $\nu$ is the unique invariant measure for $T$, and $\kappa\in V(\bfu)$. Since now
$$
c=\int f\ot \pi_0\,d\rho=\int \EE_\kappa[f\,|\,X_{\bfu}]\pi_0\,d\kappa,$$
we obtain a contradiction as $\nu$ is ergodic.
\end{proof}

\begin{Remark}
If we assume that $M(\bfu)=0$ (which is equivalent to $\EE_\kappa[\pi_0]=0$ for each Furstenberg system $\kappa$ of $\bfu$), then in the proposition above we can replace
$F_{{\rm we},0}(X_{\bfu},\kappa,S)$ with $F_{{\rm we}}(X_{\bfu},\kappa,S)$.
\end{Remark}

\begin{Remark} Note that in the proof we used unique ergodicity twice: each ergodic system has a uniquely ergodic model and then the use of  Theorem~\ref{t:orbmod} for orbital uniquely ergodic models. As the latter models are not (in general) minimal, the question arises whether orthogonality with respect to all uniquely ergodic systems is the same as orthogonality to all strictly ergodic systems.
\end{Remark}

\begin{Remark} If all Furstenberg systems of $\bfu$ are identities then $\bfu$ is orthogonal to all u.e.\ systems. Hence (see \cite{Go-Le-Ru}) all mean slowly varying functions
are orthogonal to all u.e.\ systems.
In the next section we will consider Boshernitzan's problem in the class of multiplicative functions.
\end{Remark}

\begin{Remark} \label{r:gdylog} If orthogonality is given by the logarithmic way of averaging then all the results in the paper persist, with the only change that instead of considering Furstenberg systems of $\bfu$ we need to consider all logarithmic Furstenberg systems of $\bfu$. For example, $\bfu\perp_{\rm log} {\rm UE}$ if and only if for each $\kappa\in V^{\rm log}(\bfu)$, we have $\pi_0\perp F_{{\rm we},0}(X_{\bfu},\kappa,S)$.\end{Remark}

\subsection{Pretentious multiplicative function orthogonal to all uniquely ergodic systems}\label{s:pretensious}
We consider arithmetic functions  $\bfu:\N\to\D$ which are {\em multiplicative}, i.e.\ \beq \label{comu}
\bfu(mn)=\bfu(m)\bfu(n)\text{ whenever }m,n\text{ are coprime},\eeq
where $\D$ denotes the unit disc (if~\eqref{comu} is satisfied unconditionally then we speak about {\em complete multiplicativity} of $\bfu$). Let
$$\|\bfu\|_B:=\limsup_{N\to\infty}\frac1N\sum_{n\leq N}|\bfu(n)|$$
stand for the Besicovitch pseudo-norm of $\bfu$. For example, $\bfu(n):=n^{-r}$, where $r>0$, is a completely multiplicative function of zero Besicovitch pseudo-norm.
We recall that (see, e.g.\ \cite{Be-Ku-Le-Ri}, Lemma 2.9)
\beq\label{Besico1}
\|\bfu\|_{B}=0 \Leftrightarrow \sum_{p\in\PP}\frac1p(1-|\bfu(p)|)=\infty\eeq
(the RHS in \eqref{Besico1} yields many other possibilities to create multiplicative functions of zero Besicovitch pseudo-norm). When $\|\bfu\|_{B}=0$ then $\bfu(n)\to0$ along a subsequence of full density in $\N$. It is easy to see that in this case  $\pi_0=0$ $\kappa$-a.e., where $\kappa=\delta_{\underline{0}}$ is the only Furstenberg system of $\bfu$. It follows that  multiplicative functions of Besicovitch pseudo-norm zero are orthogonal  to all topological systems (this fact does not use the multiplicativity of $\bfu$),  in particular,  to all uniquely ergodic systems. Another easy case arises when $\bfu(n)=n^{it}$ ($t\in\R$) is an Archimedean character. In this case, $\bfu$ is mean slowly varying function, i.e.\ it satisfies~\eqref{msvf}: $\frac1N\sum_{n\leq N}|\bfu(n+1)-\bfu(n)|\to0$ when $N\to\infty$ (in fact, $(n+1)^{it}-n^{it}\to0$ when $n\to\infty$); note also that, in general,  if $\|\bfu\|_B=0$ then $\bfu$ is mean slowly varying.
It is already remarked in \cite{Go-Le-Ru} that mean slowly varying functions (without the multiplicativity restriction) are precisely those whose all Furstenberg systems yield identities (so we obtain functions orthogonal to all uniquely ergodic systems). Klurman \cite{Kl} proved that if $\bfu$ is multiplicative ($|\bfu|\leq1$) and mean slowly varying then either $\|\bfu\|_B=0$ or $\bfu$ is an Archimedean character (Klurman's theorem contains also the unbounded case: $n\mapsto n^{z}$, where $0<{\rm Re}\,z<1$ which we do not consider as we assume $|\bfu(n)|\leq1$).
Recall that the multiplicative ``distance'' between two multiplicative functions $\bfu,\bfv:\N\to\D$ is defined as
\[
D(\bfu,\bfv)^2:=\sum_{p\in\PP}\frac1p(1-{\rm Re}(\bfu(p)\ov{\bfv(p)})).
\]
Then,  $\bfu$ is called {\em pretentious} if for some $t\in\R$ and a Dirichlet character $\chi$, we have
$D(\bfu,\chi\cdot n^{it})<+\infty$.
 Note that, in view of~\eqref{Besico1}, $\bfu$ is {\bf not} pretentious whenever $\|u\|_B=0$.
We aim at proving the following result.
\begin{Cor}\label{c:PretMult}
The only pretentious multiplicative functions $\bfu:\N\to\D$ orthogonal to all uniquely ergodic systems are Archimedean characters $n\mapsto n^{it}$, $t\in\R$.\end{Cor}
\begin{proof}
Suppose that $\bfu$ is not an Archimedean character. Then $D(\bfu, \chi\cdot n^{it})<\infty$, where $\chi$ is a (non-trivial) primitive Dirichlet character. If $t=0$ then by Theorem 2.8(i) \cite{Fr-Le-Ru} all Furstenberg systems of $\bfu$ are ergodic, and each of them is non-trivial (unless $\bfu=1$). If $t\neq0$, then since $\bfu$ is (by assumption) not an Archimedean character, its spectrum cannot be trivial by Theorem 2.8(ii) \cite{Fr-Le-Ru}, and it follows by the same theorem that each Furstenberg system of it is of the form ``identity$\times T$'' with $T$ ergodic, and the ergodic part is non-trivial.
Now, since $\bfu\perp{\rm UE}$, $\pi_0\perp F_{{\rm we},0}(X_{\bfu},\kappa,S)$ for each $\kappa\in V(\bfu)$ in view of Proposition~\ref{p:oues}.
By Proposition~\ref{p:Mprod}, it follows that for each Furstenberg system $\kappa$, $\pi_0$ is $\kappa$-almost-surely measurable with respect to $\ci_S$. But then the sigma-algebra generated by the process $(\pi_0\circ S^n)$ is $\ci_S$, which yields an identity action, contradicting non-triviality of the ergodic part of the Furstenberg systems.
\end{proof}
\begin{Remark}In view of the recent development around the Chowla and Sarnak conjectures, it is reasonable to expect that Archimedean characters and zero Besicovitch pseudo-norm multiplicative functions are the only (bounded by~1) multiplicative functions which are orthogonal to all uniquely ergodic sequences.\end{Remark}

\subsection{Examples of sequences orthogonal to all uniquely ergodic systems}
\begin{Lemma}\label{l:lis1}
Let $(X,T)$ be a topological system and let $(x_n)\subset X$ satisfy:\\
(i) the density of  $\{n\in\N\colon x_{n+1}\neq Tx_n\}$ is zero,\\
(ii) $(x_n)$ is quasi-generic along $(N_k)$ for $\nu\in M(X)$  (cf.\ Theorem~\ref{t:lifting}).\\
Then $\nu$ is $T$-invariant. \end{Lemma}
\begin{proof} Obvious, by considering the integrals of $f, f\circ T\in C(X)$. \end{proof}

\begin{Remark} Recall that (ii) means $\lim_{k\to\infty}\frac1{N_k}\sum_{n\leq N_k}f(x_n)=\int f\,d\nu$ for all $f\in C(X)$. A general problem arises of whether generic sequences satisfying (i) do exist. The lifting lemma (when applied directly), see Theorem~\ref{t:lifting},  yields the existence of such a sequence $(x_n)$ satisfying (i) and also (ii) along a subsequence $(N_k)$;  however we do not control other subsequences.
\end{Remark}

\begin{Lemma}\label{l:lis2} Assume that $(Z,R)$ is another topological system in which $z$ is a generic point for a measure $\kappa$ (which must be in $M(Z,R)$). Suppose that
$$
\frac1{N_k}\sum_{n\leq N_k}\delta_{(x_n,R^nz)}\to\rho \in M(X\times Z).$$
Then $\rho$ is $T\times R$-invariant.\end{Lemma}
\begin{proof} The sequence $(x_n,R^nz)$ is generic for $\rho$ along a subsequence and satisfies (i) from Lemma~\ref{l:lis1} for $T\times R$.\end{proof}
\begin{Lemma}\label{l:lis3} Assume that in Lemma~\ref{l:lis1} additionally $(X,\nu,T)\in {\rm Erg}^\perp$ for {\bf any}  $(N_k)$ along which we have convergence, i.e.\ for each $\nu\in V((x_n))$. Then, for each $f\in C(X)$ , we have
$$
(\bfu(n):=f(x_n))\perp\,{\rm UE}. $$
\end{Lemma}
\begin{proof} Take any $(Z,R)$ u.e.\ (with a unique $R$-invariant measure $\kappa$). Suppose that for some $(N_k)$ and $z\in Z$, we have
$$
\frac1{N_k}\sum_{n\leq N_k}\delta_{(x_n,R^nz)}\to \rho.$$
By Lemma~\ref{l:lis2}, it follows that $\rho\in M(X\times Z,T\times R)$ and, moreover, the marginals of $\rho$ are $\nu$ and $\kappa$, respectively. It follows that for each $\kappa$-zero mean $g\in C(Z)$, we have
$$
\lim_{k\to\infty}\frac1{N_k}\sum_{n\leq N_k}\bfu(n)g(R^nz)=\int f\ot g\,d\rho=
\int f\,d\nu\cdot\int g\,d\kappa=0$$
by the disjointness of $(X,\nu,T)$ and $(Z,\kappa,R)$.
\end{proof}

\begin{Remark}\label{r:ujeden} Assume for simplicity that $f:X\to\C$ has modulus~1. Note that for the sequence $\bfu$ in the lemma if $\frac1{N_k}\sum_{n\leq N_k}\delta_{S^n\bfu}\to\kappa$, and
$\frac1{N_k}\sum_{n\leq N_k}\delta_{x_n}\to\nu$
then  for all $m_i,\ell_i\in \Z$   ($i=1,\ldots,r$), if we set $g:=\prod_{i\leq r} \pi_{m_i}^{\ell_i}$, then using (i) from Lemma~\ref{l:lis1} to justify the fifth equality and~(ii) for the last one, we have
\begin{align*}
 \int_{X_{\bfu}} g\,d\kappa &= \lim_{k\to\infty}\frac1{N_k}\sum_{n\leq N_k}g(S^n\bfu) \\
 &= \lim_{k\to\infty}\frac1{N_k}\sum_{n\leq N_k}\prod_{i\leq r}\pi_{m_i}^{\ell_i}(S^n\bfu) \\
 &= \lim_{k\to\infty}\frac1{N_k}\sum_{n\leq N_k}\prod_{i\leq r}\bfu_{n+m_i}^{\ell_i} \\
 &= \lim_{k\to\infty}\frac1{N_k}\sum_{n\leq N_k}\prod_{i\leq r}\bigl(f(x_{n+m_i})\bigr)^{\ell_i} \\
 &= \lim_{k\to\infty}\frac1{N_k}\sum_{n\leq N_k}
\prod_{i\leq r}\bigl(f(T^{m_i}x_{n})\bigr)^{\ell_i}=\int_{X} \prod_{i\leq r}(f\circ T^{m_i})^{\ell_i}\,d\nu.
\end{align*}
It follows that, for any $H\geq1$,
$$
\int_{X_{\bfu}}\Big|\frac1H\sum_{h\leq H}\pi_0\circ S^h\Big|^2\,d\kappa=
\int_{X}\Big|\frac1H\sum_{h\leq H}f\circ T^h\Big|^2\,d\nu.$$
Hence, using the von Neumann theorem, we obtain that
\beq\label{ujeden1}
\EE_\nu[f\,|\,\ci_T]=0\text{ iff } \EE_\kappa[\pi_0\,|\,\ci_S]=0.\eeq
In view of Proposition~\ref{p:ju1}, with the additional assumption that $\EE_\nu[f\,|\,\ci_T]=0$, for each $\nu$ for which $(x_n)$ is quasi-generic, the sequence $\bfu$ which we obtain satisfies $\|\bfu\|_{u^1}=0$.
\end{Remark}

\begin{Remark}\label{r:ujedenA} If we also assume that the topological system $(X,T)$ has zero topological entropy,
and that the set $\{n\in\N\colon x_{n+1}\neq Tx_n\}$ is of the form $\{b_1<b_2<\cdots\}$ with $b_{k+1}-b_k\to\infty$, then the sequence $\bfu=(f(x_n))$ has zero topological entropy as well (see~\cite[Proof of Corollary~9]{Ab-Ku-Le-Ru}).
\end{Remark}

\vspace{1ex}

Now, we provide examples of sequences which are orthogonal to all uniquely ergodic systems. Fix an irrational $\alpha\in [0,1)$. We define the sequence $(w_n)\subset \T^2$ in the following way:
\beq\label{eq:star}\begin{array}{c}
(\alpha,0),(2\alpha,0),(2\alpha,2\alpha),(2\alpha,4\alpha),(2\alpha,6\alpha),
\ldots\\
\ldots,(k\alpha,0),(k\alpha,k\alpha),\ldots,\bigl(k\alpha,(k^2-1)k\alpha\bigr),\ldots\end{array}
\eeq
This sequence is the concatenation of pieces, indexed by $k\geq1$: at stage $k$ we take $(k\alpha,0)$ and the initial $k^2$-long piece of its orbit via the map $T:(x,y)\mapsto (x,x+y)$ on $\T^2$, cf.\ Section~\ref{e:przyk}. We claim that the sequence we defined is generic for $\nu={\rm Leb}_{\T}\ot{\rm Leb}_{\T}$. Indeed, to see this, first notice that we only need to check the case $N=1^2+2^2+\ldots+L^2$ as the RHS is of order $L^3$ while the next group in the definition of $(w_n)$ has length $(L+1)^2$ which is $o(1)$ with respect to $N$. For $r,s\in\Z$, consider $F(x,y)=e^{2\pi i(rx+sy)}$ (we check the weak convergence of our sequence testing on characters of $\T^2$). Assume first that $s\neq0$, then we have
$$
\frac1N\sum_{n\leq N}F(w_n)=\frac1N\sum_{k\leq L}e^{2\pi irk\alpha}\sum_{j=0}^{k^2-1}e^{2\pi isjk\alpha}=
\frac1N\sum_{k\leq L}e^{2\pi irk\alpha}\frac{e^{2\pi ik^2sk\alpha}-1}{e^{2\pi isk\alpha}-1}.
$$
Fix $\vep_0>0$. Then, if $L$ is large enough,
$$
\Big|\{k\leq L\colon \Big|e^{2\pi ik(s\alpha)}-1\Big|\geq\vep_0\}\Big|\geq(1-3\vep_0)L$$
and for each $k$ in this large set, we have
$$
\sum_{j=0}^{k^2-1}e^{2\pi isjk\alpha}=O(1/\vep_0).$$
For the remaining $k\leq L$ the sums are bounded by $k^2$. Thus, remembering  that $\sum_{m\leq M} m^2=\frac13M^3+O(M^2)$, we obtain
\begin{align*}
\Big|\frac1N\sum_{n\leq N}F(w_n)\Big| &\leq \frac1N\sum_{k\leq L}O(1/\vep_0)+
\frac1N\sum_{k=L(1-3\vep_0)}^Lk^2\\
&\leq O(1/\vep_0)\frac LN+\frac1N\Big(\frac13L^3-\frac13(L(1-3\vep_0))^3+O(L^2)\Big)\\
&=o(1)+O(\vep_0)+\frac{O(L^2)}N=o(1)+O(\vep_0)
\end{align*}
when $L\to\infty$.

Now, let us deal with the case $(s=0,r\neq0)$. We then write, using Abel's summation formula,
\begin{multline}
 \label{eq:Abel}
 \frac1N\sum_{n\leq N}F(w_n)=\\
 \frac1N\sum_{1\leq k\leq L} k^2 e^{2\pi ikr\alpha}
= \frac1N \sum_{1\leq j\leq L} \bigl(j^2-(j-1)^2\bigr) \sum_{j\leq k \leq L} e^{2\pi ikr\alpha}.
\end{multline}
For a fixed $\vep_0>0$, we consider the contribution of $j$'s which are less than $(1-\vep_0)L$: for such a $j$, we have
$$ \left| \frac{1}{L-j+1} \sum_{j\leq k \leq L} e^{2\pi ikr\alpha} \right| \leq \frac{1}{\vep_0L} \left|\frac{2}{1-e^{2\pi ir\alpha}}\right|= \frac1L O(1/\vep_0),$$
and since
$$ \frac1N \sum_{1\leq j\leq L} \bigl(j^2-(j-1)^2\bigr) (L-j+1) = 1 $$
(this corresponds to~\eqref{eq:Abel} when $r=0$),
we get
$$  \left| \frac1N \sum_{1\leq j\leq (1-\vep_0)L} \bigl(j^2-(j-1)^2\bigr) \sum_{j\leq k \leq L} e^{2\pi ikr\alpha}  \right| = \frac1L O(1/\vep_0). $$
In the remaining terms, corresponding to $j>(1-\vep_0)L$, we just use
$$ \left|\sum_{j\leq k \leq L} e^{2\pi ikr\alpha} \right| = O(\vep_0 L) $$
to get
$$  \left| \frac1N \sum_{ (1-\vep_0)L < j\leq L} \bigl(j^2-(j-1)^2\bigr) \sum_{j\leq k \leq L} e^{2\pi ikr\alpha}  \right| = \frac1N O(L^3\vep_0)=O(\vep_0). $$
We finally have for $(s=0,r\neq 0)$
$$\Big|\frac1N\sum_{n\leq N}F(w_n)\Big| = \frac1L O(1/\vep_0) + O(\vep_0),
$$
which concludes the proof of the genericity of the sequence $(w_n)$.

We now apply Lemma~\ref{l:lis3} to $T:\T^2\to\T^2$, $T(x,y)=(x,x+y)$ and $f(x,y)=e^{2\pi iy}$ and the sequence~\eqref{eq:star} to obtain that the sequence
\beq\label{waznyciag}\begin{array}{c}
1,1,e^{2\pi i 2\alpha},e^{2\pi i 4\alpha},e^{2\pi i 6\alpha},
\ldots\\
\ldots,1,e^{2\pi i k\alpha},\ldots,e^{2\pi i (k^2-1)k\alpha},\ldots\end{array}
\eeq
is orthogonal to all uniquely ergodic systems.

\subsection{Properties of arithmetic functions orthogonal to all u.e. systems} \label{s:arithm}
We start by the following consequence of Theorem~\ref{t:OrMaIm}.

\begin{Cor}\label{c:thi1} Let $T\in{\rm Aut}\xbm$ and let $f\perp F_{\rm we}({X},{\mu},{T})$. Let $\eta$ be a ${T}$-invariant probability measure with $\eta\ll {\mu}$. Assume that $d\eta/d{\mu}$ is bounded so that $f$ can also be considered in $L^2(X,\eta)$. Then, we have $f\perp F_{\rm we}({X},\eta,{T})$.\end{Cor}

\begin{proof} Since $(X,\mu)=\bigsqcup_n(\Xb_n\times \{1,\ldots,n\},\mub|_{\Xb_n}\ot\nu_n)\sqcup
(\Xb_\infty\times Y,\mub|_{\Xb_\infty}\ot\nu)$ (with the action $(\xb,u)\mapsto (\xb,T_{\xb}(u))$) and $f\in F_{\rm we}(X,\mu,T)$, also $f|_{X_n}\in F_{\rm we}(T|_{X_n})$. But if $n\in\N$, the action of $T|_{X_n}$ is of the form $\Id\times R_n$, with $R_n$ ergodic, whence $f|_{X_n}$ must be $T|_{X_n}$-invariant by Proposition~\ref{p:Mprod}, so it is in fact $\Xb_n$-measurable. Therefore, also $f|_{X_n}$ will  be $T|_{X_n}$-invariant for $\eta|_{X_n}$. It follows that w.l.o.g., we can assume that $T$ is aperiodic, that is, we can represent $({X},{\mu},{T})$ in the form $(\Xb\times Y,\mub\ot\nu, (\xb,y)\mapsto (\xb,T_{\xb}y))$ (with $x\mapsto T_{\xb}$ being the ergodic decomposition). Then $h:=d\eta/d(\mub\ot\nu)$ is $\Xb$-measurable (because $\eta$ is ${T}$-invariant and so is $h$). We consider the probability measure $\mub_h$ on $\Xb$ defined by $d\mub_h/d\mub=h$, so that $\eta=\mub_h\ot\nu$. Since $f\perp F_{\rm we}({X},{\mu},{T})$, the condition in the statement of Theorem~\ref{t:OrMaIm}
(see also Remark~\ref{ft:thm4.5})
is satisfied, and we rewrite it as follows: for each measurable map $\lambda:(\xb,\xb')\mapsto \lambda_{\xb,\xb'}\in J(T_{\xb},T_{\xb'})$ such that $\lambda_{\xb,\xb'}$ is ergodic $\mub\ot\mub$- a.s, we have
$$
\int_{Y\times Y}f(\xb,y)f(\xb',y')\,d\lambda_{\xb,\xb'}(y,y')=\Big(\int_Yf(\xb,y)\,d\nu(y)\Big)\Big(
\int_Yf(\xb',y)\,d\nu(y)\Big).
$$
But the same condition remains satisfied if we replace everywhere $\mub$ with $\mub_h$. Indeed, we know (see Remark~\ref{rm:nonempty}) that there exists a map $\rho:(\xb,\xb')\mapsto \rho_{\xb,\xb'}\in J(T_{\xb},T_{\xb'})$ such that $\rho_{\xb,\xb'}$ is ergodic for $\mub\ot\mub$- a.a.\ $(\xb,\xb')\in \Xb\times \Xb$. Therefore, if $\lambda:(\xb,\xb')\mapsto \lambda_{\xb,\xb'}\in J(T_{\xb},T_{\xb'})$ satisfies $\lambda_{\xb,\xb'}$ is ergodic $\mub_h\ot\mub_h$- a.e., we can always modify this map by setting
$$ \tilde\lambda_{\xb,\xb'}:=\begin{cases}
                        \lambda_{\xb,\xb'} &\text{ if }h(\xb)h(\xb')>0,\\
                        \rho_{\xb,\xb'} &\text{ otherwise.}
                       \end{cases}
$$
Then $\tilde\lambda_{\xb,\xb'}$ is ergodic for $\mub\ot\mub$-a.a.\ $(\xb,\xb')\in \Xb\times \Xb$, and therefore, for $\mub\ot\mub$-a.a.\ $(\xb,\xb')$, we have
$$
\int_{Y\times Y}f(\xb,y)f(\xb',y')\,d\tilde\lambda_{\xb,\xb'}(y,y')=\Big(\int_Yf(\xb,y)\,d\nu(y)\Big)\Big(
\int_Yf(\xb',y)\,d\nu(y)\Big).
$$
But $\tilde\lambda_{\xb,\xb'}=\lambda_{\xb,\xb'}$, $\mub_h\ot\mub_h$-a.s,
so again by Theorem~\ref{t:OrMaIm}, this proves $f\perp F_{\rm we}({X},\eta,{T})$.
\end{proof}

Now, we will prove the following result.

\begin{Prop}\label{p:thi2} Let $\bfu$, $\|\bfu\|_{u^1}=0$, be a bounded arithmetic function orthogonal to all u.e.\ systems. Then $\bfu$ is also orthogonal to all systems whose set of ergodic invariant measures is countable. More precisely: if $(Z,R)$ is a topological system whose set $\{\rho_i\colon i\in I\}$ of ergodic measures is countable, then for any $z_0\in Z$ and any function $g\in C(Z)$,
we have
$$
\frac1N \sum_{n<N} \bfu(n) g(R^nz_0)\to 0.
$$
\end{Prop}

\begin{proof}
Let $(N_k)$ be an increasing sequence along which $(u,z_0)$ is generic for some $S\times R$-invariant probability measure $\kappa$:
$$
\frac1{N_k}\sum_{n<N_k}\delta_{(S^n\bfu,R^nz_0)}\to\kappa\text{ when }k\to\infty.$$
Let $\rho$ be the marginal of $\kappa$ on $Z$. Then $\rho=\sum_{i\in I}\alpha_i\rho_i$ for some non-negative real numbers $\alpha_i$ and $\sum_{i\in I}\alpha_i=1$. For each $i\in I$, set
$$Z_i:=\{z\in Z\colon z\text{ is generic for }\rho_i\}.$$
Then, $Z_i$ is Borel, $R$-invariant and $Z_i\cap Z_j=\emptyset$ for $i\neq j$. Moreover, $\rho(\bigcup_{i\in I}Z_i)=1$. Let $\mu$ be the marginal of
$\kappa$ on $X_{\bfu}$: $\mu$ determines a Furstenberg system of $\bfu$. Since we assume $\bfu\perp{\rm UE}$, Proposition~\ref{p:oues} ensures that, in $L^2(X_{\bfu},\mu)$, we have
$$\pi_0\perp F_{\rm we}(X_{\bfu},\mu,S).$$
For each $i\in I$ such that $\alpha_i>0$, we let $\kappa_i\ll\kappa$ be the probability measure defined by
$$
\frac{d\kappa_i}{d\kappa}:=\frac1{\alpha_i}\raz_{X_{\bfu}\times Z_i}.$$
This Radon-Nikodym derivative is $S\times R$-invariant, hence, $\kappa_i$ is
$S\times R$-invariant. Let $\mu_i$ be the marginal of $\kappa_i$ on $X_{\bfu}$. We have:
\begin{itemize}
\item $\mu_i\ll\mu$,
\item $\mu_i$ is $S$-invariant,
\item $\frac{d\mu_i}{d\mu}\le \frac{1}{\alpha_i}$.
\end{itemize}
Furthermore, we have
$$\frac1{N_k}\sum_{n<N_k}\bfu(n)g(R^nz_0)\to
\int_{X_{\bfu}\times Z}\pi_0\ot g\,d\kappa=
\sum_{i\in I}\alpha_i \int_{X_{\bfu}\times Z}\pi_0\ot g\,d\kappa_i.$$
But $\kappa_i$ is a joining of $\mu_i$ with the ergodic measure $\rho_i$. By Corollary~\ref{c:thi1}, we have $\pi_0\perp F_{\rm we}(X_{\bfu},\mu_i,S)$ and finally, we get
$$
\frac1{N_k}\sum_{n<N_k}\bfu(n)g(R^nz_0)\to0.$$
\end{proof}

\begin{Remark}\label{r:blabla} If we drop the assumption $\|\bfu\|_{u^1}=0$, then we obtain that for
any function $g\in C(Z)$,
satisfying $\int_Z g \,d\rho_i=0$ for all $i\in I$,
we have
$\frac1N \sum_{n<N} \bfu(n) g(R^nz)\to 0$ when $N\to\infty$ (for all $z\in Z$).\end{Remark}

If we take $\bfu\perp {\rm UE}$ ($\|\bfu\|_{u^1}=0$), we know that $\bfu$ is in fact orthogonal to all topological systems whose set of invariant measures is countable. Using Proposition~\ref{p:thi2} and proceeding now as in the proof P$_2\Rightarrow {\rm P}_3$ in Section 2.2 \cite{Ab-Ku-Le-Ru}, we obtain the following result.

\begin{Cor}\label{c:thi3}
Suppose that $\bfu\perp {\rm UE}$ ($\|\bfu\|_{u^1}=0$). Then, for any uniquely ergodic system $(Y,S)$, for any increasing $(b_k)$ with $b_{k+1}-b_k\to\infty$, for all $(y_k)\subset Y$ and all $f\in C(Y)$ with $\int f\,d\nu=0$ ($\nu$ is the unique invariant measure on $Y$), we have
$$\lim_{K\to\infty}\frac1{b_K}\sum_{k<K}\Big|\sum_{b_k\leq n<b_{k+1}}f(S^{n-b_k}y_k)\bfu(n)\Big|=0,$$
i.e.\ $(Y,S)$ satisfies the strong MOMO property (relative to $\bfu$).\end{Cor}

\subsection{Mean slowly varying functions are multipliers of the UE$^\perp$}
\noindent
 We say that an arithmetic function $\bfv$ is a {\em multiplier} for the problem of orthogonality to all u.e.\ systems if
$$\bfu\cdot\bfv\perp {\rm UE}\text{ whenever } \bfu\perp {\rm UE}.$$
Note that if $\bfv$ is a multiplier and $\widetilde{\bfv}$ satisfies
$\|\bfv-\widetilde{\bfv}\|_B=0$ then also $\widetilde{\bfv}$ is a multiplier.
Note also that $\bfv(n)=n^{it}$ is a multiplier in the above sense since taking any u.e.\ sequence $(a(n))$ and denoting $s_n:=\sum_{k\leq n}\bfu(k)a(k)$, we have (summation by parts):
$$
\frac1N\sum_{n\leq N}\bfv(n)\bfu(n)a(n)=o(1)+\frac1N\sum_{n\leq N}n(\bfv(n+1)-\bfv(n))\frac{s_n}{n}
$$
and the claim follows from the fact that $(n+1)^{it}-n^{it}=O(1/n)$ (Archimedean characters are slowly varying functions with a speed). A natural question arises whether
we can replace here Archimedean characters by any function whose all Furstenberg systems are identities, in other words, is the above true for all mean slowly varying functions \cite{Go-Le-Ru}? We will see that the answer to this question is positive.

We need a general lemma on mean slowly varying functions.
\begin{Lemma}\label{l:thi4}
If $\bfv$ is a (bounded) mean slowly varying function then there exist:
\begin{itemize}\item
another mean slowly varying function $\widetilde{\bfv}$ such that
$$\frac1N\sum_{n\leq N}|\bfv(n)-\widetilde{\bfv}(n)|\to 0\text{ when }N\to\infty,$$
that is, the two functions are equal modulo Besicovitch pseudo-norm: $\|\bfv-\widetilde{\bfv}\|_B=0$,
\item
an increasing sequence $(b_k)$ satisfying $b_{k+1}-b_k\to\infty$,
\item a bounded sequence $(z_k)\subset\C$,
\item a sequence $0<\vep_k\to0$ monotonically\end{itemize}
such that for all $k\geq1$ and all $n\in\{b_k,\ldots,b_{k+1}-1\}$ we have $|\widetilde{\bfv}(n)-z_k|<\vep_k$.
\end{Lemma}
\begin{proof}
Since $\bfv$ is mean slowly varying, for any $\delta>0$, we have
$$
\frac1N\sum_{n<N}\raz_{|\bfv(n)-\bfv(n+1)|\geq\delta}\to 0\text{ when }N\to\infty.$$
Let us fix a decreasing sequence $0<\delta_j\to0$. Then, we get an increasing sequence $(M_j)$ such that for all $j\geq1$ and all $N\geq M_j$, we have
$$
\frac1N\sum_{n<N}\raz_{|\bfv(n)-\bfv(n+1)|\geq\delta_j}<\frac1{2^j}.$$
We now define
$$
B:=\bigcup_{j\geq1}\{n\colon M_j\leq n<M_{j+1}, |\bfv(n)-\bfv(n+1)|\geq\delta_j\}.$$
If $M_j\leq N<M_{j+1}$ then we have
$$
\frac1N\sum_{n<N}\raz_{B}(n)\leq \frac1N\sum_{n<N}\raz_{|\bfv(n)-\bfv(n+1)|\geq\delta_j}<\frac1{2^j}.$$
The $b_k$'s in the statement of the lemma are destined to cover this set where the sequence has large gaps. But as we want to have $b_{k+1}-b_k\to\infty$, we have to avoid integers in $B$ which are too close to each other. Let us define more precisely what we mean: we say that two consecutive elements $b,b'$ of $B$ are \emph{too close} if
\begin{itemize}
\item $M_j\leq b<M_{j+1}$ for some $j$, and
\item $b'\leq b+j$.
\end{itemize}
Then we set $C:=\bigcup_b\{b,b+1,\ldots,b'-1\}$, where the union ranges over all pairs $(b,b')$ in $B$ which are too close to each other. For each $n$ in $C$, we now change $\bfv(n)$ by setting $\widetilde{\bfv}(n):=\bfv(\ell(n))$, where $\ell(n)$ is defined as the largest integer smaller than $n$ which is not in $C$. For $n\in\N\setminus C$, we set $
\widetilde{\bfv}(n):=\bfv(n)$. Define $\widetilde{B}$ from $\widetilde{\bfv}$ as we defined $B$ from $\bfv$. We get $\widetilde{B}\subset B$, but we have removed
from $B$ all integers $b$ which are too close to the next element of $B$. Observe now that we have for $M_j\leq N<M_{j+1}$
$$
\frac1N|C\cap\{1,\ldots, N\}|\leq j\frac1N|B\cap\{1,\ldots,N\}|\leq \frac{j}{2^j}\to0$$
when $j\to\infty$. This shows that $\frac1N\sum_{n<N}|\bfv(n)-\widetilde{\bfv}(n)|\to 0$. Moreover, the set $\widetilde{B}$ (where $|\widetilde{\bfv}(n)-\widetilde{\bfv}(n-1)|$ is ``too large'') now satisfies: if $b<b'$ are two elements in $\widetilde{B}$ with $M_j\leq b<m_{j+1}$ then $b'-b>j$. Then, we can choose the increasing sequence $(b_k)$ such that:
\begin{itemize}
\item $\widetilde{B}\subset \{b_k\colon k\geq1\}$,
\item if $M_j\leq b_k< M_{j+1}$ then $b_{k+1}-b_k\in\{j,\ldots,2j-1\}$.
\end{itemize}
Then, the inequality $b_{k+1}-b_k\geq j$ ensures that $b_{k+1}-b_k\to\infty$,
and the inequality  $b_{k+1}-b_k\leq 2j$ together with the fact that all $n\in\{b_k+1,\ldots, b_{k+1}-1\}$ are not in $\widetilde{B}$ ensure that for any such $n$,
     $$
     |\widetilde{\bfv}(n)-\widetilde{\bfv}(b_k)|\leq 2j\cdot \delta_j.$$
We set $z_k:=\widetilde{\bfv}(b_k)$, and we choose $\delta_j$ so that $2j\delta_j<2^{-j}$. Then we get
$$
\sup_{n\in\{b_k,\ldots, b_{k+1}-1\}}|\widetilde{\bfv}(n)-z_k|\to0$$ when $k\to\infty$.
\end{proof}
\begin{Prop}\label{p:t+m}
If $\bfv$ is mean slowly varying then $\bfv$ is a multiplier of the class UE$^\perp$.\end{Prop}
\begin{proof} Using Lemma~\ref{l:thi4}, we only need to show that $\widetilde{\bfv}$ is a multiplier. Given $\bfu\perp{\rm UE}$ and $(X,T)$ a uniquely ergodic system with $f\in C(X)$ and $x\in X$,
we have
$$
\Big|\frac1{b_K}\sum_{k<K}\sum_{b_k\leq n<b_{k+1}}\bfu(n)\widetilde{\bfv}(n)f(T^nx)\Big|\leq$$$$
\frac1{b_K}\sum_{k<K}|z_k|\Big|\sum_{b_k\leq n<b_{k+1}}\bfu(n)f(T^nx)\Big|+
\Big(\sum_{k<K}\vep_k\frac{b_{k+1}-b_k}{b_K}\Big)\|\bfu\|_\infty\|f\|_\infty$$
and when $K\to\infty$, the first summand is going to zero because of the strong MOMO property (Corollary~\ref{c:thi3}), while the second also goes to zero since $\vep_k\to0$.
\end{proof}

We have been unable to answer the converse: Is every multiplier of ${\rm UE^{\perp}}$ a mean slowly varying function?

\subsection{How to recognize that a self-joining of a Furstenberg system has product measure as projection?}
\label{s:howtorecognize}
So we have a Furstenberg system $(X_{\bfu},\kappa,S)$ of $\bfu$ with identifications
$$
(X_{\bfu},\kappa)=
\bigsqcup_n(\Xb_{\bfu,n}\times\{1,\ldots,n\},
\ov{\kappa}|_{\Xb_{\bfu,n}}\ot\nu_n)\sqcup (\Xb_{\bfu,\infty}\times Y,
\ov{\kappa}|_{\Xb_{\bfu,\infty}}\ot\nu)$$
Let $\rho\in J_2(S,\kappa)$.
Then, in view of~\eqref{now3},
$$\rho|_{\Xb_{\bfu}\times \Xb_{\bfu}}=\ov{\kappa}\ot\ov{\kappa}\text{ if and only if }
\rho|_{(X_{\bfu}\times \Xb_{\bfu})}=\kappa\ot\ov{\kappa}$$
and the latter holds if and only if
for all ``monomials'' $P=\prod_{j=1}^k\pi_0^{p_j}\circ S^{r_j}$ and $Q=\prod_{j=1}^\ell\pi_0^{q_j}\circ S^{s_j}$, we have
\begin{multline*}
\int_{X_{\bfu}\times X_{\bfu}} P(x)\,\EE_\kappa[Q\,|\,\Xb_{\bfu}](x')\,d\rho(x,x')= \\
\int_{X_{\bfu}}P\,d\kappa\int_{\Xb_{\bfu}}
\EE_\kappa[Q\,|\,\Xb_{\bfu}](\xb')\,d\ov{\kappa}(\xb')=
\int_{X_{\bfu}}  P\,d\kappa\int_{X_{\bfu}} Q\,d\kappa.
\end{multline*}
Now, in $L^2(X_{\bfu},\kappa)$, we have
$$
\EE_\kappa[Q\,|\,\Xb_{\bfu}]=
\lim_{N\to\infty}\frac1N\sum_{n\leq N}Q\circ S^n.$$
In this equality, we can  replace $L^2(\kappa)$ with $L^2(\rho)$ which yields
$$
\int_{X_{\bfu}\times X_{\bfu}} P(x)\,\EE_\kappa[Q\,|\,\Xb_{\bfu}](x')\,d\rho(x,x')=
\lim_{N\to\infty}\frac1N\sum_{n\leq N}\int_{X_{\bfu}\times X_{\bfu}}P\ot Q\circ S^n\,d\rho.$$
Collecting the remarks above, we have proved the following result:

\begin{Prop}\label{p:orzucie} Under the notation above, let $\rho\in J_2(S,\kappa)$. Then $\rho|_{\Xb_{\bfu}\times \Xb_{\bfu}}=\ov{\kappa}\ot\ov{\kappa}$ if and only if  we have
\beq\label{comput1}
\lim_{N\to\infty}\frac1N\sum_{n\leq N}\int_{X_{\bfu}\times X_{\bfu}}P\ot Q\circ S^n\,d\rho=
\int_{X_{\bfu}}  P\,d\kappa\int_{X_{\bfu}} Q\,d\kappa\eeq
for all monomials $P,Q$.
\end{Prop}

An alternative arises by using, instead of ``monomials'', the characteristic functions $\raz_{[B]_0}$: $B\in A^k$, where $A$
is the (finite) set of values  of $\bfu$ and $k\geq1$ (here $[B]_s=\{x\in X_{\bfu}\colon \bfu[s,s+k-1]=B\}$. As shifts of such functions yield a linearly dense subset in $L^2(\kappa)$, by repeating all arguments that led to Proposition~\ref{p:orzucie}, we obtain the following:
\begin{Cor}\label{c:orzucie} Under the notation above, let $\rho\in J_2(S,\kappa)$. Then $\rho|_{\Xb_{\bfu}\times \Xb_{\bfu}}=\ov{\kappa}\ot\ov{\kappa}$ if and only if  we have
\beq\label{comput1a}
\lim_{N\to\infty}\frac1N\sum_{n\leq N}\rho([B]_0\times [C]_{-n})=\kappa([B]_0)\kappa([C])\eeq
for all blocks $B\in A^k,C\in A^\ell$ with $k,\ell\geq1$.
\end{Cor}
\begin{Remark} An equivalent form of this corollary is that: if $\Phi_\rho$ stands for the Markov operator corresponding to $\rho$ then $\Phi_\rho|_{L^2(\ci_S)}=0$ if and only if $\Phi_\rho\circ \frac1N\sum_{n\leq N}S^n\to 0$ weakly in $L_0^2(\kappa)$. Indeed, $\frac1N\sum_{n\leq N}S^n\to \EE_\kappa[\cdot\,|\,\ci_S]$ (in $L^2(\kappa)$) by the von Neumann theorem.
\end{Remark}
Note that the quantities we have: the LHS terms and the RHS term in~\eqref{comput1} and in~\eqref{comput1a} are computable through the generic sequence $(S^n\bfu,S^{\phi_k(n)}\bfu)$ along $(N_k)$, see Proposition~\ref{prop:selfjoiningsofFS}, and
the generic point $\bfu$ along $(N_k)$, respectively.

\subsection{Summing up} Our aim is to describe those $\bfu$ which are orthogonal to all u.e.\ sequences (we assume that $\bfu$ has zero mean on typical short interval). For that, for all Furstenberg systems $\kappa\in V(\bfu)$ we need to check the assertion of Corollary~\ref{c:relvN2}. In this corollary, we need to deal with some self-joinings $\la$ (of $\kappa$). In fact, all of them are described combinatorially, using only $\bfu$, see Proposition~\ref{prop:selfjoiningsofFS}. We need to check the assertion of Corollary~\ref{c:relvN2}, which by Remark~\ref{r:Bosher} (see $(\ast)$ there) is
$$
\lim_{L\to\infty}\int_{X_{\bfu}\times X_{\bfu}}\Big|\frac1L\sum_{\ell\leq L}(\pi_0\ot\pi_0)(S\times S)^\ell\Big|^2d\la=0.$$
That is, given $\vep>0$, for $L>L_0$, we want to see
$$\int_{X_{\bfu}\times X_{\bfu}}\Big|\frac1L\sum_{\ell\leq L}(\pi_0\ot\pi_0)(S\times S)^\ell\Big|^2d\la<\vep,$$ where the integral is computable  along a subsequence $(N_k)$ (in fact, a subsequence of it), so (using  Proposition~\ref{prop:selfjoiningsofFS}), we need
$$\limsup_{k\to\infty}\frac1{N_k}\sum_{n\leq N_k}\Big|\frac1L\sum_{\ell\leq L}(\pi_0\ot\pi_0)\circ(S\times S)^\ell(S^n\bfu,S^{\phi_k(n)}\bfu)\Big|^2=$$$$
\limsup_{k\to\infty}\frac1{N_k}\sum_{n\leq N_k}\Big|\frac1L\sum_{\ell\leq L}\bfu(n+\ell)\bfu(\phi_k(n)+\ell)\Big|^2<\vep,$$
which is, from the combinatorial point of view,  a certain condition on the behaviour on short intervals. However, it is clear that if we consider all self-joinings, then,  the above condition is not satisfied (take the diagonal  self-joining). The key here is that we only consider those $\la$ which satisfy  \eqref{comput1a} (or \eqref{comput1}).

\Large
\begin{center} Part II
\end{center}

\normalsize
\section{Characteristic classes and orthogonality to uniquely ergodic systems}
\subsection{Characteristic classes and disjointness} \label{s:ccd}
In what follows, we consider $\cf$ a characteristic class.
We recall that each non-trivial characteristic class contains the class ID of all identities \cite{Ka-Ku-Le-Ru}.

\begin{Prop}\label{p:eqdis} Let ${T}\in {\rm Aut}({X},{\mu})$ and let $\cf$ be a characteristic class. Let $\af$ be the largest $\cf$-factor of ${T}$. Then the following conditions are equivalent:\\
(i) ${T}\perp \cf\cap{\rm Erg}$,\\
(ii) ${T}|_{\af}\perp \cf\cap{\rm Erg}$,\\
(iii) ${T}|_{\af}\perp {\rm Erg}$.\end{Prop}
\begin{proof}
(ii) $\Rightarrow$ (i)   Take any $R$ ergodic in $\cf$. Then ${T}|_{\af}$ is disjoint with $R$ by (ii), and this disjointness lifts to $ T$ by Corollary~\ref{c:rozchar}, as $R\in\cf$.

(ii) $\Rightarrow$ (iii)   Take any $R$ ergodic.  Then ${T}|_{\af}$ is disjoint with the largest $\cf$-factor of $R$ (by (ii)), and then this disjointness lifts to $R$ by Corollary~\ref{c:rozchar}, as ${T}|_{\af}\in\cf$.

The other implications are straightforward.
\end{proof}

Remembering that $\ci_{{T}}\subset\af \mod{\mu}$ and using Proposition~\ref{p:eqdis} together with Theorem~\ref{t:duzoroz}, we obtain the following.

\begin{Cor}\label{c:charconf} ${T}\perp \cf\cap{\rm Erg}$ if and only if
${T}|_{\af}$ is a confined extension of the sigma-algebra of invariant sets.\end{Cor}

Before, we formulate a result which resembles more Theorem~\ref{t:duzoroz} than the above corollary, let us consider two examples.

\begin{Example}\label{e:przyk1}
Let $\cf$=DISP be the (characteristic) class of all automorphisms with discrete spectrum. Let $T:\T^2\to\T^2$, $T(x,y)=(x,y+x)$ (considered with Leb$_{\T^2}$). It is not hard to see that $\af $ is precisely $\ci_T$ (that is, the sigma-algebra of the first coordinate). Indeed, note  that on $L^2(\T^2)\ominus L^2(\T\ot\{\emptyset,\T\})$ the spectrum is purely Lebesgue. We also have $T\in {\rm Erg}^\perp$. Furthermore, for any automorphism from Erg$^\perp$ its only eigenvalue is~1, so its Kronecker factor is {\bf always} the sigma-algebra of invariant sets, while for the fiber automorphisms we can have discrete spectrum. On the other hand, for our $T$, on a.a.\ fibers (ergodic components) the spectrum is discrete, so on a.a.\ fibers the largest $\cf$-factor is the whole space. In other words, the restriction of the factor $\af $ to fibers {\bf does not} give the (largest) $\cf$-factors on the fibers.\end{Example}

\begin{Example}\label{e:przyk2}
Let $\cf=$ZE be the  class of  automorphisms with zero entropy. According to \cite{Ka-Ku-Le-Ru}, this is the largest (proper) characteristic class. Assume that ${T}$  is of the form ${T}(\xb,u)=(\xb, T_{\xb}u)$ acting on
$X=\bigsqcup_{n\geq1}\Xb_n\times\{1,\ldots,n\}\sqcup \Xb_\infty\times Y$, where $\mu|_{\Xb_n\times\{1,\ldots,n\}}=\mub|_{\Xb_n}\ot\nu_n$ and
$\mu|_{\Xb_\infty\times Y}=\mub|_{\Xb_\infty}\ot\nu$
(note that for each $2\leq n\in\N$, $T|_{X_n}$ is not disjoint with $R_n$ which is ergodic with zero entropy).
In Proposition~\ref{prop:Pinsker} below, we will show that contrary to the phenomenon in the previous example, here, the largest $\cf$-factor (the Pinsker sigma-algebra) of ${T}$ is also the largest $\cf$-factor for a.a.\ fiber automorphisms. By applying  Theorem~\ref{t:duzoroz} (the equivalence of (ii) and (iii)), we hence obtain the following result.

\begin{Cor}\label{c:duzorozPinsker} The following conditions are equivalent:\\
(i)
${T}\perp {\rm Erg}\cap{\rm ZE}$.\\
(ii) The extension ${T}|_{\Pi({T})}\to \textcolor{DarkViolet}\Id_{\Xb}$ is confined.\\
(iii) For $\mub\ot\mub$-a.a.\ $(\xb,\xb')\in \Xb\times \Xb$, we have $\Pi(T_{\xb})\perp \Pi(T_{\xb'})$.\end{Cor}
\end{Example}

In particular, if the fiber automorphisms $T_{\xb}$ are Kolmogorov, then ${T}\perp {\rm Erg}\cap{\rm ZE}$.

\begin{Example}\label{e:przyk2a} Let $\Xb=(0,1]$, $Y=SL_2(\R)/SL_2(\Z)$ (considered with the corresponding Haar measure) and ${T}(t,x\Gamma):=(t,g_tx\Gamma)$, where $\Gamma=SL_2(\Z)$,  $g_t:=\left[\begin{array}{cc}e^{-t}&0\\0&e^t\end{array}\right]$. Then ${T}\perp {\rm Erg}\cap{\rm ZE}$ while ${T}$ is not disjoint with Erg (because the fiber automorphisms are not a.a.\ disjoint; $T_t$ is Bernoulli with entropy $t$) neither with ZE (its Pinsker factor equal to the sigma-algebra of invariant sets is non-trivial).\end{Example}

On the other hand, note that ${T}$ in Example~\ref{e:przyk2a} has no non-trivial ergodic factor as each such would have to be a factor of a.a.\ ergodic components which is in conflict with entropy on the fibers.

\begin{Remark} Two further classes: ID and NIL$_1$~(as proved in \cite{Ka-Ku-Le-Ru}, the latter class consists of automorphisms whose a.a.\ ergodic components have discrete spectrum) behave similarly to ZE. For the ID class simply the trace of the sigma-algebra of invariant sets is the trivial sigma-algebra on each fiber. For the NIL$_1$ class the largest $\cf$-factor is the relative Kronecker over the sigma-algebra of invariant sets. It is proved in \cite{Ka-Ku-Le-Ru} that this factor comes exactly from the Kronecker factors on the fibers and, in the ergodic case, the relative Kronecker factor over the sigma-algebra of invariant sets is exactly the Kronecker factor.  \end{Remark}

On the base of the above, we have:\\
\underline{Conjecture:}
Corollary~\ref{c:duzorozPinsker} holds for each characteristic class $\cf$ satisfying $\cf=\cf_{\rm ec}$, that is, almost every ergodic component of a member in $\cf$ also belongs to $\cf$. (Note that the last three classes satisfy $\cf=\cf_{\rm ec}$.)

\subsection{Orthogonality to zero entropy and uniquely ergodic systems. General characteristic class case}
We now merge Boshernitzan's problem with the approach of \cite{Ka-Ku-Le-Ru} to characterize $\bfu$ orthogonal to the systems whose invariant measures determine systems belonging to a fixed characteristic class. 
As noticed in Proposition~\ref{ec99}, we do not need to resort to the so-called ec-classes in the setting of Boshernitzan's problem.

\begin{Prop} \label{p:ouesZEgen}
Let $\bfu:\N\to\D$. Let $\cf$ be any non-trivial characteristic class. Then $\bfu\perp$ uniquely ergodic systems in $\cf$ if and only if for each Furstenberg system $\kappa\in V(\bfu)$,
\beq\label{char17}
\EE_\kappa[\pi_0\,|\,\ca_\cf(X_{\bfu},\kappa,S)] \perp F_{\rm we}(X_{\bfu},\kappa,S).\eeq
\end{Prop}

\begin{proof} $\Rightarrow$ Assume that $\bfu\perp \rm{UE}\cap\cf$, and suppose that for some $\kappa\in V(\bfu)$ there are an ergodic system $(Z',\nu',R')$, a joining $\rho'\in J(\nu',\kappa)$ and a function $g\in L^2_0(\nu')$ for which 
$$\EE_\kappa[\pi_0\,|\,\af(X_{\bfu},\kappa,S)] \not\perp  \Phi_{\rho'}(g). $$
Since the function on the left-hand side is measurable with respect to a factor in the class $\cf$, we can replace above $g$ by $\EE_{\nu'}[g\,|\,\af(Z',\nu',R')]$.
Thus, replacing if necessary $(Z',\nu',R')$ by its largest $\cf$-factor, we can assume with no loss of generality that $R'\in\cf$. But once we know that, we have
$$ 0\neq \EE_\kappa\bigl[ \EE_\kappa[\pi_0\,|\,\af(X_{\bfu},\kappa,S)] \, \Phi_{\rho'}(g) \bigr] = \EE_\kappa\bigl[ \pi_0 \, \Phi_{\rho'}(g) \bigr]. $$
Then we proceed as in the proof of Proposition~\ref{p:oues}: Using the Jewett-Krieger theorem, we obtain that there exists a uniquely ergodic system (belonging to $\cf$) $(Z,\nu,R)$ and a joining $\rho\in J(\nu,\kappa)$ such that $\pi_0\not\perp {\rm Im}(\Phi_\rho)$ (in $L^2(\kappa)$).  It follows that there exists $f\in C(Z)$ such that $\int  \Phi_\rho(f)\,\pi_0\,d\kappa\neq0$, in other words
$$\int_{Z\times X_{\bfu}}f\ot \pi_0\,d\rho\neq0.$$
Let $(N_m)$ satisfy $\frac1{N_m}\sum_{n\leq N_m}\delta_{S^n\bfu}\to\kappa$. In view of Theorem~\ref{t:lifting}, there exist a sequence $(z_n)\subset Z$ and a subsequence $(N_{m_\ell})$ such that
$$
\frac1{N_{m_\ell}}\sum_{n\leq N_{m_\ell}}\delta_{(z_n, S^n\bfu)}\to \rho$$
and the set $\{n\geq1:\: z_{n}\neq Rz_{n-1}\}$ is contained in a subset of $\NN$ of the form $\{b_1<b_2<\ldots\}$, where
$b_{k+1}-b_k\to\infty$. Adding if necessary some more $b_k$'s to this set, we may also assume that
$\lim_{k\to\infty} \frac{b_{k+1}}{b_k} = 1$. In this way, defining $K_\ell:=\max\{k:b_k\le N_{m_\ell}\}$, we get $\lim_{\ell\to\infty} \frac{b_{K_\ell}}{N_{m_\ell}} = 1$, and it follows that
\begin{multline*}
0\neq \int f\ot \pi_0\,d\rho=\lim_{\ell\to\infty} \frac1{N_{m_\ell}}\sum_{n\leq N_{m_\ell}}f(z_n)\bfu(n)\ = \\
\lim_{\ell\to\infty}\frac1{b_{K_\ell}}\sum_{k<K_\ell}\Big(\sum_{b_k\leq n<b_{k+1}}f(R^{n-b_k}z_{b_k})\bfu(n)\Big).
\end{multline*}

However, from our assumption of orthogonality on $\bfu$, the latter limit is~$0$ because the sequence $\bigl(f(z_n)\bigr)$ can be observed in the orbital system described in Theorem~\ref{t:orbmod}, which is uniquely ergodic and in $\cf$.

$\Leftarrow$  Let $(X,T)$ be uniquely ergodic and in $\cf$. Suppose that, for some $f\in C(X)$, some $x\in X$ and some increasing sequence $(N_k)$    we have the existence of the limit
$$ c:=\lim_{k\to\infty}\frac1{N_k}\sum_{n\leq N_k}f(T^nx)\bfu(n). $$
Extracting a subsequence if necessary, we can also assume  that
$$\frac1{N_k}\sum_{n\leq N_k}\delta_{(T^nx,S^n\bfu)}\to\rho.$$
Then $\rho$ is a joining, $\rho\in J(\nu,\kappa)$, where $\nu$ is the unique invariant measure for $T$, and $\kappa\in V(\bfu)$.
Now, since $(X,\nu,T)\in\cf$, we get by Theorem~\ref{t:largest}
$$ c=\int f\ot \pi_0\,d\rho=\int f\ot \EE_\kappa[\pi_0\,|\,\af(X_{\bfu},\kappa,S)]\,d\rho.$$
So, if we assume that
$$\EE_\kappa[\pi_0\,|\,\af(X_{\bfu},\kappa,S)] \perp F_{\rm we}(X_{\bfu},\kappa,S),$$
we get $c=0$.
\end{proof}

\begin{Remark}The above result also holds in the logarithmic case. Now, by Frantzikinakis-Host's theorem \cite{Fr-Ho}, all zero entropy, uniquely ergodic systems are M\"obius orthogonal. By the above Proposition~\ref{p:ouesZEgen}, the relevant Veech condition is satisfied for $\mob$. However, to read a combinatorial reformulation of Frantzikinakis-Host's theorem is unclear because we have been unable to get a logarithmic version of Proposition~\ref{prop:selfjoiningsofFS} (see Section~\ref{s:question}).
\end{Remark}

Note that if for $\cf$ we take the class of all measure-preserving systems, then we obtain the original Boshernitzan's problem, and we recover the result stated in Proposition~\ref{p:oues}.

Proposition~\ref{p:ouesZEgen} allows us to use directly Theorem~\ref{t:OrMaIm} and Corollary~\ref{c:relvN2} (with ${T}$ replaced by any Furstenberg system of $\bfu$).

\begin{Cor}\label{c:enfin} The following conditions are equivalent:\\
(a) $\bfu \perp \cf\cap {\rm UE}$.\\
(b) For each Furstenberg system $\kappa$ of $\bfu$ the following holds: for each $\la\in J^{\rm RelErg}_2(X_{\bfu},\kappa,S)$ (in particular, $\la|_{\ci_S\ot\ci_S}=\kappa|_{\ci_S}\ot\kappa|_{\ci_S}$), we have
\beq\label{warunekM}
\EE_\lambda\Bigl[\EE_\kappa[\pi_0\,|\,\af]\ot \EE_\kappa[\pi_0\,|\,\af] \, |\, \ci_S\ot\ci_S\Bigr]=
\EE_\kappa[\pi_0\,|\,\ci_S] \ot
\EE_\kappa[\pi_0\,|\,\ci_S].\eeq
(c) For each Furstenberg system $\kappa$ of $\bfu$ the following holds: for each $\la\in J_2(X_{\bfu},\kappa,S)$ with $\la|_{\ci_S\ot\ci_S}=\kappa|_{\ci_S}\ot\kappa|_{\ci_S}$, we have
\beq\label{warunekMbis}\EE_\lambda\Bigl[\EE_\kappa[\pi_0\,|\,\af]\ot \EE_\kappa[\pi_0\,|\,\af] \, |\, \ci_{S\times S}\Bigr]=
\EE_\kappa[\pi_0\,|\,\ci_S] \ot
\EE_\kappa[\pi_0\,|\,\ci_S].\eeq
\end{Cor}

\subsection{Orthogonality to systems with countably many ergodic invariant measures, all in $\cf$}\label{s:cmec}

The purpose of this section is to prove the following generalization of Proposition~\ref{p:thi2}.

\begin{Prop}\label{p:thi3} Let $\bfu$, $\|\bfu\|_{u^1}=0$, be a bounded arithmetic function orthogonal to all u.e.\ systems in some characteristic class $\cf$. Then $\bfu$ is also orthogonal to all systems whose set of ergodic invariant measures is countable, all of them giving rise to systems in $\cf$. More precisely: if $(Z,R)$ is a topological system whose set $\{\rho_i\colon i\in I\}$ of ergodic measures is countable, and satisfy for all $i$, $(Z,\rho_i,R)\in\cf$, then for any $z_0\in Z$ and any function $g\in C(Z)$,
we have
$$
\frac1N \sum_{n<N} \bfu(n) g(R^nz_0)\to 0.
$$
\end{Prop}

The proof of the above proposition follows essentially the same lines as the proof of Proposition~\ref{p:thi2}, but with a further refinement to take into account the underlying characteristic class $\cf$. This is what the lemmas below are made for.
The framework of these lemmas is quite the same as in Corollary~\ref{c:thi1}: we have a measure-preserving system $(X,\mu,T)$, and a $T$-invariant probability measure $\eta$ such that $\eta\ll \mu$, and $\phi:=d\eta/d\mu$ is assumed to be bounded. We identify elements of $L^\infty(\eta)$ with elements of $L^\infty(\mu)$ which vanish on $\{x:\phi(x)=0\}$. We constantly use this identification in the following lemmas.

\begin{Lemma}
\label{lemma:CC1}
 In the framework described above, let $f\in L^\infty(\mu)$ satisfy $f=0$ on $\{x:\phi(x)=0\}$.
 Then, we have the equivalence:
 $$ f\text{ is }\af(X,\mu,T)\text{-measurable} \Longleftrightarrow f\text{ is }\af(X,\eta,T)\text{-measurable}.
 $$
\end{Lemma}

\begin{proof}
 Let $0<\alpha<1$ be such that $\alpha\phi\le1$, so that $\alpha \eta \le \mu$. For clarity, we will use the notation $X_\mu$ (respectively, $X_\eta$) for the space $X$ endowed with the measure $\mu$ (respectively the measure $\eta$). We first construct a joining $\lambda$ of $(X_\mu,\mu,T)$ and $(X_\eta,\eta,T)$ by setting, for all measurable subsets $A\subset X_\mu$, $B\subset X_\eta$,
 \begin{align*}
  \lambda(A\times B) &:= \int_A \Bigl(\alpha\phi \ind{B} + \bigl(1-\alpha\phi\bigr)\eta(B)\Bigr)\,d\mu \\
  &= \alpha \eta(A\cap B) + \bigl(\mu(A)-\alpha\eta(A)\bigr)\eta(B)
 \end{align*}
 (note that the associated Markov operator $\Phi_\lambda:L^2(X,\eta)\to L^2(X,\mu)$ is given by $\Phi_\lambda(f)=\alpha\phi f+(1-\alpha\phi)\int f\,d\eta$).
 In particular, the conditional distribution of the second coordinate $x_\eta$ given the first one $x_\mu$ is
 $$ \alpha\phi(x_\mu)\delta_{x_\mu}+\bigl(1-\alpha\phi(x_\mu)\bigr)\eta.$$
 We then construct the joining $\lambda_\infty^{\eta}$ as the relative product of infinitely many copies of $(X_\mu\times X_\eta, \lambda)$ over $(X_\mu,\mu)$:
 this is the $T\times T^{\NN}$- invariant probability measure on $X_\mu\times X_\eta^{\NN}$ whose marginal on $X_\mu$ is $\mu$, and whose conditional distribution on $X_\eta^\NN$ given $x_\mu$ is the infinite product measure
 $$ \Bigl(\alpha\phi(x_\mu)\delta_{x_\mu}+\bigl(1-\alpha\phi(x_\mu)\bigr)\eta\Bigr)^{\otimes \NN}.$$
 By applying the law of large numbers in each fiber determined by $x_\mu$, we get that for $\lambda_\infty^{\eta}$-almost all $\bigl(x_\mu,(x_\eta^k)_{k\in\NN}\bigr)$, the limit
 $$ \ell=\ell\bigl(x_\mu,(x_\eta^k)_{k\in\NN}\bigr) := \lim_{K\to\infty}\frac{1}{K}\sum_{k=1}^K f(x_\eta^k)  $$
 exists, and satisfy
 $$ \ell = \alpha\phi(x_\mu)f(x_\mu)+ \bigl(1-\alpha\phi(x_\mu)\bigr)\int f\, d\eta.$$
It follows that, $\lambda_\infty^{\eta}$-almost surely, we can get the value $f(x_\mu)$ by the formula
$$ f(x_\mu) = \ind{\phi(x_\mu)>0} \dfrac{\ell-\bigl(1-\alpha\phi(x_\mu)\bigr)\int f\, d\eta }{\alpha\phi(x_\mu)}. $$
Now, if $f$ is $\af(X_\eta,\eta,T)$-measurable, 
the limit $\ell$ is measurable with respect to an infinite joining of $\cf$-systems (namely the infinite self-joining of $\af(X_\eta,\eta,T)$ arising from $\lambda_\infty^{\eta}$). Moreover, as $\phi$ is $T$-invariant, and since the class $\cf$ must contain the ID class, $\bigl(x_\mu,(x_\eta^k)_{k\in\NN}\bigr)\mapsto \phi(x_\mu)$ is measurable with respect to $\af(X_\mu\times X_\eta^{\NN}, \lambda_\infty^{\eta}, T\times T^{\NN})$. Therefore, the same holds for $\bigl(x_\mu,(x_\eta^k)_{k\in\NN}\bigr)\mapsto f(x_\mu)$, and this proves that $f$ is $\af(X,\mu,T)$-measurable.
\smallskip

Conversely, assume that $f$ is $\af(X,\mu,T)$-measurable. We will use the same joining $\lambda$ as before, but now we disintegrate it with respect to $x_\eta$: we have for all measurable subsets $A\subset X_\mu$, $B\subset X_\eta$,
\begin{align*}
 & \lambda(A\times B) \\
 & = \int_{X_\mu} \ind{A}(x_\mu) \left( \int_{X_\eta} \ind{B}(x_\eta)
         \left(  \alpha \phi(x_\mu)\, d\delta_{x_\mu}(x_\eta) + \bigl(1-\alpha \phi(x_\mu)\bigr) d\eta(x_\eta)  \right)\right)d\mu(x_\mu)  \\
 & = \int_{X_\eta} \ind{B}(x_\eta) \int_{X_\mu}   \ind{A}(x_\mu)
         \left( \alpha \, d\delta_{x_\eta}(x_\mu) + (1-\alpha)\frac{1-\alpha\phi(x_\mu)}{1-\alpha}  d\mu(x_\mu)\right) d\eta(x_\eta). \\
\end{align*}
Thus, under $\lambda$, the conditional distribution of the first coordinate $x_\mu$ given the second one $x_\eta$ is
$$ \alpha\delta_{x_\eta}+(1-\alpha) \frac{\mu-\alpha\eta}{1-\alpha}.$$
Now, we construct  the joining $\lambda_\infty^{\mu}$ as the relative product of infinitely many copies of $(X_\mu\times X_\eta, \lambda)$ over $(X_\eta,\eta)$:
 this is the $T\times T^{\NN}$- invariant probability measure on $X_\eta\times X_\mu^{\NN}$ whose marginal on $X_\eta$ is $\eta$, and whose conditional distribution on $X_\mu^\NN$ given $x_\eta$ is the infinite product measure
 $$ \Bigl(\alpha\delta_{x_\eta}+(1-\alpha) \frac{\mu-\alpha\eta}{1-\alpha}\Bigr)^{\otimes \NN}.$$
 Again, we apply the law of large number in each fiber determined by $x_\eta$. We get that for $\lambda_\infty^{\mu}$-almost all $\bigl(x_\eta,(x_\mu^k)_{k\in\NN}\bigr)$, the limit
 $$ \bar\ell=\bar\ell\bigl(x_\eta,(x_\mu^k)_{k\in\NN}\bigr) := \lim_{K\to\infty}\frac{1}{K}\sum_{k=1}^K f(x_\mu^k)  $$
 exists, and satisfy
 $$ \bar\ell = \alpha f(x_\eta)+ \int f\,d\mu - \alpha \int f\,d\eta.$$
 It follows that, $\lambda_\infty^{\mu}$-almost surely, we can get the value $f(x_\eta)$ by the formula
$$ f(x_\eta) = \frac{1}{\alpha}\left(\bar\ell- \int f\,d\mu + \alpha \int f\,d\eta\right). $$
But, since we assume here that $f$ is  $\af(X,\mu,T)$-measurable, the limit $\bar\ell$ is measurable with respect to an infinite joining of $\cf$-systems. Therefore, the function $\bigl(x_\eta,(x_\mu^k)_{k\in\NN}\bigr)\mapsto f(x_\eta)$ is measurable with respect to $\af(X_\eta\times X_\mu^{\NN}, \lambda_\infty^{\mu}, T\times T^{\NN})$, and we conclude that $f$ is $\af(X,\eta,T)$-measurable.
\end{proof}

\begin{Lemma}
\label{lemma:CC2}
 We keep the same framework as in Lemma~\ref{lemma:CC1}. For all $f\in L^\infty(\mu)$, we have
 $$
 \EE_\eta[f\,|\,\af(X,\eta,T)] = \ind{\phi>0} \EE_\mu[f\,|\,\af(X,\mu,T)] = \EE_\mu[\ind{\phi>0} f\,|\,\af(X,\mu,T)].
 $$
\end{Lemma}
\begin{proof}
 The equality
 $$
 \ind{\phi>0} \EE_\mu[f\,|\,\af(X,\mu,T)] = \EE_\mu[\ind{\phi>0} f\,|\,\af(X,\mu,T)]
 $$
 is a straightforward consequence of the fact that $\phi$, being $T$-invariant, is $\af(X,\mu,T)$-measurable. It remains to prove that
 $$
 \EE_\eta[f\,|\,\af(X,\eta,T)] = \EE_\mu[\ind{\phi>0} f\,|\,\af(X,\mu,T)].
 $$
 The right-hand side is a bounded, $\af(X,\mu,T)$-measurable function which vanishes on the set $(\phi=0)$, so, by applying Lemma~\ref{lemma:CC1}, it is also $\af(X,\eta,T)$-measurable. So, we just have to check that for a given bounded $\af(X,\eta,T)$-measurable function g, we have
 \begin{equation}
  \label{eq:reste}
  \EE_\eta[gf] = \EE_\eta\Bigl[ g\, \EE_\mu[\ind{\phi>0} f\,|\,\af(X,\mu,T)] \Bigr].
 \end{equation}
 But this can be done through the following chain of equalities:
 \begin{align*}
  \EE_\eta[gf] & = \EE_\mu[\phi g f] \quad\text{(by the definition of $\phi$)} \\
  & =  \EE_\mu[\phi g\, \ind{\phi>0} f] \\
  & = \EE_\mu\Bigl[ \phi g\, \EE_\mu[\ind{\phi>0} f\,|\,\af(X,\mu,T)] \Bigr]  \\
  & = \EE_\eta\Bigl[ g\, \EE_\mu[\ind{\phi>0} f\,|\,\af(X,\mu,T)] \Bigr].
 \end{align*}
The justification of the third equality comes from Lemma~\ref{lemma:CC1} applied to $\phi g$, which proves that this function is $\af(X,\mu,T)$-measurable.
\end{proof}

\begin{proof}[Proof of Proposition~\ref{p:thi3}]
 We repeat verbatim the proof of Proposition~\ref{p:thi2} to define $\kappa$, $\rho$, $(\alpha_i)_{i\in I}$, $(\kappa_i)_{i\in I}$ and $(\mu_i)_{i\in I}$. Our assumption now is that $\bfu\perp \text{UE}\cap\cf$, so, by Proposition~\ref{p:ouesZEgen}, we get
 $$\EE_\mu[\pi_0\,|\,\ca_\cf(X_{\bfu},\mu,S)] \perp F_{\rm we}(X_{\bfu},\mu,S)$$
 for all Furstenberg system $\mu\in V(\bfu)$. It follows from Remark~\ref{r:module} that if we multiply $\pi_0$ by a bounded $T$-invariant function, then the product remains orthogonal to  $F_{\rm we}(X_{\bfu},\mu,S)$.  Therefore, setting for all $i\in I$, $\phi_i:=\frac{d\mu_i}{d\mu}$, we still have
 $$\ind{\phi_i>0}\EE_\mu[\pi_0\,|\,\ca_\cf(X_{\bfu},\mu,S)] \perp F_{\rm we}(X_{\bfu},\mu,S).$$
 We can apply Corollary~\ref{c:thi1} to $\ind{\phi_i>0}\EE_\mu[\pi_0\,|\,\ca_\cf(X_{\bfu},\mu,S)]$, which yields
 $$ \ind{\phi_i>0}\EE_\mu[\pi_0\,|\,\ca_\cf(X_{\bfu},\mu,S)]\perp F_{\rm we}(X_{\bfu},\mu_i,S), $$
 and then Lemma~\ref{lemma:CC2} gives
 $$ \EE_{\mu_i}[\pi_0\,|\,\ca_\cf(X_{\bfu},\mu_i,S)]\perp F_{\rm we}(X_{\bfu},\mu_i,S). $$
 Since $\kappa_i$ is a joining between $(X_{\bfu},\mu_i,S)$ and the ergodic system $(Z,\rho_i,R)$ in $\cf$, the above property ensures by Theorem~\ref{t:largest} that
 $$ \int_{X_{\bfu} \times Z} \pi_0\otimes g \, d\kappa_i = 0. $$
 Then we can conclude as in the proof of Proposition~\ref{p:thi2}.
\end{proof}

As a consequence of Lemma~\ref{lemma:CC1}, we are also ready to prove Proposition~\ref{p:countableCC}.

\noindent
\begin{proof}[Proof of Proposition~\ref{p:countableCC}]
We assume that $\mu=\sum_{i=1}^\infty\alpha_i\mu_i$, where $\mu_i$ are $T$-invariant, ergodic and mutually singular. For all $i\geq1$, set $\phi_i:=\frac{d\mu_i}{d\mu}$. We can write any $f\in L^\infty(X,\mu_i)$ as
$$ f = \sum_{i\geq 1}  f\, \ind{\phi_i>0}. $$
Since $\phi_i$ is $T$-invariant, $\ind{\phi_i>0}$ is measurable with respect to $\af(X,\mu,T)$. Moreover, we can apply Lemma~\ref{lemma:CC1} to the function $f\, \ind{\phi_i>0}$: this function is measurable with respect to $\af(X,\mu,T)$ if and only if it is measurable with respect to $\af(X,\mu_i,T)$. It follows that the following conditions are equivalent:
\begin{itemize}
 \item $f$ is $\af(X,\mu,T)$-measurable,
 \item for each $i$, $f\, \ind{\phi_i>0}$ is $\af(X,\mu,T)$-measurable,
 \item for each $i$, $f\, \ind{\phi_i>0}$ is $\af(X,\mu_i,T)$-measurable.
\end{itemize}
The equivalence in the statement of the proposition follows easily.
\end{proof}

\subsection{Examples}
\subsubsection{The ID class}
\begin{Example} Consider the class ID of all identities. Then $\EE_\kappa[ \pi_0\, |\, \af ]$ is $\af $-measurable, where $\af =\ci_S$, and
$\EE_\kappa[ \pi_0\, |\, \af ]\ot \EE_\kappa[ \pi_0\, |\, \af ]$ is $\ci_S\ot\ci_S$-measurable, so \eqref{warunekM} holds. It follows the assertion~(i) of Corollary~\ref{c:enfin} is satisfied. Hence the only restriction on $\pi_0$ (and $\kappa$) is \eqref{calkazero}, i.e.\ the mean of $\bfu$ equals zero.  Notice that since the only ergodic identity is the one-point system, this fits to the obvious condition of zero mean  of $\bfu$ as those arithmetic functions being orthogonal to all uniquely ergodic identities.\end{Example}

\subsubsection{Discrete spectrum case}
We will need the following observation ($R_\alpha (z)=ze^{2\pi i\alpha}$ stands for the irrational rotation by $\alpha$ on $\bs^1$):

\begin{Lemma}\label{l:prince1}
Assume that ${T}\in{\rm Aut}\xbm$.
Assume also that $0\neq F\in L^2(X,\mu)$ satisfies $F\circ {T}=e^{2\pi i\alpha}F$ (a.e.), where $\alpha\in[0,1)$ is irrational.  Then there exist $g\in L^2(\bs^1, {\rm Leb})$, $g\circ R_\alpha=e^{2\pi i\alpha}g$, and $\rho\in J({T},R_\alpha)$ such that
$\int_{X\times\bs^1}F(x)g(z)\,d\rho(x,z)\neq0$.
\end{Lemma}
\begin{proof} Because $\alpha$ is irrational, $F|_{\bigsqcup_{n\in\N}X_n}=0$, therefore we can assume w.l.o.g.\ that $T$ is aperiodic, i.e.\ $T$
is an automorphism of $({\Xb}\times Y,\mub\ot\nu)$, where
${T}({\xb},y)=({\xb},T_{\xb}y)$ with ${\xb}\mapsto T_{\xb}\in {\rm Aut}(Y,\nu)$ being the ergodic decomposition of ${T}$.

Let $A:=\{{\xb}\in {\Xb}\colon F({\xb},\cdot)\neq0\;\nu-\text{a.e.}\}$. Then $\mub(A)>0$.
Note that for ${\xb}\in A$, $F({\xb},\cdot)$ is an eigenfunction for $T_{\xb}$ corresponding to the eigenvalue $e^{2\pi i\alpha}$. By ergodicity, it follows that $|F({\xb},\cdot)|=:\xi({\xb})>0$   (for a.a.\ ${\xb}\in A$; outside of $A$, $F$ vanishes).

Let $G:A\times Y\to \bs^1$, $G({\xb},y)=F({\xb},y)/\xi({\xb})$. Then $G$ is still an eigenfunction for the automorphism ${T}|_{A\times Y}$ corresponding to the eigenvalue $e^{2\pi i\alpha}$. Moreover, this is readily translated to: $G\circ{T}|_{A\times Y}=R_\alpha \circ G$. Moreover, $G_\ast(\mub|_A\ot\nu)={\rm Leb}$, since the LHS measure must be $R_\alpha$-invariant. It follows that $G$ establishes a factor map from ${T}|_{A\times Y}$ to $R_\alpha$, and we can consider the corresponding graph joining $\rho_1$. We also join ${T}|_{A^c\times Y}$ with $R_\alpha$ by taking $\rho_2$ being the product measure. The final joining $\rho$ is given as $\mub(A)\rho_1+(1-\mub(A))\rho_2$, that is, for $\tilde{B}\subset {\Xb}\times Y$ and $C\subset\bs^1$,
$$
\rho(\tilde{B}\times C)=\mub(A)\rho_1((\tilde{B}\cap (A\times Y))\times C)+(1-\mub(A))\rho_2((\tilde{B}\cap (A^c\times Y))\times C).$$
Set $\chi(z)=z$, which is an eigenfunction of $R_\alpha$ (corresponding to $e^{2\pi i\alpha}$). We have
$$
\int_{{\Xb}\times Y\times\bs^1}F({\xb},y)\ov{\chi}(z)\;d\rho({\xb},y,z)=$$$$
\mub(A)\int_{A\times Y\times \bs^1}F({\xb},y)\ov{\chi}(z)\,d\rho_1({\xb},y,z)
+(1-\mub(A))\int_{A^c\times Y\times \bs^1}F({\xb},y)\ov{\chi}(z)\,d\rho_2({\xb},y,z)=$$$$
\mub(A)\int_{A\times Y\times \bs^1}F({\xb},y)\ov{\chi\circ G}({\xb},y)\,d(\mub|_A\ot\nu)({\xb},y)=$$$$\int_{A\times Y}F({\xb},y)\ov{G({\xb},y)}\,d(\mub|_A\ot\nu)({\xb},y)=\int_A\xi({\xb})\,d\mub({\xb})>0$$
and the result follows.
\end{proof}

\begin{Remark} Note that the assertion of Lemma~\ref{l:prince1} fails if $\alpha=0$. Indeed, each zero mean invariant function $F$ is orthogonal to all ergodic Markov images (as $F$ is measurable with respect to the sigma-algebra of invariant sets and this factor is disjoint from all ergodic automorphisms). In the above proof it was important that the mean of $\chi$ is zero which in the final computation made disappear the part of joining $\rho$ given by product measure.\end{Remark}

Despite the above remark, we will show that the assertion of Lemma~\ref{l:prince1} holds for $\alpha\neq 0$.
For simplicity, assume that $\alpha=1/2$, that is, we consider eigenvalue -1.  Assume first that $T$ is still aperiodic.
Note that if $F\circ {T}=-F$, then we have the same relation on the ergodic components. Then $F^2$ is an invariant function, so we see that on each ergodic component $\{{\xb}\}\times Y$ the function $F({\xb},\cdot)$ has (as before) constant modulus $\xi({\xb})$ and takes two values either $c_{\xb}$ or $-c_{\xb}$. Moreover, if $F({\xb},y)=c_{\xb}$ then $F\circ{T}({\xb},y)=-c_{\xb}$. If we fix $y_0\in Y$, then by taking $F({\xb},y_0)$ we make a measurable choice $d_{\xb}$ of either $c_{\xb}$ or $-c_{\xb}$ at each fiber $\{{\xb}\}\times Y$.

We claim now that we can find a positive measure subset $A'\subset A$ and take a measurable choice $A'\ni {\xb}\mapsto d_{\xb}$ so that
\beq\label{prince1}
\int_{A'\times Y} |F({\xb},y)|^2\frac1{d_{\xb}}\,d(\mub\ot\nu)({\xb},y)=\int_{A'} \xi({\xb})^2/d_{\xb}\,d\mub({\xb})\neq 0.
\eeq
Indeed, the existence of $A'$ follows from the fact that the integrand function is different from zero.
We define $G:A'\times Y\to\{-1,1\}$ by $G({\xb},y)=F({\xb},y)/d_{\xb}$ and repeat the previous proof with the circle replaced by the group $\{-1,1\}$, $R_{1/2}(j)=-j$ and $\chi(j)=j$ (whose mean is zero).

To cope with the general case, first assume that the aperiodic part is non-trivial, find a good joining of it with $R_{1/2}$ and then complete to the full joining  of $T$ and $R_{1/2}$ by taking the product joining of $T|_{X_n}$ ($n\geq1$) with $R_{1/2}$. Finally, if there is no aperiodic part, there must exists a non-trivial periodic part $T|_{X_{n_0}}$ for some $n_0\geq2$. Then simply repeat the proof of the aperiodic case with $T$ replaced by $T|_{X_{n_0}}$.

Given a countable subgroup $G$ denote by DISP$_G$ the family of discrete spectrum automorphisms whose set of eigenvalues is contained in $G$. This family is characteristic.

\begin{Prop}\label{p:prince1} Let $G$ be a countable subgroup of the circle.
Then
$\bfu\perp {\rm UE}\cap {\rm DISP}_G$ if and only if for each Furstenberg system $\kappa\in V(\bfu)$, we have $\sigma_{\pi_0,\kappa}(\{z\})=0$ for each $z\in G\setminus\{1\}$.
\end{Prop}
\begin{proof} $\Rightarrow$ (by contraposition) Suppose that for some $\kappa\in V(\bfu)$, say $$\kappa=\lim_{N\to\infty}\frac1{N_k}\sum_{n\leq N_k}\delta_{S^n\bfu},$$ and for some $1\neq z_0\in G$, we have $\sigma_{\pi_0,\kappa}(\{z_0\})>0$, that is, the spectral measure of $\pi_0$ has an atom at $z_0$. It follows that
$$
\pi_0=\pi_{z_0}+\pi_{z_0}^\perp,$$
where $\pi_{z_0}$ stands for the orthogonal projection of $\pi_0$ on the subspace of eigenfunctions corresponding to $z_0=e^{2\pi i\alpha}$ (we assume that $\alpha$ is irrational, in the  rational case $z_0\neq1$ the proof goes similarly). In view of Lemma~\ref{l:prince1}, there exists a joining $\rho\in J((X_{\bfu},\kappa,S), R_\alpha)$ such that
\beq\label{pronce3}
\int \pi_{z_0}(\omega)\overline{\chi(z)}\,d\rho(\omega,z)\neq0.
\eeq
By Theorem~\ref{t:lifting}, passing to a subsequence of $(N_k)$ if  necessary, we obtain $((S^n\bfu),(w_n))$  a generic sequence (along $(N_k)$) for $\rho$:
$$
\frac1{N_k}\sum_{n\leq N_k}\delta_{(S^n\bfu,w_n)}\to\rho,$$
where the set $\{n\colon w_{n+1}\neq R_\alpha w_n\}$ is of the form $b_1<b_2<\ldots$ with $b_{k+1}-b_k\to\infty$.
Using $(w_n)$, we pass to the corresponding orbital model to obtain a new uniquely ergodic system $(Y,S)$ which is a model of the irrational rotation $R_\alpha$. It follows that ($\underline{w}:=(w_n)$ and $\widetilde{\chi}((y_n))=\chi(y_0)$)
$$
\frac1{N_k}\sum_{n\leq N_k}\bfu(n)\tilde{\chi}(S^n\underline{w})=
\frac1{N_k}\sum_{n\leq N_k}\pi_0(S^n\bfu)\chi(w_n)\to \int\pi_0\ot\ov{\chi}\;d\rho.$$
Since $\pi_{z_0}^\perp \perp \chi$ (computed in $L^2(\rho)$, the spectral measures $\sigma_{\pi_{z_0}^\perp\ot 1}$ and $\sigma_{1\ot \chi}$ are mutually singular),
$$
\int\pi_0\ot\ov{\chi}\;d\rho=\int(\pi_{z_0}+\pi_{z_0}^\perp)\ot\ov{\chi}\;d\rho
$$$$
=\int\pi_{z_0}\ot\ov{\chi}\;d\rho\neq 0$$
and the result follows ($\bfu$ correlates with a uniquely ergodic system having the group of eigenvalues contained in $G$).

$\Leftarrow$  (by contraposition) If
$\bfu\not\perp (X,T)$ (for some u.e.\ $(X,T)$ with the group of eigenvalues contained in $G$), then for some $(N_k)$, $f\in C(X)$ and $x\in X$, we have $\frac1{N_k}\sum_{n\leq N_k}\bfu(n)f(T^nx)\to c\neq0$. Passing to a further subsequence of $(N_k)$ if necessary, we obtain then $c=\int\pi_0\ot f\,d\rho=\int \pi_0\,\EE_\rho[1\ot f\,|\,X_{\bfu}]\,d\kappa$. Since $(X,T)$ is u.e.\ with discrete spectrum ($\subset G$, we can also assume that $\int f=0$),  the spectral measure of $f$ has only atoms $\in G\setminus\{1\}$. The spectral measure of $\EE_\rho[1\ot f\,|\,X_{\bfu}]$ is absolutely continuous w.r.t.\ $\sigma_f$, so it is also  purely atomic and has atoms only in $G\setminus\{1\}$. Since $\pi_0$ is not orthogonal to $\EE_\rho[1\ot f\,|\,X_{\bfu}]$, its spectral measure has an atom belonging to $G\setminus\{1\}$.
\end{proof}

\begin{Remark}\label{f:MarIm} In fact, we have proved the following:
{\em Assume that $T\in{\rm Aut}\xbm$ and $f\in L^2_0\xbm$. Then $f\perp L^2({\rm Im}(\Phi_\rho))$ for each $\rho\in J(R,T)$ and all $R$ ergodic with discrete spectrum if and only if $\sigma_f\ll \eta+\delta_{\{1\}}$ for some continuous measure $\eta$ (on $\bs^1$).}\end{Remark}

Since the condition $\bfu\perp {\rm UE}\cap {\rm DISP}$  is equivalent to $\bfu\perp {\rm UE}\cap{\rm DISP}_G$ for all countable subgroups $G\subset\bs^1$, we obtain the following result.

\begin{Prop}\label{p:scd1a} Assume that $\bfu:\N\to\D$, $M(\bfu)=0$. Then
$\bfu\perp {\rm UE}\cap{\rm DISP}$ if and only if for each $\kappa\in V(\bfu)$ the spectral measure $\sigma_{\pi_0,\kappa}\ll\eta+\delta_{\{1\}}$ for some continuous measure $\eta$.\end{Prop}

An important observation is then that by  Wiener's lemma: for any measure $\sigma$ on the circle
$$
\sum_{1\neq z}|\sigma(\{z\})|^2=\sum_{z}|\sigma(\{z\})|^2-|\sigma(\{1\})|^2=$$
$$
\lim_{H\to\infty}\frac1H\sum_{h\leq H}|\widehat{\sigma}(h)|^2-\Big|
\lim_{H\to\infty}\frac1H\sum_{h\leq H}\widehat{\sigma}(h)\Big|^2.$$
The limits of that kind applied to the function $\pi_0$ in any Furstenberg system $(X_{\bfu},\kappa,S)$ of $\bfu$ are expressible in terms of autocorrelations of $\bfu$.  For example:

\begin{Cor} \label{c:scd2} Assume that $\bfu$ is generic. Then $\bfu$ is orthogonal to all uniquely ergodic models of ergodic transformations with discrete spectrum  if and only if
$$\lim_{H\to\infty}\frac1H\sum_{h\leq H}\lim_{N\to\infty}\frac1N\Big|\sum_{n\leq N}\bfu(n+h)\ov{\bfu(n)}\Big|^2=$$
$$\lim_{H\to\infty}\frac1H\sum_{h\leq H}\Big|\lim_{N\to\infty}\frac1N\sum_{n\leq N}\bfu(n+h)\ov{\bfu(n)}\Big|^2=$$$$
\Big|\lim_{H\to\infty}\frac1H\sum_{h\leq H}\lim_{N\to\infty}\frac1N\sum_{n\leq N}\bfu(n+h)\ov{\bfu(n)}\Big|^2.$$
\end{Cor}

\subsubsection{When Furstenberg systems are almost ergodic}
We consider $\bfu$ satisfying \eqref{now5} for each Furstenberg system $\kappa\in V(\bfu)$.
That is, we assume that
$$
\EE_\kappa[\pi_0\,|\,\ci_S]=0.$$
Let $\cf$ be any characteristic class. We recall that in view of \cite{Ka-Ku-Le-Ru} (see Theorem~B therein)
\beq\label{ve2}
\Big[\forall \kappa\in V(\bfu)\;\;\pi_0\perp L^2(\af(X_{\bfu},\kappa,S))\Big] \Rightarrow \bfu\perp \mathscr{C}_{\cf}, \eeq
where by $\mathscr{C}_{\cf}$ we denote the class of topological systems whose all visible invariant measures yield measure-preserving systems belonging to $\cf$.
We aim at proving the following result.

\begin{Prop}\label{p:alaFran} Assume that $\bfu$ satisfies \eqref{now5} and each $\kappa\in V(\bfu)$ has purely atomic ergodic decomposition (i.e.\ it has countably many ergodic components). Assume that $\cf$ is a characteristic class. Then
$$
\bfu\perp \mathscr{C}_{\cf}\text{ if and only if }\bfu\perp {\rm UE}\cap\cf.$$
\end{Prop}

\begin{proof} We only need to show that if $\bfu\perp {\rm UE}\cap\cf$, then the LHS (Veech condition) of the implication~\eqref{ve2} is satisfied.
Let $\kappa\in V(\bfu)$, so, by assumption,
$$
\mbox{$\ci_S$ is purely atomic}.$$
Let $\af $ stand for the largest $\cf$-factor of $(X_{\bfu},\kappa,S)$.
We set
$$
g=\EE_\kappa[\pi_0|\af]=\proj_{L^2(\af )}(\pi_0).$$
By our assumption, the ergodic decomposition of $\kappa$ is purely atomic, i.e.
$$X_{\bfu}/\ci_S=\textcolor{DarkViolet}{\Xb_{\bfu}}=\{c_1,c_2,\ldots\}=\{1,2,\ldots\}$$
with $q:X_{\bfu}\to X_{\bfu}/\ci_S$ the quotient map and $\kappa=\sum_{i\geq1}\alpha_i\kappa_{i}$   ($\kappa_{i}$ are all ergodic and supported on the fiber over $i$; $\alpha_i>0$, $\sum_i\alpha_i=1$).
We now define a self-joining $\lambda$ of $(S|_{\af},\kappa|_{\af})$ (in fact, of $(S,\kappa)$; the spaces of ergodic components for $\kappa$ and $\kappa|_{\af}$ are the same) by putting
$$\la|_{q^{-1}(i)\times q^{-1}(i)}=\alpha_i^2\Delta_{\kappa_{i}}\text{ and }
\la|_{q^{-1}(i)\times q^{-1}(j)}=\alpha_i\alpha_j(\kappa_{i}\ot\kappa_{j})$$ for $i\neq j$.  Since $\la|_{X_{\bfu}/\ci_S\times
X_{\bfu}/\ci_S}(i,j)=\alpha_i\alpha_j$, it is not hard to see that $\la|_{\ci_S\ot\ci_S}$ equals $\kappa|_{\ci_S}\ot \kappa|_{\ci_S}$. On the other hand, for $i\neq j$, we have
$$
\EE_{\la}[g\ot \ov{g}\,|\,\ci_S\ot\ci_S](i,j)=
\int g\ot\ov{g}d\la_{i,j}=\int g\,d\kappa_i\cdot\ov{\int g\,d\kappa_j}=0$$
because of our \eqref{now5} assumption. Moreover,
\beq\label{perg}
\EE_{\la}(g\ot \ov{g}\,|\,\ci_S\ot\ci_S)(i,i)
=\int|g|^2\,d\kappa_{i}\eeq
But by Corollary~\ref{c:enfin}, it follows that
$\EE_{\la}(g\ot \ov{g}\,|\,\ci_{S\times S})=0$, so the more
$\EE_{\la}(g\ot \ov{g}\,|\,\ci_S\ot\ci_S)=0$, and
finally $g=0$ by \eqref{perg}, whence the Veech condition holds.
\end{proof}
Note that if  all $\kappa\in V(\bfu)$ are ergodic then \eqref{now5} follows immediately whenever the mean of $\bfu$ is zero, see \eqref{ghk2}.

\begin{Remark} A slightly more general form of Frantzikinakis' theorem \cite{Fr}:
\begin{quote}
{\em If all logarithmic Furstenberg systems for $\mob$ have purely atomic ergodic decomposition then the logarithmic Chowla holds}
\end{quote}
follows now from  Frantzikinakis-Host's theorem \cite{Fr-Ho} on the logarithmic M\"obius orthogonality of all zero entropy uniquely ergodic systems which, by Proposition~\ref{p:alaFran}, implies the validity of logarithmic Sarnak's conjecture
and then by the Tao's result about the equivalence of logarithmic versions of Sarnak and Chowla conjectures \cite{Ta}.\end{Remark}

\subsection{Veech and Sarnak's conditions are not equivalent - a counterexample} \label{s:aoma}
Let $\cf$ be a characteristic class.
Given an arithmetic function $\bfu:\N\to\D$, following \cite{Ka-Ku-Le-Ru}, we say that $\bfu$ {\em satisfies the Sarnak condition} (relative to $\cf$) if
$$
\bfu\perp (X,T)\text{ for all }(X,T)\in\mathscr{C}_{\cf}.$$
In \cite{Ka-Ku-Le-Ru}, it has been proved that the Veech condition (which is the LHS of \eqref{ve2}) implies the Sarnak condition for $\cf$, and the two conditions are equivalent whenever $\cf=\cf_{\rm ec}$ (see Section~\ref{sec:cc} for the definition of an ec-class). One of the  problems in \cite{Ka-Ku-Le-Ru} left open was the question whether the Veech and Sarnak conditions are equivalent for an {\bf arbitrary} characteristic class $\cf$. We will now show that it is not the case.\footnote{The result has been obtained jointly with A. Kanigowski.}

Let $\cf={\rm Rig}((q_n))$, that is the class of automorphisms which are $(q_n)$-rigid: $R\in\cf$ if $f\circ R^{q_n}\to f$ for all $f$ in the $L^2$-space of $R$. Similarly as in \cite{Ka-Ku-Le-Ru}, we use the result of Fayad and Kanigowski \cite{Fa-Ka} which gives
\beq\label{rigMa0}
\mbox{$(q_n)$ which is a rigidity time for a weakly mixing automorphism,}\eeq but
\beq\label{rigMa}
\mbox{$(q_n)$ is not a rigidity time for {\bf any} (non-trivial) rotation.}\eeq

Consider ${T}(x,y)=(x,y+x)$ on $\T^2$. We consider this automorphism with invariant measure ${\mu}:=\sigma\ot{\rm Leb}_{\T}$ (cf.\ Section~\ref{e:przyk}), where $\sigma$ is the continuous measure given by the maximal spectral type of the weakly mixing automorphism in~\eqref{rigMa0}. As proved in \cite{Ka-Ku-Le-Ru}, $\cf_{\rm ec}\subsetneq \cf$, but in fact, $({T},{\mu})\in\cf$ while the maximal $\cf_{\rm ec}$-factor of it, in view of~\eqref{rigMa}, is equal to $\ci_{{T}}$ (corresponding to the first coordinate sigma-algebra $(\T,\sigma)$). Let us choose any Borel function $F:\T^2\to\{-1,1\}$, so that $\EE[F\,|\,\ci_{{T}}]=0$ (for example $F$ equals -1 on $\T\times[0,1/2)$ and 1 otherwise). Consider now the stationary process $(F\circ {T}^n)_{n\in\Z}$ denoting by $\kappa$ its distribution. As a measure-preserving system, it is a factor of $({T},{\mu})$ and $F$ is orthogonal to $L^2(\T,\sigma)$ which is the largest $\cf_{\rm ec}$-factor (of ${T}$), so the more $F$ is orthogonal to the $L^2$-space of its largest $\cf_{\rm ec}$-factor. By passing to the shift model, we obtain a stationary $\pm1$-valued process $(\pi_0\circ S^n)$ with distribution $\kappa$. Since we deal now with a shift invariant measure on the full shift, due to \cite{Co-Do-Se}, there is a generic point $\bfu:\N\to\{-1,1\}$ for $\kappa$.
Since $\pi_0$ is orthogonal to the $L^2$-space of the largest $\cf_{\rm ec}$-factor, the Veech condition is satisfied for $\bfu$. By Theorem~B in \cite{Ka-Ku-Le-Ru}, it follows that the Sarnak condition (for $\cf_{\rm ec}$) is satisfied. That is, $\bfu\perp \mathscr{C}_{\cf_{\rm ec}}$. However, the class we deal with satisfies $\cf_{\rm ec}\subset \cf$, so as observed in \cite{Ka-Ku-Le-Ru},
$$
\mathscr{C}_{\cf_{\rm ec}}=\mathscr{C}_{\cf}.$$
It follows that $\bfu$ satisfies the Sarnak condition for $\cf$. But $\bfu$ cannot satisfy the Veech condition for $\cf$, as $({T},{\mu})\in\cf$, so also the factor determined by the process $(\pi_0\circ S^n)$ is in $\cf$, and the largest $\cf$-factor is the whole sigma-algebra. If the Veech condition holds, then $F=0$ which is an absurd.

\begin{Remark} In our reasoning, it was important that the class we have chosen
satisfies
$$\cf_{\rm ec}\subsetneq \cf,$$
since on one hand $\mathscr{C}_{\cf}=\mathscr{C}_{\cf_{\rm ec}}$ and on the other we can play with members which are in $\cf$ but not in $\cf_{\rm ec}$.
Note that the characteristic classes included in Erg$^\perp$ pointed out in \cite{Be-Go-Ru} (there are $R\in{\rm Erg}^\perp$ such that {\bf all} self-joinings remain in Erg$^\perp$)
will satisfy our requirements as each characteristic class $\cf\subset {\rm Erg}^\perp$ satisfies $\cf_{\rm ec}=$ID, so our reasoning applies.

Note finally that the sequence $\bfu$ which we used to obtain the counterexample is generating for a measure which yields the (unique) Furstenberg system and the corresponding measure-preserving system is in Erg$^\perp\cap {\rm Rig}((q_n))$.
\end{Remark}

\subsection{Application: an averaged Chowla property}\label{s:acp} We will now show that an averaged Chowla conjecture for $\bfu$ (defined below) is equivalent to an orthogonality conjecture~\eqref{sarnak10} for topological systems whose all invariant measures yield systems from a special characteristic class. We fix a bounded $\bfu:\N\to\C$ and start with a small extension of Proposition~\ref{p:ju1}.

\begin{Prop}\label{p:dz1} The following conditions are equivalent:\\
(a) $\bfu$ has zero mean on typical short interval (this is equivalent to the fact that the first GHK-norm of $\bfu$ vanishes).\\
(b)  $\EE_\kappa[\pi_0\,|\,\ci_S]=0$ for each Furstenberg system $\kappa\in V(\bfu)$.\\
(c) The spectral measure $\sigma_{\pi_0,\kappa}$ has no atom at~1 for each Furstenberg system $\kappa\in V(\bfu)$.
\end{Prop}
\begin{proof} The fact that (b) and (c) are equivalent comes from spectral theory.\end{proof}

We now consider $\bfu$ satisfying an averaged Chowla property \cite{Ka-Ku-Le-Ru} based on \cite{Ma-Ra-Ta}: \beq\label{aChp}\lim_{H\to\infty}\frac1H\sum_{h\leq H}\lim_{k\to\infty}\frac1{M_k}\Big|\sum_{m\leq M_k}\bfu(m+h)\bfu(m)\Big|=0\eeq
for each $(M_k)$ defining a Furstenberg system of $\bfu$.

\begin{Remark}As stated, \eqref{aChp} should be called an averaged 2-Chowla property. However, as shown in \cite{Ma-Ra-Ta}, see also Appendix~A in \cite{Ka-Ku-Le-Ru}, the averaged 2-Chowla property implies: $$\lim_{H\to\infty}\frac1{H^k}\sum_{h_1,\ldots,h_k\leq H}\lim_{k\to\infty}\frac1{M_k}\Big|\sum_{n\leq M_k}\bfu(n)\bfu(n+h_1)\ldots\bfu(n+h_k)\Big|=0$$
for each $k\geq1$, $1\leq h_1<\ldots<h_k$. From the ergodic theory point of view, it is the classical result that weak mixing property implies weak mixing of all orders.\end{Remark}
As shown in \cite{Fe-Ku-Le}, \cite{Ka-Ku-Le-Ru}, \eqref{aChp} is equivalent to the fact that $\sigma_{\pi_0,\kappa}$ is continuous for each Furstenberg system $\kappa\in V(\bfu)$. It follows (see Proposition~\ref{p:dz1} above) that the averaged Chowla property implies the zero mean on typical short interval, that is (see Proposition~\ref{p:ju1}), $\|\bfu\|_{u^1}=0$. We aim at proving the following result.

\begin{Th}\label{t:dz2} Assume that $\|\bfu\|_{u^1}=0$. The following conditions are equivalent:\\
(a) $\bfu$ satisfies an averaged Chowla property.\\
(b) $\bfu$ is orthogonal with all uniquely ergodic models of discrete spectrum automorphisms.\\
(c) $\bfu$ is orthogonal to all topological systems whose all invariant measures yield discrete spectrum measure-preserving systems.

Moreover, every of the conditions (a), (b) and (c) implies:\\
(d) For each $\vep>0$ there exists $H_0\geq1$ such that for each $H\geq H_0$ and each $Q\geq1$, we have\footnote{For the Liouville function this assertion has been proved to hold by Sacha Mangerel in 2019 (private communication). The property in (d) gives an averaged uniformity of zero mean on ``intervals'' along arithmetic progressions.}
$$
\frac1Q\sum_{q\leq Q}\limsup_{N\to\infty}\frac1N\sum_{n\leq N}
\left| \frac1H\sum_{h\leq H}\bfu(hq+n)\right|^2<\vep.$$ \end{Th}
\begin{proof} (a)$\Rightarrow$(b)  Remembering that satisfying the averaged Chowla conjecture yields the continuity of spectral measure of $\pi_0$, the claim follows directly from Proposition~\ref{p:ouesZEgen} and the result in Remark~\ref{f:MarIm} (see also Corollary~\ref{c:scd2}).

(b)$\Rightarrow$(a) Again it follows from the result formulated in Remark~\ref{f:MarIm} (remembering that by Proposition~\ref{p:dz1},  the spectral measure of $\pi_0$ has no atom at~1).

Since (c)$\Rightarrow$(b), it is enough to show that (a)$\Rightarrow$(c) which follows from \cite{Fe-Ku-Le} (see the proof of Corollary 3.20 therein) since (a) implies that the spectral measure of $\pi_0$ is continuous (for each $\kappa\in V(\bfu)$).

As the proof of the second assertion, it is interesting for its own, we postpone it to a separate section,  see Section~\ref{s:dowodMangerel}.
\end{proof}

\begin{Remark} Note that the implication (c)$\Rightarrow$(a) is implicit in \cite{Ka-Ku-Le-Ru}. Indeed, the Veech condition for DISP means that the spectral measure $\sigma_{\pi_0,\kappa}$ has no atoms (for each $\kappa\in V(\bfu)$). Then in Section~5.6 \cite{Ka-Ku-Le-Ru} the Veech condition is obtained for all $\bfu$ satisfying strong $\bfu$-MOMO property for all rotations on the circle. Now, (c) implies (b) and (b) implies that $\bfu$ satisfies the strong $\bfu$-MOMO property for all irrational rotations. On the other hand,  (c) also implies that $\bfu$ is orthogonal to all systems whose all invariant measures yield measure-preserving systems with rational discrete spectrum.  But this class of measure-preserving systems forms so called ec characteristic class. So the claim now follows from Proposition~2.17 in \cite{Ka-Ku-Le-Ru}.\end{Remark}
If in the assumption of Theorem~\ref{t:dz2} we know additionally that $\sigma_{\pi_0}$ cannot have irrational atoms, then (d) is equivalent to all other conditions (a)-(c).
As shown by Frantzikinakis and Host, Theorem 1.5 \cite{Fr-Ho2}, if $\bfu$ is multiplicative (bounded by~1) then for no $\kappa\in V^{\rm log}(\bfu)$, the spectral measure of $\pi_0$ has an irrational atom. Therefore, we obtain the following (cf.\ Remark~\ref{r:gdylog}).

\begin{Cor}\label{c:nasze}Let $\bfu$ be a bounded by 1 multiplicative function satisfying (the logarithmic) $\|\bfu\|_{u^1}=0$. Then, the following conditions are equivalent:\\
(a) $\bfu$ satisfies the averaged log Chowla property: $$\lim_{H\to\infty}\frac1H\sum_{h\leq H}\lim_{k\to\infty}\frac1{\log M_k}\Big|\sum_{m\leq M_k}\frac1m\bfu(m+h)\bfu(m)\Big|=0$$ for each $(M_k)$ defining a logarithmic Furstenberg system of $\bfu$.\\
(b) $\bfu$ is log orthogonal to all uniquely ergodic models of discrete spectrum automorphisms.\\
(c) $\bfu$ is log orthogonal to all topological systems whose all invariant measures yield discrete spectrum measure-preserving systems.\\
(d) For each $\vep>0$ there exists $H_0\geq1$ such that for each $H\geq H_0$ and each $Q\geq1$, we have
$$
\frac1Q\sum_{q\leq Q}\limsup_{N\to\infty}\frac1{\log N}\sum_{n\leq N}
\frac1n\left| \frac1H\sum_{h\leq H}\bfu(hq+n)\right|^2<\vep.$$ \end{Cor}

It is reasonable to conjecture that the above is also true for the Ces\`aro averages. For example, recently in \cite{Fr-Le-Ru}, it has been proved that for pretentious multiplicative functions we cannot have irrational eigenvalues for the Furstenberg systems.

\section{Proof of the second assertion in Theorem~\ref{t:dz2}}\label{s:dowodMangerel}
\subsection{Ergodic theory}
\begin{Prop}\label{p:p1}
Let $(X,\mathcal{B}_X,\mu,T)$ be a measure-preserving system. Let $f\in L_0^2\xbm$ and assume that its spectral measure $\sigma_f$ has no atoms at rationals. Then, for each $\vep>0$ there exists $H_0\geq1$ such that for each $H\geq H_0$ and each $Q\geq1$, we have
\beq\label{sacha}
\frac1Q\sum_{q\leq Q}\left\|
\frac1H\sum_{h\leq H}f\circ T^{qh}\right\|_{L^2(\mu)}^2<\vep.\eeq
\end{Prop}
\begin{proof}
By assumption, $\sigma_f(\{0\})=0$ and we can assume that $\sigma_f(\T)=1$ (which is equivalent to $\|f\|_{L^2(\mu)}=1$).

Suppose that \eqref{sacha} does not hold. Therefore,
\begin{multline}
\;\;\;\;\;\;\;\;\;\;\;\;\;\left(\exists \vep>0\right)\left(\forall H_0\geq1\right)\left(\exists H\geq H_0\right)\left(\exists Q\geq 1\right)\\
\frac1Q\sum_{q\leq Q}\left\|
\frac1H\sum_{h\leq H}f\circ T^{qh}\right\|_{L^2(\mu)}^2\geq\vep.\end{multline}
Equivalently (with the same quantifiers),
$$
\int_{\T}\frac1Q\sum_{q\leq Q}\left|
\frac1H\sum_{h\leq H}\left(e^{2\pi i qt}\right)^h\right|^2\,d\sigma_f(t)\geq \vep.$$
For $t\in\T=[0,1)$, denote
$$
\phi^H_q(t):=\left|
\frac1H\sum_{h\leq H}\left(e^{2\pi i qt}\right)^h\right|^2,$$
$$
\Phi_Q^H(t)=\frac1Q\sum_{q\leq Q}\phi^H_q(t).$$
We have $0\leq \phi^H_q(t)\leq1$ and $0\leq \Phi^H_Q(t)\leq1$.
It follows that
$$
\int_{\T} \Phi^H_Q(t)\,d\sigma_f(t)\leq\frac{\vep}{2}\sigma_f\left(\Big\{t\in\T:\:
\Phi^H_Q(t)<\frac{\vep}2\Big\}\right)
+\sigma_{f}\left(\Big\{t\in\T:\:\phi^H_Q(t)\geq \frac{\vep}2\Big\}\right).$$
Since $\int_{\T} \phi^H_Q(t)\,d\sigma_f(t)\geq\vep$ (and $\sigma_f(\T)=1$), we have
$$\sigma_f\left(\Big\{t\in\T:\:\phi^H_Q(t)\geq\frac{\vep}2\Big\}\right)\geq
\frac{\vep}2.$$
The same argument shows that
$$ \Phi^H_Q(t)\leq\frac{\vep}{4}\frac{|\{q\leq  Q:\:\phi^H_q(t)<\vep/4\}|}{Q}+\frac{
|\{q\leq Q:\:\phi^H_q(t)\geq \vep/4\}|}Q,$$
so if $\Phi^H_Q(t)\geq \vep/2$, we obtain
$$
\frac{|\{q\leq Q:\: \phi^H_q(t)\geq\vep/4\}|}Q\geq\frac{\vep}4.$$
On the other hand,
$$
\phi^H_q(t)\leq\frac4{H^2\left|1-e^{2\pi iqt}\right|^2},$$
so if $\phi^H_q(t)\geq\vep/4$, we obtain that
$$
\left|1-e^{2\pi i qt}\right|\leq\frac1H\cdot\frac4{\sqrt\vep}.$$
We hence obtain that:
\beq\label{partielle}
\mbox{for infinitely many $H\geq1$, we have $\sigma_f(A_H)\geq\frac{\vep}2$},\eeq
where $A_H:=\left\{t\in\T:\:\frac1Q\sum_{q\leq Q}\raz_{\left|1-e^{2\pi iqt}\right|\leq\frac1H\cdot\frac4{\sqrt\vep}}\geq\frac{\vep}4\right\}$.
It follows that
$$
\sigma_f\left(\bigcap_{H_0}\bigcup_{H\geq H_0}A_H\right)\geq \frac{\vep}2.$$
Let us take $t\in \bigcap_{H_0}\bigcup_{H\geq H_0}A_H$. Hence, for infinitely many $H$, there is $Q$ (which depends on $H$) such that
\beq\label{th1}
\frac1Q\left|\left\{q\leq Q:\:\left|1-e^{2\pi iqt}\right|\leq\frac1H\cdot\frac4{\sqrt\vep}\right\}\right|
\geq\frac{\vep}4.\eeq
Assume that $H$ and $Q$ satisfy~\eqref{th1}. Let
$$E:=\left\{q\leq Q:\:\left|1-e^{2\pi iqt}\right|\leq\frac1H\cdot\frac4{\sqrt \vep}\right\}\neq\emptyset.$$
Then either
\begin{itemize}
\item there is only one element in $E$ and then (by~\eqref{th1}) $1\leq q\leq Q\leq\frac4{\vep}$ or
\item there are at least two distinct elements in $E$.
\end{itemize}
In the latter case, set $m:=\min\{q_2-q_1:\:q_1<q_2\text{ both in }E\}$. Then, $m\leq Q$ and
$$
|E|\leq \frac Qm+1\leq\frac{2Q}m.$$
But, by \eqref{th1}, we know that $\frac{|E|}Q\geq\frac{\vep}4$, so
$\frac2m\geq\frac{\vep}4$, that is, $m\leq \frac{8}{\vep}$. Hence, in $E$, there are two different integers $q_1,q_2$ with $0<q_2-q_1\leq\frac8{\vep}$. Since
$$
\left|1-e^{2\pi i q_jt}\right|\leq\frac1H\cdot\frac4{\sqrt\vep}$$
for $j=1,2$, we obtain
$$
\left|1-e^{2\pi i (q_2-q_1)t}\right|=\left|e^{2\pi iq_1t}-e^{2\pi i q_2t}\right|\leq\frac1H\cdot\frac8{\sqrt\vep}.$$

We have proved that (in both cases!) whenever $t\in\bigcap_{H_0\geq1}\bigcup_{H\geq H_0}A_H$, then for infinitely many $H$ there is $1\leq q\leq8/\vep$ satisfying
$\left|1-e^{2\pi i qt}\right|\leq\frac1H\cdot\frac8{\sqrt\vep}$. It follows that the set $\bigcap_{H_0\geq1}\bigcup_{H\geq H_0}A_H$ contains only rational numbers with denominators at most $8/\vep$. On the other hand $\sigma_f(\bigcap_{H_0\geq}\bigcup_{H\geq H_0}A_H)>0$, so the spectral measure of $f$ has rational atoms, a contradiction.
\end{proof}

\begin{Remark}The other direction in Proposition~\ref{p:p1} is trivial: Let $F_r\subset L^2_0(\mu)$ be the space of eigenfunctions corresponding to $1/r$, $f=g+h$, $g\in F_r$, $h\in F_r^\perp$ (both $F_r$ and $F_r^\perp$ are $T$-invariant). By the Pythagorean theorem: $$\left\|
\frac1H\sum_{h\leq H}f\circ T^{qh}\right\|_{L^2(\mu)}^2\geq \left\|
\frac1H\sum_{h\leq H}g\circ T^{qh}\right\|_{L^2(\mu)}^2.$$ Since $g\circ T^{qh}=e^{2\pi iqh/r}g,$
$$\left\|
\frac1H\sum_{h\leq H}g\circ T^{qh}\right\|_{L^2(\mu)}^2=\left|
\frac1H\sum_{h\leq H}e^{2\pi iqh/r}\right|\|g\|_{L^2(\mu)}^2.$$ Now, for $q$ which is a multiple of $r$ we obtain the constant value $\|g\|_{L^2(\mu)}$, otherwise, it is zero.
\end{Remark}

\subsection{(a) implies the second assertion}
Our aim is to give an ``ergodic proof'' of the following result:
\begin{Prop}\label{p:ju3} Assume that $\bfu$ is a (bounded) arithmetic  function such that the spectral measure of $\pi_0$ has no rational atoms for all $\kappa\in V(\bfu)$. Then,
for each $\vep>0$ there exists $H_0\geq1$ such that for each $H\geq H_0$ and each $Q\geq1$, we have
$$
\frac1Q\sum_{q\leq Q}\limsup_{N\to\infty}\frac1N\sum_{n\leq N}
\left| \frac1H\sum_{h\leq H}\bfu(hq+n)\right|^2<\vep.$$\end{Prop}
\begin{proof} Suppose the result does not hold. So there exists $\vep_0>0$ such that for each $H_0\geq1$ there exist $H\geq H_0$ and $Q$ (which will depend on $H$) such that
$$
\frac1Q\sum_{q\leq Q}\limsup_{N\to\infty}\frac1N\sum_{n\leq N}
\left| \frac1H\sum_{h\leq H}\bfu(hq+n)\right|^2\geq\vep_0.$$
By taking $H_0=k$, we obtain infinitely many $H_k$ (with the corresponding choice of $Q_k$) such that
$$
\frac1{Q_k}\sum_{q\leq Q_k}\limsup_{N\to\infty}\frac1N\sum_{n\leq N}
\left| \frac1{H_k}\sum_{h\leq H_k}\bfu(hq+n)\right|^2\geq\vep_0.$$
Then choose $N_k$ so that
\beq\label{dosp100}
\frac1{Q_k}\sum_{q\leq Q_k}\frac1{N_k}\sum_{n\leq N_k}
\left| \frac1{H_k}\sum_{h\leq H_k}\bfu(hq+n)\right|^2\geq\vep_0/2.\eeq
By passing to a subsequence if necessary, we can  assume that $\frac1{N_k}\sum_{n\leq N_k}\delta_{S^n\bfu}\to\kappa$. Due to our assumption on $\bfu$, we can apply Proposition~\ref{p:p1} (for $f=\pi_0\in L^2(X_{\bfu},\kappa)$)  to obtain: for some $H_0\geq1$, all $H\geq H_0$ and all $Q\geq1$, we have
$$
\frac1Q\sum_{q\leq Q}\int_{X_{\bfu}}\left|\frac1H\sum_{h\leq H}\pi_0\circ S^{qh}\right|^2\,d\kappa<\vep_0/4,$$
or, equivalently (by the definition of $\kappa$)
$$
\frac1Q\sum_{q\leq Q}\lim_{k\to\infty}\frac1{N_k}\sum_{n\leq N_k}\left|\frac1H\sum_{h\leq H}\bfu(qh+n)\right|^2<\vep_0/4,$$
which because of quantifiers on $H$ and $Q$ is in conflict with~\eqref{dosp100}.
\end{proof}

\begin{Remark}
For $\bfu$ equal to the Liouville function, Proposition~\ref{p:ju3} has been proved by S. Mangerel using purely number theoretic tools in 2019 (private communication).
\end{Remark}

\section{Empirical approximations of self-joinings of Furstenberg systems}\label{s:osiem}

Given a bounded arithmetic function $\bfu:\N\to\D$, we know concretely how to approach the Furstenberg systems of $\bfu$: they are (by definition) all weak$^*$-limits of empirical measures, that is, of the form
\begin{equation}
 \label{eq:FS}
 \kappa = \lim_{k\to\infty} \frac{1}{N_k} \sum_{1\le n\le N_k}\delta_{S^{n}\bfu},
\end{equation}
where $S$ is the shift map on $\D^\N$, and $(N_k)$ is any increasing sequence of natural integers such that the limit exists.

The purpose of this section is to get a similar description of {\bf any 2-fold self-joining of Furstenberg systems} of $\bfu$. Since the subshift context is irrelevant, we consider the more general setting: $X$ is a compact metric space, $S$ is a homeomorphism of $X$, and $\bfu\in X$. We continue to refer to any weak$^*$-limit of the form~\eqref{eq:FS} as a Furstenberg system of $\bfu$. We need the following definition.

\begin{Def}[Locally orbital sequence of permutations]
 Let $(N_k)_{k\ge1}$ be an increasing sequence of natural integers, and for each $k$, let $\phi_k$ be a permutation of $\{1,\ldots,N_k\}$. The sequence $(\phi_k)_{k\ge1}$ is said to be \emph{locally orbital} if
 \[
  \frac{1}{N_k} \Bigl| \Bigl\{ n\in\{1,\ldots,N_k-1\}: \phi_k(n+1)=\phi_k(n)+1 \Bigr\} \Bigr| \tend{k}{\infty} 1.
 \]
\end{Def}

Here is the explanation for the terminology: given such a sequence of permutations, we consider on the Cartesian square $X\times X$ the sequence of ``empirical measures'' of the form
\begin{equation}
 \label{eq:empirical_joining_of_FS}
 \frac{1}{N_k} \sum_{1\le n\le N_k} \delta_{(S^{n}\bfu,S^{\phi_k(n)}\bfu)}.
\end{equation}
When $(\phi_k)$ is locally orbital, for large $k$ these empirical measures are mostly supported on long pieces of orbits for $S\times S$.

Observe that any empirical measure of the form~\eqref{eq:empirical_joining_of_FS} has both marginals equal to $\frac{1}{N_k} \sum_{1\le n\le N_k}\delta_{S^{n}\bfu}$.
Therefore, a necessary condition for such a sequence to converge is that $\bfu$ be quasi-generic for some $S$-invariant measure $\kappa$ along $(N_k)$. Moreover, if $\lambda$ is the weak$^*$-limit of such a sequence, then the locally-orbital condition implies the $S\times S$-invariance of $\lambda$, hence under this assumption $\lambda$ is a 2-fold self-joining of a Furstenberg system $\kappa$ of $\bfu$. Our goal now is to prove the reciprocal:

\begin{Prop}
\label{prop:selfjoiningsofFS}
 Let $\bfu$ be quasi-generic along $(N_k)$ for some Furstenberg system $\kappa$, and let $\lambda$ be a 2-fold self-joining of $\kappa$. Then there exists a locally orbital sequence $(\phi_k)$, where each $\phi_k$ is a permutation of $\{1,\ldots,N_k\}$, such that
 \begin{equation}
  \label{eq:selfjoinings_of_FS}
  \lambda = \lim_{k\to\infty} \frac{1}{N_k} \sum_{1\le n\le N_k} \delta_{(S^{n}\bfu,S^{\phi_k(n)}\bfu)}.
 \end{equation}
 \end{Prop}

\subsection{Without dynamics}

We first describe a strategy to construct appropriate permutations without taking the dynamics into account. We only use for now that $\kappa$ is a probability measure on $X$ and that $\lambda$ is a 2-fold coupling of $\kappa$, that is, a probability measure on $X\times X$ with both marginals equal to $\kappa$. We fix a finite partition $\P$ of $X$, each atom $P$ of which satisfies $\kappa(P)>0$. We consider a finite number of points $x_1,\ldots,x_N\in X$, and for each atom $P\in\P$, we define the number of visits in $P$ by
\begin{equation}
 \label{eq:defV}
 V(P):=\sum_{1\le n\le N} \ind{P}(x_n).
\end{equation}
Let us assume that the empirical measure $(1/N)\sum_{1\le n\le N} \delta_{x_n}$ is a good approximation of $\kappa$ on $\P$ in the following sense: for some (small) real number $\eps>0$, we have
\begin{equation}
 \label{eq:approximation}
 \forall P\in\P,\quad \left| \frac{V(P)}{N} - \kappa(P) \right| \le \eps\kappa(P).
\end{equation}

\begin{Lemma}
 \label{lemma:sharing}
 Under the above hypotheses, for $N$ large enough (depending on $\P$, $\eps$, $\kappa$ and $\lambda$), we can define a family of nonnegative integers
 \[ \bigl(V(P\times P')\bigr)_{P,P'\in\P} \]
 such that
 \begin{align}
  \forall P,P'\in\P,&\quad \left| \frac{V(P\times P')}{N} - \lambda(P\times P') \right| \le 2\eps\lambda(P\times P'),\tag{C1}\label{eq:C1} \\
  \forall P\in\P,&\quad \sum_{P'\in\P}V(P\times P') \le V(P), \tag{C2}\label{eq:C2}\\
  \forall P'\in\P,&\quad \sum_{P\in\P}V(P\times P') \le V(P'). \tag{C3}\label{eq:C3}
 \end{align}

 Moreover, these numbers are such that the following implication holds: for all $P_1$, $P'_1$, $P_2$, $P'_2$ atoms of $\P$,
\begin{equation}
 \left.
 \begin{aligned}
   \kappa(P_1) &=\kappa(P_2),\\
   V(P_1)&=V(P_2),\\
   \kappa(P'_1)&=\kappa(P'_2),\\
   V(P'_1)&=V(P'_2),\\
   \lambda(P_1\times P'_1)&=\lambda(P_2\times P'_2),\\
 \end{aligned}
 \right\}
\Longrightarrow V(P_1\times P'_1)=V(P_2\times P'_2)
\tag{C4}\label{eq:C4}
\end{equation}

\end{Lemma}

\begin{proof}
 For $P,P'$ atoms of $\P$, we first define
 \[
    V_1(P\times P') := \left\lfloor \frac{V(P)\, \lambda(P\times P')}{\kappa(P)}\right\rfloor.
 \]
 Note that, if $\lambda(P\times P')=0$, then Condition \eqref{eq:C1} automatically holds for $V_1(P\times P')$.
 Now, for atoms $P,P'$ of $\P$ such that $\lambda(P\times P')>0$, using~\eqref{eq:approximation} we get
 \begin{align*}
  &\left| \frac{V_1(P\times P')}{N} - \lambda (P\times P')  \right| \\
  & \leq \frac{1}{N} \left| V_1(P\times P') - \frac{V(P)\, \lambda(P\times P')}{\kappa(P)} \right|
    + \frac{ \lambda(P\times P')}{\kappa(P)}  \left| \frac{V(P)}{N} - \kappa(P) \right| \\
  & \leq \frac{1}{N} + \eps \lambda(P\times P').
 \end{align*}
 Hence, we also have \eqref{eq:C1} for $V_1(P\times P')$ provided
 \begin{equation}
  \label{eq:N_large}
  N \ge \max_{P,P'\in\P:\lambda(P\times P')>0}\ \frac{1}{\eps \lambda(P\times P')}.
 \end{equation}

Since $V_1(P\times P') \le \frac{V(P)\, \lambda(P\times P')}{\kappa(P)}$, and using the fact that the first marginal of $\lambda$ is $\kappa$, summing this inequality over $P'\in\P$ we get Condition~\eqref{eq:C2} for $V_1$. But there is no obvious reason why \eqref{eq:C3} should hold for $V_1$, which is why we also introduce, for atoms $P,P'$ of $\P$,
 \[
    V_2(P\times P') := \left\lfloor \frac{V(P')\, \lambda(P\times P')}{\kappa(P')}\right\rfloor.
 \]
By the same arguments, \eqref{eq:C1} holds for $V_2(P\times P')$ if $N$ satisfies~\eqref{eq:N_large}, and this time \eqref{eq:C3} is valid for $V_2$. And finally we set
\[
 \forall P,P'\in\P,\quad V(P\times P'):= \min \bigl\{ V_1(P\times P'),V_2(P\times P')\bigr\}.
\]
Then $V$ clearly satisfies all required conditions.
\end{proof}

 \begin{Cor}
 \label{cor:construction}
  With the same assumptions as in Lemma~\ref{lemma:sharing}, and if $N$ is large enough to satisfy~\eqref{eq:N_large}, we can construct a permutation $\phi$ of $\{1,\ldots,N\}$ such that
  \begin{equation}
   \label{eq:phi_good}
   \forall P,P'\in\P,\ \left| \frac{1}{N} \sum_{1\le n\le N} \ind{P\times P'}(x_n,x_{\phi(n)}) - \lambda(P\times P') \right|
   \le 4 \eps.
  \end{equation}
 \end{Cor}

 \begin{proof}
  We use the family of numbers $\bigl(V(P\times P')\bigr)_{P\in\P,P'\in\P}$ provided by Lemma~\ref{lemma:sharing}.
  For a fixed atom $P\in\P$, we consider
  \begin{equation}
   \label{eq:defA}
     A(P):= \Bigl\{ n\in\{1,\ldots,N\}: x_n\in P \Bigr\}.
  \end{equation}
  Note that $|A(P)|=V(P)$.  Using \eqref{eq:C2}, we can find disjoint subsets
  \[ A(P\times P')\subset A(P),\quad P'\in \P, \]
  with $|A(P\times P')|=V(P\times P')$ for all $P'\in\P$. We denote
  \[ A:=\bigsqcup_{P,P'\in\P} A(P\times P'). \]

  Likewise, for any fixed $P'\in\P$, using \eqref{eq:C3} we can find disjoint subsets
  \[ A'(P\times P')\subset A(P'),\quad P\in \P, \]
  with $|A'(P\times P')|=V(P\times P')$ for all $P\in\P$.

  Then we can build a permutation $\phi$ of $\{1,\ldots,N\}$ as follows:
  \begin{itemize}
   \item For all $P,P'\in\P$, we define $\phi|_{A(P\times P')}$ as an arbitrary bijection from $A(P\times P')$ to $A'(P\times P')$.
   \item Then, we define $\phi|_{\{1,\ldots,N\}\setminus A}$ as an arbitrary bijection from the complement of $A$ to the complement of $\bigsqcup_{P,P'\in\mathscr{P}} A'(P\times P')$.
  \end{itemize}
  We observe that, with this choice of $\phi$, for all $n\in A$ and $P,P'\in\P$, we have
  \[  (x_n,x_{\phi(n)}) \in P\times P'\Longleftrightarrow n\in A(P\times P'). \]
  Therefore, for all $P,P'$ in $\P$,
  \begin{equation}
  \label{eq:sum_A}
   \sum_{n\in A} \ind{P\times P'}(x_n,x_{\phi(n)}) = |A(P\times P')| = V(P\times P').
  \end{equation}
  By \eqref{eq:C1}, for all $P,P'\in\P$, we have the inequality
  \[
     |A(P\times P')| = V(P\times P') \ge N\lambda(P\times P')(1-2\eps).
  \]
  Summing over all $P,P'$, we get
  \begin{equation}
  \label{eq:bound_A}
      |A| = \sum_{P,P'\in\P} V(P\times P') \ge N(1-2\eps).
  \end{equation}
Using~\eqref{eq:sum_A},
\eqref{eq:C1}  and \eqref{eq:bound_A}, we get
\begin{align*}
 &\left| \frac{1}{N} \sum_{1\le n\le N} \ind{P\times P'}(x_n,x_{\phi(n)}) - \lambda(P\times P') \right| \\
  \le\ &\left| \frac{1}{N} \sum_{n\in A} \ind{P\times P'}(x_n,x_{\phi(n)}) - \lambda(P\times P') \right| + \frac{N-|A|}{N}\\
  =\ &\left| \frac{V(P,P')}{N} - \lambda(P\times P') \right| + 1 - \frac{|A|}{N} \\
  \le\ & 2\eps \lambda(P\times P') + 2 \eps \le 4\eps,
\end{align*}
which is the announced inequality.
 \end{proof}

\subsection{Taking dynamics into account}
\label{sec:dynamics}
We now want to improve Corollary~\ref{cor:construction}, by taking into account the action of $S$ on $(X,\kappa)$. We will use the same strategy and the same notations as in the preceding section, but with some additional ingredients that make it possible to obtain the ``locally orbital'' condition on the permutations.

Here is the new setting. The measure $\lambda$ is a self-joining of $(X,\kappa,S)$ (compared to the previous hypotheses, we further assume that $\lambda$ is $(S\times S)$-invariant). We fix a finite partition $\Q$ of $X$, which at the end serves to estimate the gap between the empirical measure and the joining $\lambda$. But first we repeat the arguments of the preceding section with another partition $\P$, obtained as follows:
suppose that for some (large) integer $h$, we have a Rokhlin tower $\bigl(B,SB,\ldots,S^{h-1}B\bigr)$ in $(X,\kappa,S)$, which means that the subsets $B,SB,\ldots,S^{h-1}B$ are disjoint. We denote $F:=\bigsqcup_{0\le j\le h-1} S^j B$, and we assume that for some real number $\eps>0$, we have $\kappa(F)>1-\eps$. We consider the partition $\P$ whose atoms are $X\setminus F$, and all subsets of the form
\[ S^jB \cap \bigcap_{-j\le r\le h-1-j} S^{-r}Q_r, \]
for $0\le j\le h-1$, and $Q_j,\ldots,Q_{h-1-j}$ arbitrary atoms of $\Q$ (provided such a subset has positive $\kappa$-measure). In other words, the information provided by $\P$ is exactly: in which level of the Rokhlin tower the point is (or if it is in $X\setminus F$) and, when the point is in $F$, what is the $\Q$-name read on the piece of the orbit corresponding to the ascent in the Rokhlin tower. Note that for each $\P$-atom contained in $S^jB$ for some $0\le j\le h-2$, $SP$ is a $\P$-atom  contained in $S^{j+1}B$.

For some $N$ large enough to satisfy~\eqref{eq:N_large}, we have $N$ points $x_1,\ldots,x_N\in X$ which are successive images by $S$: $x_n=S^{n-1}x_1$ for all $2\le n\le N$. For each atom $P\in\P$ we define $V(P)$ and $A(P)$ as before (see~\eqref{eq:defV} and~\eqref{eq:defA}), and we assume that for the same number $\eps>0$, \eqref{eq:approximation} is satisfied. For technical reasons, we also assume that
\begin{equation}
 \label{eq:bord}
 x_1\notin\bigsqcup_{j=1}^{h-1}S^jB,\text{ and }x_N\notin\bigsqcup_{j=0}^{h-2}S^jB,
\end{equation}
so that for each $\P$-atom $P\subset\bigsqcup_{j=0}^{h-2}S^jB$,
\begin{equation}
 \label{eq:conservationV}
 V(SP) = V(P),
\end{equation}
and
\begin{equation}
 \label{eq:shiftA}
 A(SP) = A(P)+1.
\end{equation}

%

\begin{Lemma}
 \label{lemma:dynamics}
 Under the above hypotheses, we can construct a permutation $\phi$ of $\{1,\ldots,N\}$ satisfying
 \begin{equation}
  \label{eq:good_approximation_on_Q}
  \forall Q,Q'\in\Q,\quad \left| \frac{1}{N} \sum_{1\le n\le N} \ind{Q\times Q'}(x_n,x_{\phi(n)}) - \lambda(Q\times Q') \right|
   \le 8 \eps,
 \end{equation}
 and
 \begin{equation}
  \label{eq:locally_orbital_condition}
  \frac{1}{N} \Bigl| \Bigl\{ n\in\{1,\ldots,N-1\}: \phi(n+1) \neq \phi(n)+1 \Bigr\} \Bigr| \le 4\eps + \frac{2}{h}+ \frac{2}{N}.
 \end{equation}
\end{Lemma}

\begin{proof}
 As all the assumptions of Lemma~\ref{lemma:sharing} hold, we have at our disposal the numbers $V(P\times P')$ provided by this lemma. We will follow the same strategy as in the proof of Corollary~\ref{cor:construction} to get the permutation $\phi$, that is, we will construct the subsets $A(P\times P')$ and $A'(P\times P')$ satisfying all the above-mentioned properties, and define $\phi$ by its restriction to each $A(P\times P')$ as a bijection to $A'(P\times P')$.
 In the non-dynamical context of Corollary~\ref{cor:construction}, we have a significant flexibility to choose the subsets and the bijection. Here, we will also use the specific structure of the partition $\P$ to impose additional conditions leading to the ``locally orbital'' condition. To do so, we observe that the partition of $X$ defined by the Rokhlin tower $\bigl(B,SB,\ldots,S^{h-1}B\bigr)$ induces a partition of $X\times X$ into a family of disjoint Rokhlin towers for $S\times S$. These Rokhlin towers are of the form
 \[
  \Bigl( B\times S^jB, SB\times S^{j+1}B,\ldots,S^{h-1-j}B\times S^{h-1}B \Bigr)\quad (0\le j\le h-1),
 \]
and
 \[
  \Bigl( S^jB\times B, S^{j+1}B\times SB,\ldots,S^{h-1}B\times S^{h-1-j}B \Bigr)\quad (1\le j\le h-1).
 \]
 (See Figure~\ref{fig:towers}.)
 \begin{figure}\begin{center}
  \includegraphics[width=9cm]{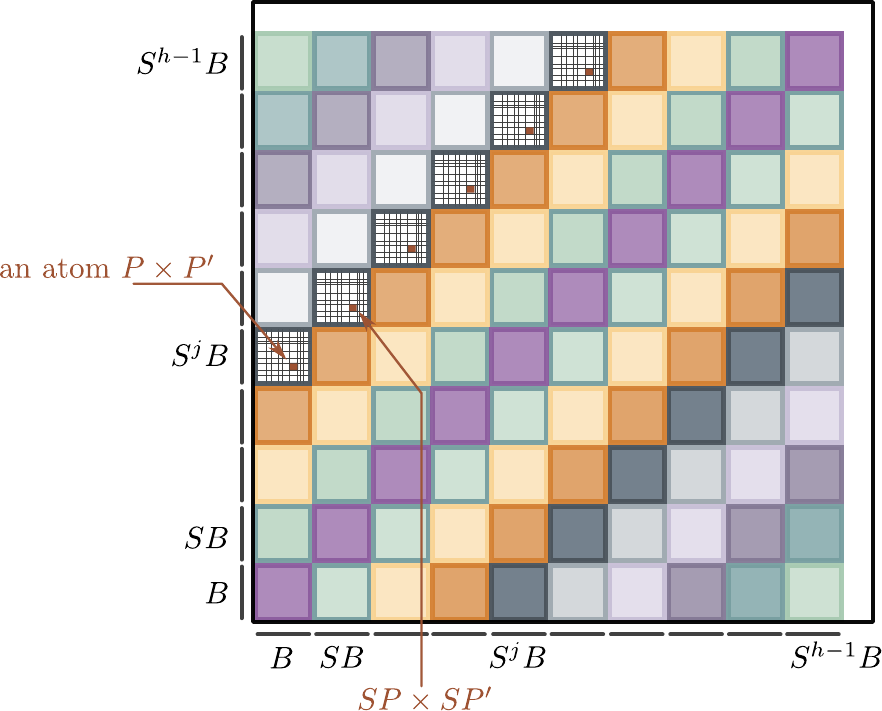}                                                          
  \end{center}
\caption{The partition of $X\times X$ into Rokhlin towers. We represented the atoms of $\P\times \P$ inside only one of these Rokhlin towers.}
 \label{fig:towers}
 \end{figure}

 We use this structure to define the subsets $A(P\times P')$ and $A'(P\times P')$, and the restriction of the desired permutation $\phi$ to $A(P\times P')$, in a specific order, so that the following condition holds: for all atoms $P,P'\in\P$, if for some $j,j'\in\{0,\ldots,h-2\}$ we have $P\subset S^jB$ and $P'\subset S^{j'}B$ (that is: $P\times P'$ is contained in a level of one of the $(S\times S)$-Rokhlin towers which is not the top one), then
 \begin{align}
 \label{eq:TI}
  \begin{split}
   A(SP\times SP') &= A(P\times P')+1,\\
   A'(SP\times SP') &= A'(P\times P')+1,\\
   \text{and }\forall n\in A(P\times P'),\ \phi(n+1) &= \phi(n)+1.
  \end{split}
 \end{align}
 Note that in the above requirements, the first two equalities are likely to be achieved since, by $S$-invariance of $\kappa$, $(S\times S)$-invariance of $\lambda$ and~\eqref{eq:conservationV},
 Condition \eqref{eq:C4} ensures that $V(SP\times SP')=V(P\times P')$.

 Here is how we proceed for the construction.
 \begin{itemize}
  \item First, we define the subsets $A(P\times P')$ and $A'(P\times P')$ for all atoms $P\times P'$ contained in the basis $B\times B$ of the highest Rokhlin tower, and we define the restriction of $\phi$ to any of these $A(P\times P')$ as an arbitrary bijection between $A(P\times P')$ and $A'(P\times P')$.
  \item Next, for the same  atoms $P\times P'$, we use Conditions~\eqref{eq:TI} to define inductively the subsets $A(S^jP\times S^jP')$ and $A'(S^jP\times S^jP')$, $(1\le j\le h-1)$, and the restrictions of $\phi$ to all of these $A(S^jP\times S^jP')$. At this point we have processed all atoms of $\P\times\P$ contained in the Rokhlin tower whose basis is $B\times B$.
  \item We proceed in the same way for the second Rokhlin tower (the one whose basis is $B\times SB$): for all $\P$-atoms $P\subset B$, $P'\subset SB$, we start by choosing the subsets $A(P\times P')$ and $A'(P\times P')$, ensuring that $A(P\times P')$ is disjoint from all $A(P\times P'_1)$, $P'_1\subset B$, previously chosen (this is always possible by~\eqref{eq:C2}), and that $A'(P\times P')$ is disjoint from all $A'(P_1\times P')$, $P_1\subset SB$, previously chosen (again, this is always possible by~\eqref{eq:C3}). Then we use Conditions~\eqref{eq:TI} to define inductively the subsets $A(S^jP\times S^jP')$ and $A'(S^jP\times S^jP')$, $(1\le j\le h-2)$, and the restrictions of $\phi$ to all of these $A(S^jP\times S^jP')$. Note that the disjunction at the base level implies disjunction at higher levels.
  \item We treat in this way all the Rokhlin towers one by one. For each tower successively, we first choose the subsets $A(P\times P')$ and $A'(P\times P')$ in the basis of the tower, taking care of ensuring all necessary disjunctions which is possible by~\eqref{eq:C2}  and~\eqref{eq:C3}, then we choose an arbitrary bijection between these two subsets. Once the basis is processed, we extend to all higher levels in the tower by~\eqref{eq:TI}.
 \end{itemize}
With this procedure, we choose all subsets $A(P\times P')$ and $A'(P\times P')$ and all bijections
\[
 \phi|_{A(P\times P')}: A(P\times P')\to A'(P\times P'),
\]
for all $\P$-atoms $P,P'\subset F$, ensuring that~\eqref{eq:TI} holds. Let us denote
  \[ A:=\bigsqcup_{\substack{P,P'\in\P \\ P\subset F,P'\subset F}} A(P\times P'). \]
Observe that, by~\eqref{eq:C1}, and using the assumption $\kappa(F)\ge 1-\eps$, we have
\begin{align}
\label{eq:bound_A2}
\begin{split}
 \frac{|A|}{N} &= \frac{1}{N} \sum_{\substack{P,P'\in\P \\ P\subset F,P'\subset F}} V(P\times P') \\
                &\ge (1-2\eps)\sum_{\substack{P,P'\in\P \\ P\subset F,P'\subset F}} \lambda(P\times P') \\
                &= (1-2\eps) \lambda(F\times F) \\
                &\ge (1-2\eps)^2 \ge 1-4\eps.
\end{split}
\end{align}

  To complete the picture, we define $\phi$ arbitrarily on $\{1,\ldots,N\}\setminus A$ to get a permutation of $\{1,\ldots,N\}$.

  With this choice of $\phi$, let us check the validity of~\eqref{eq:good_approximation_on_Q}. For this, we have to estimate the sum
  \[
    \sum_{1\le n\le N} \ind{Q\times Q'}(x_n,x_{\phi(n)}),
  \]
  in which we distinguish two types of $n$'s: those who are in $A$ and those who are not. By~\eqref{eq:bound_A2}, the contribution to the sum of $\{1,\ldots,N\}\setminus A$ is bounded by $4\eps N$. Now, for $n$ in $A$, there exist (unique) $\P$-atoms $P,P'\subset F$ such that $n\in A(P\times P')$, and then by the construction of $\phi$ we know that $(x_n,x_{\phi(n)})\in P\times P'$. Therefore, the contribution of this $n$ is 1 if and only if $P\subset Q$ and $P'\subset Q'$. It follows that the total contribution of $n$'s in $A$ to the sum amounts to
\[ \sum_{\substack{P\subset Q\cap F\\P'\subset Q'\cap F}} V(P\times P'). \]
Thus, using~\eqref{eq:C1}, we get
\begin{align*}
&\left| \frac{1}{N} \sum_{1\le n\le N} \ind{Q\times Q'}(x_n,x_{\phi(n)}) - \lambda(Q\times Q') \right|\\
\le &\frac{1}{N} \sum_{n \notin A} \ind{Q\times Q'}(x_n,x_{\phi(n)}) + \left| \frac{1}{N} \sum_{\substack{P\subset Q\cap F\\P'\subset Q'\cap F}} V(P\times P') - \lambda(Q\times Q') \right|\\
\le & 4\eps + \left|  \sum_{\substack{P\subset Q\cap F\\P'\subset Q'\cap F}} \left(\frac{V(P\times P')}{N} - \lambda(P\times P')\right) \right| + \lambda\bigl((Q\times Q')\setminus (F\times F)\bigr).\\
\end{align*}
By~\eqref{eq:C1}, the second term in the above sum is bounded by $2\eps$, and the same for the last term since $\kappa(F)\ge 1-\eps$. Thus,~\eqref{eq:good_approximation_on_Q} is established and it remains now to prove~\eqref{eq:locally_orbital_condition}.
By~\eqref{eq:TI}, the set
\[
 \Bigl\{ n\in\{1,\ldots,N-1\}: \phi(n+1) \neq \phi(n)+1 \Bigr\}
\]
is contained in the union of the following three subsets of $\{1,\ldots,N\}$:
\begin{itemize}
 \item $\{1,\ldots,N\}\setminus A$, whose cardinality is bounded by $4\eps N$ by~\eqref{eq:bound_A2};
 \item $\{n:x_n\in S^{h-1}B\}$, but since the $x_n$'s are successive points on some $S$-orbit, the gap between 2 integers in this set is always at most $h$, therefore the cardinality of this set is at most $\frac{N}{h}+1$;
 \item $\{n:x_{\phi(n)}\in S^{h-1}B\}$, which has the same cardinality as the preceding one because $\phi$ is a bijection.
\end{itemize}
The union of these 3 subsets has cardinality bounded by $4\eps N+\frac{2N}{h}+2$, which proves~\eqref{eq:locally_orbital_condition}.
\end{proof}

\subsection{Proof of Proposition~\ref{prop:selfjoiningsofFS}}

We assume now that~\eqref{eq:FS} holds: $\bfu\in X$ is quasi-generic along some increasing sequence $(N_k)$ for some Furstenberg system $\kappa$, $\lambda$ is a 2-fold self-joining of $\kappa$, and we explain the strategy to construct the locally orbital sequence $(\phi_k)$ announced in the statement of the proposition.

First, we point out that, without loss of generality, we may always assume that the measure-preserving system $(X,\kappa, S)$ is aperiodic. Indeed, if it is not the case, we fix some irrational number $\alpha$ such that for all $n\in \Z\setminus\{0\}$, $e^{2\pi in\alpha}$ is not an eigenvalue of the Koopman operator $f\mapsto f\circ S$ on $L^2(X,\kappa)$. We consider the auxilliary measure-preserving system $(Y,\nu,T_\alpha)$ where $Y:=\R/\Z$, $\nu$ is the normalized Haar measure and $T_\alpha:y\mapsto y+\alpha\mod 1$. As $(Y,T_\alpha)$ is uniquely ergodic, any point $y\in Y$ is generic for $\nu$ (we take for example $y=0$). Since $(X,\kappa, S)$ is disjoint from $(Y,\nu,T_\alpha)$, in the system $(X\times Y, S\times T_\alpha)$, the point $(\bfu,0)$ is quasi-generic along the same sequence $(N_k)$ for the product measure $\kappa\otimes\nu$. So, we can consider $(\bfu,0)\in X\times Y$ instead of $\bfu\in X$, and then the measure-preserving system $(X\times Y, \kappa\otimes \nu, S\times T_\alpha)$ is aperiodic. In this new setting we can extend the self-joining $\lambda$ to $\nu\otimes\lambda\otimes \nu$ on $(X\times Y)\times(X\times Y)$.

We fix a so-called ``good sequence of partitions'' $(\Q_\ell)_{\ell\ge1}$ for $(X,\kappa)$, which is a sequence of finite partitions satisfying:
\begin{itemize}
\item For all $\ell\ge1$, $\Q_{\ell+1}$ refines $\Q_\ell$,
\item $\max_{Q\in\Q_\ell} \diam(Q) \tend{\ell}{\infty} 0$,
\item $\forall \ell\ge1,\forall Q\in \Q_\ell$, $\kappa(\partial Q)=0$.
\end{itemize}
It is proved in~\cite{Ka-Ku-Le-Ru} that such a sequence always exists and that it satisfies the additional property: $(\Q_\ell\times\Q_\ell)$ is also a good sequence of partitions for $(X\times X,\lambda)$. In particular, weak$^*$-convergence of a sequence $(\lambda_n)$ of probability measures on $X\times X$ to $\lambda$ is equivalent to
\begin{equation}
 \label{eq:caract-wstar}
 \forall Q,Q'\in\Q,\quad \lambda_n(Q\times Q')\tend{n}{\infty} \lambda (Q\times Q').
\end{equation}

We also fix a sequence $(\eps_\ell)$ of positive numbers decreasing to $0$, and a sequence $(h_\ell)$ of integers such that $\frac{1}{h_\ell}<\eps_\ell$ for all $\ell$. It is also explained in~\cite{Ka-Ku-Le-Ru} that, for each $\ell$, we can find a Rokhlin tower $(B_\ell, SB_\ell,\ldots,S^{h_\ell-1}B_\ell)$ such that
\begin{itemize}
 \item $\kappa\left(\bigsqcup_{0\le j\le h_\ell-1}S^jB\right)>1-\eps_\ell$,
 \item for all $j$, $\kappa\bigl(\partial(S^jB)\bigr)=0$.
\end{itemize}
Then, for each $\ell$, we construct the partition $\P_\ell$ from $\Q_\ell$ and the above Rokhlin tower, in the same way as $\P$ is constructed from $\Q$ and the Rokhlin tower $(B,\ldots,S^{h-1}B)$ in the beginning of Section~\ref{sec:dynamics}. Since all atoms $P$ of these partitions $\P_\ell$ satisfy $\kappa(\partial P)=0$, \eqref{eq:FS} ensures that
\[
 \forall\ell,\forall P\in\P_\ell,\quad \frac{1}{N_k} \sum_{1\le n\le N_k} \ind{P}(S^n\bfu) \tend{k}{\infty} \kappa(P).
\]
This enables us to define
\begin{multline*}
 k_1 := \min\Biggl\{
                K\ge1 : \forall k\ge K, \\
                \eqref{eq:N_large}\text{ holds for }N\ge N_k-2h_1,\ \eps=\eps_1\text{ and }\P=\P_1,\\
                \text{ and }\forall P\in\P_1,\
               \frac{4h_1}{N_k}  + \left| \frac{1}{N_k}\sum_{1\le n\le N_k} \ind{P}(S^n\bfu) - \kappa(P) \right|\le \eps_1
                \Biggr\},
\end{multline*}
and inductively, for all $\ell\ge2$,
\begin{multline*}
 k_{\ell} := \min\Biggl\{
                K\ge k_{\ell-1}+1 : \forall k\ge K, \\
                \eqref{eq:N_large}\text{ holds for }N\ge N_k-2h_\ell,\ \eps=\eps_\ell\text{ and }\P=\P_\ell,\\
                \text{ and }\forall P\in\P_\ell,\
               \frac{4h_{\ell}}{N_k} + \left| \frac{1}{N_k}\sum_{1\le n\le N_k} \ind{P}(S^n\bfu) - \kappa(P) \right|  \le \eps_{\ell}
                \Biggr\}.
\end{multline*}

Consider an integer $k$ with $k_\ell\le k<k_{\ell+1}$ for some $\ell\ge1$. We want to use Lemma~\ref{lemma:dynamics} to construct a permutation $\phi_k$ of $\{1,\ldots,N_k\}$, and for this we have to precise with which points $x_n$ we apply the lemma. With the goal of fulfilling Property~\eqref{eq:bord}, we set
\[ x_1:=S^{i_1}\bfu,\text{ where }i_1:=\min\left\{i\ge1:S^i\bfu\notin\bigsqcup_{1\le j\le h_\ell-1}S^j B_\ell\right\}. \]
We have $i_1\le h_\ell$. Then, we also consider
\[
 i_2:=\max \left\{i\le N_k:S^i\bfu\notin\bigsqcup_{0\le j\le h_\ell-2}S^j B_\ell\right\}.
\]
We also have $N_k-i_2\le h_\ell-1$. Thus, setting $N:=i_2-i_1+1$, we have $N_k-2h_\ell\le N\le N_k$. Now, the points $x_n=S^{n+i_1-1}\bfu$, $1\le n\le N$, satisfy for each atom $P\in\P_\ell$,
\begin{align*}
  &\left| \frac{1}{N}\sum_{1\le n\le N} \ind{P}(x_n) - \kappa(P) \right| \\
  \le &  \left| \frac{1}{N}\sum_{1\le n\le N} \ind{P}(x_n)  - \frac{1}{N_k}\sum_{1\le n\le N_k} \ind{P}(S^n\bfu) \right|
  + \left| \frac{1}{N_k}\sum_{1\le n\le N_k} \ind{P}(S^n\bfu) - \kappa(P) \right|\\
  \le & \frac{4h_\ell}{N_k} + \left| \frac{1}{N_k}\sum_{1\le n\le N_k} \ind{P}(S^n\bfu) - \kappa(P) \right| \le \eps_\ell.
\end{align*}
All assumptions of Lemma~\ref{lemma:dynamics} are satisfied, so we get a permutation $\phi$ of $\{i_1,\ldots,i_2\}$ such that
\begin{equation*}
  \forall Q,Q'\in\Q_\ell,\quad \left| \frac{1}{i_2-i_1+1} \sum_{i_1\le n\le i_2} \ind{Q\times Q'}(S^n\bfu,S^{\phi(n)}\bfu) - \lambda(Q\times Q') \right|
   \le 8 \eps_\ell,
 \end{equation*}
 and
 \begin{equation}
  \label{eq:locorb}
  \frac{1}{N} \Bigl| \Bigl\{ n\in\{i_1,\ldots,i_2\}: \phi(n+1) \neq \phi(n)+1 \Bigr\} \Bigr| \le 4\eps_\ell + \frac{2}{h_\ell}+ \frac{2}{N_k}.
 \end{equation}
Then, we can extend $\phi$ to a permutation $\phi_k$ of $\{1,\ldots,N_k\}$, for example by setting $\phi_k|_{\{1,\ldots,i_1-1\}\cup\{i_2+1,\ldots,N_k\}}:=\Id$.

It is clear that the sequence $(\phi_k)$ constructed in this way satisfies for all $\ell\ge1$,
\[
 \forall Q,Q'\in\Q_\ell,\quad \left| \frac{1}{N_k} \sum_{1\le n\le N_k} \ind{Q\times Q'}(S^n\bfu,S^{\phi(n)}\bfu) - \lambda(Q\times Q') \right| \tend{k}{\infty} 0.
\]
Thanks to~\eqref{eq:caract-wstar}, this gives~\eqref{eq:selfjoinings_of_FS}.
Moreover, \eqref{eq:locorb} ensures that the sequence of permutations $(\phi_k)$ is locally orbital.
This concludes the proof of Proposition~\ref{prop:selfjoiningsofFS}.

\subsection{Question}\label{s:question}

What could be an analog of Proposition~\ref{prop:selfjoiningsofFS} for the {\bf logarithmic} Furstenberg systems? An important difficulty seems to be that, in this case, all points in the orbit of $\bfu$ do not have the same weights, hence the permutation of indices changes the second marginal.

\section{Cross-sections of ergodic components}\label{s:RelErg}

The purpose of this section is to provide a proof of Proposition~\ref{p:tim1}. The arguments presented below were communicated to us by Tim Austin.

\subsection{Measure theory}

%
%
%

Consider a diagram of measurable spaces
\begin{center}
$\phantom{i}$\xymatrix{ & X \ar_\phi[dl] \ar^\psi[dr]\\
Y && Z.
}
\end{center}
Let $\mu$ be a probability measure on $X$, and $\mu' :=\phi_\ast\mu$ be its image on $Y$.  Let $\ca$ be the sigma-algebra of $Y$, and assume that $Z$ is standard.  Let $m$ be Lebesgue measure restricted to $\cb_{[0,1]}$.

\begin{Lemma}\label{lem:rep}
There is a measurable map $\pi$ from $Y\times [0,1]$ to $Z$ that represents the joint distribution of $\phi$ and $\psi$ in the following sense:
\[\int_X (f\circ \phi)(g\circ \psi)\,d\mu = \int_Y f(y) \int_0^1 g(\pi(y,t))\ d t\ d\mu'(y)\]
for any bounded measurable functions $f$ on $Y$ and $g$ on $Z$.

Moreover, $\pi$ can be taken measurable with respect to $\ca_0\times \cb_{[0,1]}$, where $\ca_0$ is some countably generated sigma-subalgebra of $\ca$.
\end{Lemma}

This follows directly from two classical results in the measure theory of standard Borel spaces; we refer the reader to Kallenberg's book~\cite{Kal}.

\begin{proof}
First, since $Z$ is standard, we may represent the conditional distribution of $\psi$ given $\phi$ using a probability kernel, say
\[Y\mapsto M(Z):y\mapsto \nu_y.\]
This is the disintegration theorem, as in~\cite[Theorem 8.5]{Kal}.

Since $M(Z)$ is standard, its sigma-algebra is countably generated.  The pullback of this sigma-algebra to $Y$ is a countably generated sigma-subalgebra of $\ca$.  Let this be $\ca_0$.

Again since $Z$ is standard, the kernel $\nu_\bullet$ has a representation using an independent random variable distributed uniformly in $[0,1]$: see~\cite[Lemma 4.22]{Kal}.  This result gives a map $\pi:Y\times [0,1]\to Z$ that is measurable with respect to $\ca_0\times \cb_{[0,1]}$ and satisfies
\[\nu_y(C) = m\{t:\ \pi(y,t) \in C\} \quad\forall y \in Y,\ C \in \cb_Z.\]

Putting these two representations together, we obtain
\begin{align*}
\int_X (f\circ \phi)(g\circ \psi)\,d\mu &= \int_Y f(y)\int g\ d\nu_y\ d\mu'(y) \\ &= \int_Y f(y) \int_0^1 g(\pi(y,t))\ d t\ d\mu'(y),
\end{align*}
as required.
\end{proof}

\subsection{Ergodic theory}
Now let $T$ be an automorphism of a standard probability space $(X,\cb_X,\mu)$. We recall that $\ci_T$ stands for the sub-sigma-algebra of $T$-invariant sets .  Let
\[\mu = \int \mu_x\ d\mu(x)\]
be the ergodic decomposition of $(X,\mu,T)$: that is, the disintegration of $\mu$ over $\ci_T$.

Let $Z:=M^e(X,T)$.  This is a Borel subset of $M(X)$, so the induced measurable structure on $Z$ is still standard.  The map $x\mapsto \mu_x$ is $T$-invariant and measurable from $X$ to $Z$.

Now let $(Y,\ca)$ be another measurable space, not necessarily standard.  Let $\phi:X\to Y$ be a measurable map with the property that $\phi^{-1}\ca \subseteq \ci_T$.  (In case $\ca$ separates the points of $Y$, this implies that $\phi$ is $T$-invariant, but we do not need this additional assumption.)  In applications it could be that $Y$ is the same set as $X$, $\ca$ is a proper sub-sigma-algebra of $\ci_T$, and $\phi(x) = x$. However, the explanation below seems clearer if we give $Y$ its own name.  Let $\mu' := \phi_\ast\mu$ be the image measure of $\mu$ under $\phi$.

We are now ready to prove Proposition~\ref{p:tim1}.

\noindent
{\em Proof of Proposition~\ref{p:tim1}}.
Begin by applying Lemma~\ref{lem:rep} to the maps $\phi:X\to Y$ and $\mu_\bullet:X\to Z$.  The result is a map from $Y\times [0,1]$ to $Z$ satisfying the equation promised in that lemma, and measurable with respect to $\ca_0\times \cb_{[0,1]}$ for some countably generated sigma-subalgebra $\ca_0$ of $\ca$.  Since $Z$ is itself a space of measures, this new map is actually a probability kernel, say $\theta_{(y,t)}$.

Let us consider the conclusion of Lemma~\ref{lem:rep} in case $f$ is a bounded measurable function on $Y$ and $g$ is a function on $Z$ of the form
\[g(\nu) := \int G\,d\nu\]
for some bounded Borel function $G$ on $X$.  Then that conclusion becomes
\begin{equation}\label{eq:repagain}
\int f(\phi(x))\Big(\int G\,d\mu_x\Big)\ d\mu(x) = \int_Y f(y)\int_0^1 \Big(\int G\,d\theta_{(y,t)}\Big)\ d t\ d\mu'(y).
\end{equation}
We make use of~\eqref{eq:repagain} through two special cases:
\begin{enumerate}
\item Taking $f = 1$ but allowing arbitrary $G$, equation~\eqref{eq:repagain} shows that
\begin{equation}\label{eq:disint}
\mu = \int \mu_x\ d\mu(x) = \int_Y \int_0^1 \theta_{(y,t)}\ d t\ d\mu'(y).
\end{equation}
\item Take $f = \ind{A}$ for some $A \in \ca$, and take $G = \ind{B}$ where $B$ is either $\phi^{-1}A$ or $X\setminus \phi^{-1}A$.  Equation~\eqref{eq:repagain} becomes
\begin{equation}\label{eq:repagain2}
\int_{\phi^{-1}A} \mu_x(B)\,d\mu(x) = \int_A \int_0^1 \theta_{(y,t)}(B)\ d t\ d\mu'(y).
\end{equation}
Since $B$ is $T$-invariant, the properties of the ergodic decomposition give
\[\mu_x(B) = \ind{B}(x) \quad \hbox{for}\ \mu\hbox{-a.a.}\ x,\]
so the left-hand side of~\eqref{eq:repagain2} is either $\mu(\phi^{-1}A) = \mu'(A)$ (in case $B = \phi^{-1}A$) or $0$ (in case $B = X\setminus \phi^{-1}A$).  Allowing $A$ to vary and looking at the right-hand side of~\eqref{eq:repagain2}, these outcomes are possible only if we have
\begin{equation}\label{eq:const}
\theta_{(y,t)}(\phi^{-1}A) = \ind{A}(y) \quad \hbox{for}\ (\mu'\times m)\hbox{-a.a.}\ (y,t).
\end{equation}
Equation~\eqref{eq:const} plays an important role below, but we need something slightly stronger.  Since $\ca_0$ is countably generated, it contains a countable generating subalgebra of sets ${\cal{S}}$.  Since ${\cal{S}}$ is countable, there is a Borel set $F \subseteq [0,1]$ with $m(F) = 1$ and such that~\eqref{eq:const} holds for $\mu'$-a.a.\ $y$ whenever $t \in F$ and $A \in {\cal{S}}$: that is, the set of `good' values of $t$ does not depend on the choice of $A$ from ${\cal{S}}$.  This conclusion now extends from ${\cal{S}}$ to $\ca_0$.  Indeed, let $t \in F$ and let $Y_t \in \ca_0$ be the set of all $y \in Y$ such that
\[\theta_{(y,t)}(\phi^{-1}A) = \ind{A}(y)\quad \forall A\in{\cal{S}}.\]
If $A \in \ca_0$ and $y \in Y_t\cap A$, then
\beq\label{TAdow}A \supseteq \bigcap_{A' \in {\cal{S}}:\ A' \ni y}A'\eeq
Indeed, let us say that two points $u,v$ in $Y$ are ``the same according to ${\cal S}$'' if every set in ${\cal S}$ either contains both $u$ and $v$, or neither of them.  It is easy to see that if $u,v$ are the same according to ${\cal S}$, then they are the same according to the sigma-algebra $\ca_0$ generated by  ${\cal S}$ (indeed: the collection of all sets according to which $u$ and $v$ are the same is a sigma-algebra containing ${\cal S}$).
Now, the set on the right is the set of all points in $Y$ that are the same as $y$ according to ${\cal S}$.  By the observation above, this is the same as ``all points in $Y$ that are the same as $y$ according to $\ca_0$''.  Since $A$ is one particular member of $\ca_0$, and $y$ lies in $A$, all of these other points must also lie in $A$.

Since in~\eqref{TAdow} we deal with a countable intersection, we have
\[\theta_{(y,t)}(\phi^{-1}A) \ge \theta_{(y,t)}\Big(\bigcap_{A' \in {\cal{S}}:\ A' \ni y}A'\Big) = 1.\]
A symmetrical argument shows that $\theta_{(y,t)}(\phi^{-1}A) = 0$ if $y \in Y_t\setminus A$.  Thus, for $t \in F$, we now know that
\begin{equation}\label{eq:const2}
\theta_{(y,t)}(\phi^{-1}A) = \ind{A}(y) \quad \forall A \in \ca_0,\ \hbox{for}\ \mu'\hbox{-a.a.}\ y.
\end{equation}
\end{enumerate}

For each $t \in F$, let
\[\nu_t := \int_Y \theta_{(y,t)}\ d\mu'(y).\]
For $t \in [0,1]\setminus F$, take $\nu_t$ to equal $\nu_s$ for some $s \in F$; the choice does not matter because the set of these $t$ is negligible.

Now, we can show properties (a--c). The first two are simple: (a) holds because $\nu_t$ is a mixture of $T$-invariant measures; and (b) follows from~\eqref{eq:disint} and Fubini's theorem.

We have $\ci_T \supseteq \phi^{-1}\cal{A}$ by assumption, so property (c) requires only the reverse inclusion modulo $\nu_t$.  Suppose that $B \in \ci_T$, and let
\[A := \{y \in Y:\ \theta_{(y,t)}(B) = 1\}.\]
Then $A \in \ca_0$, by the measurability of $\theta$ and because we are holding $t$ fixed.  Moreover, the value $\theta_{(y,t)}(B)$ equals $0$ or $1$ for every $(y,t)$, since $\theta_{(y,t)}$ is ergodic, and therefore
\begin{equation}\label{eq:complement}
Y\setminus A = \{y \in Y:\ \theta_{(y,t)}(B) = 0\}.
\end{equation}

Finally, we have
\begin{align*}
\nu_t(B\setminus \phi^{-1}A) &= \int_Y \theta_{(y,t)}(B\setminus \phi^{-1}A)\ d\mu'(y)\\
&= \int_A \theta_{(y,t)}(B\setminus \phi^{-1}A)\ d\mu'(y) + \int_{Y\setminus A} \theta_{(y,t)}(B\setminus A)\ d\mu'(y).
\end{align*}
The first of these integrals is zero because~\eqref{eq:const2} gives
\[\theta_{(y,t)}(\phi^{-1}A) = 1 \quad \hbox{for}\ \mu'\hbox{-a.a.}\ y \in A,\]
and the second integral is zero because
\[\theta_{(y,t)}(B) = 0 \quad \hbox{for}\ \mu'\hbox{-a.a.}\ y \in Y\setminus A.\]
Similarly, $\nu_t(\phi^{-1}A\setminus B) = 0$, and so $B = \phi^{-1}A$ modulo $\nu_t$.
\bez

\section{Pinsker factor in ergodic components}\label{s:PpEC}
 Let $(X,\mu,T)$ be a measure-preserving system, where $(X,\mu)$ is a standard Borel probability space and $T:X\to X$ is invertible and preserves $\mu$. We recall that $\ci_T$ denotes the factor sigma-algebra of $T$-invariant subsets. Recalling that $M^e(X,T)$ stands for the set of ergodic $T$-invariant probability measures on $(X,\cb_X)$, we can view the ergodic decomposition of $\mu$ as a measurable map $x\mapsto \mu_x$ from $X$ to $M^e(X,T)$ such that, for all $f\in L^1(\mu)$,
\begin{equation}
 \label{eq:ergdec}
 \EE_\mu [f\,|\,\ci_T] = \int_X f\,d\mu_x \quad\text{($\mu$-a.e.)}
 \end{equation}

 We denote by $\Pi(X,\mu,T)$ the Pinsker factor sigma-algebra of $(X,\mu,T)$, that is the sub-sigma-algebra of all Borel subsets $A$ of $X$ such that $h_\mu(\P_A,T)=0$, where $\P_A$ is the partition of $X$ into $A$ and $X\setminus A$. For each ergodic component $\mu_x$, we also have an associated Pinsker factor sigma-algebra $\Pi(X,\mu_x,T)$. The following proposition clarifies the relationship between these Pinsker factors.

 \begin{Prop}
  \label{prop:Pinsker}
  There exists a countably generated sub-sigma-algebra $\cc\subset\cb_X$ such that
  \begin{itemize}
   \item $\Pi(X,\mu,T)=\cc \mod \mu$;
   \item for almost every ergodic component $\mu_x$, $\Pi(X,\mu_x,T)=\cc \mod \mu_x$.
  \end{itemize}
  Moreover, for all $f\in L^1(\mu)$, we have $\mu$-almost surely
\begin{equation}
 \label{eq:equalityPinsker}
 \EE_\mu [f\,|\,\cc](x) = \EE_{\mu_x}[f\,|\,\cc](x).
\end{equation}
\end{Prop}

For a factor sub-sigma-algebra $\ca\subset\cb_X$, we can also consider the corresponding Pinsker factor sub-sigma-algebra  $\Pi(X,\ca,\mu,T)$ which obviously satisfies:
$$ \Pi(X,\ca,\mu,T)= \Pi(X,\mu,T)\cap \ca.$$

 We recall the following classical result: if $\P$ is a \textbf{finite} partition of $X$, and if
 $\ca=\bigvee_{n\in\ZZ}T^{-n}\P$,
 then
 \begin{equation}
  \label{eq:pinsker}
  \Pi(X,\ca,\mu,T)=\bigcap_{m\in\ZZ^-}\bigvee_{n\le m}T^{-n}\P \mod \mu.
 \end{equation}

 We will also need two other facts, which we state as lemmas below.

 \begin{Lemma}
 \label{lemma:pinskersup}
  Let $\ca$ and $\cb$ be two factor sub-sigma-algebras, and assume that $\cb\subset \Pi(X,\mu,T)$. Then
  $$ \Pi\bigl(X,\ca\vee\cb,\mu,T\bigr) = \Pi(X,\ca,\mu,T)\vee\Pi(X,\cb,\mu,T) = \Pi(X,\ca,\mu,T)\vee\cb\mod\mu. $$
 \end{Lemma}

 (Note that the result is not true in general if we do not assume that one of the two factors has zero entropy:  any zero entropy system can be seen as a factor of a joining of two Bernoulli shifts, cf.\ \cite{Sm-Th}.)

 \begin{proof}
  The inclusion $\Pi(X,\ca,\mu,T)\vee\Pi(X,\cb,\mu,T) \subset \Pi\bigl(X,\ca\vee\cb,\mu,T\bigr)$ is clear, and does not require that one of the two factors be of zero entropy.
  For the reverse inclusion, let us consider $f\in L^2(X,\ca,\mu)$, $g\in L^2(X,\cb,\mu)$, and let $h$ be bounded and measurable with respect to $\Pi(X,\ca\vee\cb,\mu,T)$. Since the factor generated by $g$ and $h$ has zero entropy, using Theorem~\ref{t:largest}, we have
  $$ \EE_\mu [ f g h ] = \EE_\mu \bigl[ \EE_\mu[f\,|\,\Pi(X,\ca,\mu,T)] \, g h \bigr], $$
  and this shows that
  $$ \EE_\mu \bigl[ fg \,|\, \Pi(X,\ca\vee\cb,\mu,T) \bigr] =  \EE_\mu[f\,|\,\Pi(X,\ca,\mu,T)] \, g \quad \mu\text{-a.e.} $$
  In particular, $\EE_\mu \bigl[ fg \,|\, \Pi(X,\ca\vee\cb,\mu,T) \bigr]$ is measurable with respect to $$\Pi(X,\ca,\mu,T)\vee\Pi(X,\cb,\mu,T).$$ This remains true if we replace $fg$ by a finite linear combination of $f_ig_i$'s, where each $f_i\in L^2(X,\ca,\mu)$ and each $g_i\in L^2(X,\cb,\mu)$. Then by the density of these linear combinations we get that $\EE_\mu \bigl[ h \,|\, \Pi(X,\ca\vee\cb,\mu,T) \bigr]$ is measurable with respect to $\Pi(X,\ca,\mu,T)\vee\Pi(X,\cb,\mu,T)$ for any $h\in L^2(X,\ca\vee\cb,\mu)$. Now, we conclude by taking $h$ measurable with respect to $\Pi(X,\ca\vee\cb,\mu,T)$.
 \end{proof}

 \begin{Lemma}
 \label{lemma:pinskerlimit}
  Let $(\ca_k)_{k\in\NN}$ be an increasing sequence of factor sub-sigma-algebras, and let $\ca:=\bigvee_{k\in\NN}\ca_k$.
  Then
  $$ \Pi(X,\ca,\mu,T) = \bigvee_{k\in\NN} \Pi(X,\ca_k,\mu,T) \mod \mu .$$
 \end{Lemma}

\begin{proof}
 Again, the inclusion $\bigvee_{k\in\NN} \Pi(X,\ca_k,\mu,T) \subset \Pi(X,\ca,\mu,T)$ is clear.
 For the converse, let us take $f\in L^2(X,\Pi(X,\ca,\mu,T),\mu)$. Since $f$ is $\ca$-measurable, the martingale theorem tells us that, in $L^2(\mu)$,
 $$ f = \lim_{k\to\infty} \EE_\mu[f\,|\,\ca_k]. $$
 But $f$ is measurable with respect to a zero-entropy factor, thus, in view of Theorem~\ref{t:largest}, for any $k\in\NN$ and any $g\in L^2(X,\ca_k,\mu)$, we have
 $$ \EE_\mu [  fg  ] = \EE_\mu \bigl[  f\,\EE[g\,|\,\Pi(X,\ca_k,\mu)]  \bigr]
, $$
 and it follows that
 $$
 \EE_\mu[f\,|\,\ca_k] = \EE_\mu[f\,|\,\Pi(X,\ca_k,\mu)]\quad \mu\text{-a.e.}
 $$
 Therefore, $f$ is the $L^2$-limit of a sequence $(f_k)$, where each $f_k$ is $L^2(\Pi(X,\ca_k,\mu))$-measurable.
\end{proof}

Of course, the preceding results are also true if we replace $\mu$ by any ergodic component $\mu_x$.

\begin{Remark}
 Lemmas~\ref{lemma:pinskerlimit} and~\ref{lemma:pinskersup} remain valid if we replace everywhere the Pinsker factor by the largest $\cf$-factor for an arbitrary characteristic class $\cf$. Indeed, their proofs only use the fact that ZE is a characteristic class, and that the Pinsker factor is the largest ZE-factor.
\end{Remark}

We will also need the following result about conditional expectations.

\begin{Lemma}
\label{lemma:c_exp}
 For any bounded measurable function $\phi$ and any finite partition $\Q$ of $X$, we have for $\mu$-almost every $x\in X$
 \begin{equation}
  \label{eq:mux}
  \EE_\mu\left[ \phi \, \left|\, \ci_T \vee \Q \right.\right](x) =  \EE_{\mu_x}\left[ \phi \, \left|\, \Q \right.\right](x).
 \end{equation}
 \end{Lemma}

 \begin{proof}
Since $\Q$ is a finite partition, the RHS is
\beq\label{mux1} \sum_{\substack{B\text{ atom of }\Q,\\ \mu_x(B)>0}} \ind{B}(x) \frac{1}{\mu_x(B)} \int_B \phi\,d\mu_x. \eeq
As $x\mapsto \mu_x$ is $\ci_T$-measurable, the above function of $x$ is $\ci_T\vee\Q$-measurable. To verify that it corresponds to the conditional expectation on the LHS of~\eqref{eq:mux}, we have to multiply it (i.e.\ \eqref{mux1}) by $\ind{C}\, h$ for some atom $C$ of $\Q$ and some bounded $\ci_T$-measurable function $h$, and integrate with respect to $\mu$. We get
\begin{align*}
 \int_X \ind{C}(x)\, h(x) \, \EE_{\mu_x}\left[ \phi \, \left|\, \Q \right.\right](x)\, d\mu(x)
   &= \int_X  \left(\ind{C}(x) \,  h(x) \, \frac{1}{\mu_x(C)} \int_C \phi\,d\mu_x\right)\, d\mu(x)
\end{align*}
But $h$ and $x\mapsto \mu_x$ are both $\ci_T$-measurable, therefore in the last integral we can replace $\ind{C}(x)$
by $\EE_\mu[\ind{C}\,|\,\ci_T](x)$, which is equal to $\mu_x(C)$ for $\mu$-almost every $x$. After cancellation with the denominator, we are left with
$\int_X  \left( h(x) \int_C \phi\,d\mu_x\right)\, d\mu(x)$, which is equal to
$$ \int_X h(x)\, \EE_\mu[\ind{C}\, \phi\,|\,\ci_T](x)\, d\mu(x) =  \int_X h(x)\, \ind{C}(x)\, \phi(x)\,d\mu(x). $$
This achieves the proof of the lemma.
\end{proof}

\begin{proof}[Proof of Proposition~\ref{prop:Pinsker}]
 We fix a countable family $(A_k)_{k\in\NN}$ of Borel subsets of $X$ that separates points, and we call $\P_k$ be the finite partition of $X$ generated by $A_1,\ldots,A_k$. It follows that for any probability measure $\nu$ on $X$, we have
 $$\cb_X=\bigvee_{k\in\NN}\P_k\mod\nu.$$
 We define an increasing sequence $(\ca_k)_{k\in\NN}$ of factor sub-sigma-algebras by setting
 \beq\label{dodc} \ca_k := \ci_T\vee \bigvee_{n\in\ZZ}T^{-n}\P_k. \eeq

 We also fix a countable set $\Phi=(\phi_j)_{j\in\NN}$ of bounded, Borel-measurable functions on $X$ satisfying the following:
 for any probability measure $\nu$ on $X$, $\Phi$ is dense in $L^2(X,\nu)$. There are several ways to get such a set: if $X$ is compact, we can take a countable dense set in $C(X)$. We can also take $\Phi$ as the set of all finite linear combinations with rational coefficients of $\ind{A_i}$ where the $A_i$'s are atoms of a partition $\P_k$ for some $k$.

 We first establish the conclusion of the proposition for $\Pi(X,\ca_k,\mu,T)$ and $\Pi(X,\ca_k,\mu_x,T)$.
 As $\Phi$ is dense in $L^2(X,\mu)$, the Pinsker factor $\Pi(X,\ca_k,\mu,T)$ coincides modulo $\mu$ with the sigma-algebra generated by all conditional expectations
 of the form $\EE_\mu[\phi\,|\,\Pi(X,\ca_k,\mu,T)]$ ($\phi\in\Phi$).
 From~\eqref{eq:pinsker},~\eqref{dodc} and Lemma~\ref{lemma:pinskersup}, we get
 $$ \Pi(X,\ca_k,\mu,T) = \ci_T \vee \bigcap_{m\in\ZZ^-}\bigvee_{n\le m}T^{-n}\P_k \mod\mu. $$
 For integers $\ell\le m\le 0$, let us denote $\Q_{\ell,m}$ the finite partition $\bigvee_{\ell\le n\le m}T^{-n}\P_k$.
 By application of the reverse martingale theorem and the martingale theorem, we have
 \begin{equation}
  \label{eq:convpinsker}
  \EE_\mu\bigl[\phi\,|\,\Pi(X,\ca_k,\mu,T)\bigr] = \lim_{m\to -\infty} \lim_{\ell\to -\infty} \EE_\mu\left[ \phi \ \left|\ \ci_T \vee \Q_{\ell,m}\right.\right],
 \end{equation}
 where all limits are pointwise, and exist for $\mu$-a.a.\ $x\in X$.

 We can make the same analysis in the ergodic system $(X,\nu ,T)$ for any $T$-invariant ergodic measure $\nu $: the Pinsker factor $\Pi(X,\ca_k,\nu ,T)$ coincides modulo $\nu $ with the sigma-algebra generated by all conditional expectations
 of the form $\EE_{\nu }[\phi\,|\,\Pi(X,\ca_k,\nu ,T)]$ when $\phi$ runs over $\Phi$. Again, from~\eqref{eq:pinsker} and Lemma~\ref{lemma:pinskersup}, we get
 $$ \Pi(X,\ca_k,\nu ,T) = \ci_T \vee \bigcap_{m\in\ZZ^-}\bigvee_{n\le m}T^{-n}\P_k = \bigcap_{m\in\ZZ^-}\bigvee_{n\le m}T^{-n}\P_k\mod\nu . $$
 (For the last equality, we used the ergodicity of $\nu $ which ensures that $\ci_T$ is trivial under $\nu $.) For any $\phi\in\Phi$, we then have $\nu $-a.e.
 $$
 \EE_{\nu }\bigl[\phi\,|\,\Pi(X,\ca_k,\nu ,T)\bigr] = \lim_{m\to -\infty} \lim_{\ell\to -\infty} \EE_{\nu }\left[ \phi \ \left|\ \Q_{\ell,m}\right.\right].
 $$
 In particular, the above is true for $\nu=\mu_y$, for $\mu$-almost all $y$, and we can write this by making the variables explicit: for $\mu$-almost all $y$, $\mu_y$ gives full measure to the set of $x\in X$ satisfying
 $$
 \EE_{\mu_y }\bigl[\phi\,|\,\Pi(X,\ca_k,\mu_y ,T)\bigr] (x) = \lim_{m\to -\infty} \lim_{\ell\to -\infty} \EE_{\mu_y }\left[ \phi \ \left|\ \Q_{\ell,m}\right.\right](x).
 $$
 But we can also observe that,  for $\mu$-almost all $y$, $\mu_y$ gives full measure to the set of $x\in X$ satisfying $\mu_x=\mu_y$. We get: for $\mu$-almost all $y$, $\mu_y(A)=1$, where
 $$ A:= \Bigl\{x\in X: \EE_{\mu_x }\bigl[\phi\,|\,\Pi(X,\ca_k,\mu_x ,T)\bigr] (x) = \lim_{m\to -\infty} \lim_{\ell\to -\infty} \EE_{\mu_x }\left[ \phi \ \left|\ \Q_{\ell,m}\right.\right](x)\Bigr\}. $$
It follows that the above set $A$ satisfies $\mu(A)=1$.
Now, take $x\in A$ such that the conclusion of Lemma~\ref{lemma:c_exp} is true at $x$ for all partitions $\Q_{\ell,m}$, and for which~\eqref{eq:convpinsker} holds (the set of such $x$ has $\mu$-measure 1). Then, we have
\begin{align*}
 \EE_{\mu_x }&\bigl[\phi\,|\,\Pi(X,\ca_k,\mu_x ,T)\bigr] (x) \\
 & = \lim_{m\to -\infty} \lim_{\ell\to -\infty} \EE_{\mu_x }\left[ \phi \ \left|\ \Q_{\ell,m}\right.\right](x)\quad(\text{as $x\in A$})\\
 & = \lim_{m\to -\infty} \lim_{\ell\to -\infty} \EE_{\mu}\left[ \phi \ \left|\ \ci_T\vee\Q_{\ell,m}\right.\right](x)\quad (\text{by Lemma~\ref{lemma:c_exp}}) \\
 & = \EE_\mu\bigl[\phi\,|\,\Pi(X,\ca_k,\mu,T)\bigr] (x) \quad (\text{by~\eqref{eq:convpinsker}}).
\end{align*}

Now, for each $\phi\in\Phi$, we can fix a Borel-measurable (defined everywhere) version $c_k^\phi$ of the conditional expectation  $\EE_\mu\bigl[\phi\,|\,\Pi(X,\ca_k,\mu,T)\bigr]$. The above equalities ensure that for $\mu$-almost every $x$, we have for all $\phi\in\Phi$,
\begin{equation}
\label{eq:equalityPinskerInPhi}
 c_k^\phi(x) = \EE_{\mu_x}\bigl[\phi\,|\,\Pi(X,\ca_k,\mu_x ,T)\bigr] (x) = \EE_\mu\bigl[\phi\,|\,\Pi(X,\ca_k,\mu,T)\bigr] (x).
\end{equation}
%
%
To conclude the proof for the factor $\ca_k$, all that remains is to define $\cc_k$ as the sub-sigma-algebra generated by the countable family of functions $c_k^\phi$, $\phi\in\Phi$.

The general case follows from a straightforward application of Lemma~\ref{lemma:pinskerlimit}. We can obtain the required sub-sigma-algebra $\cc$ by setting
$$ \cc := \bigvee_{k\ge1}\cc_k = \sigma\Bigl(\bigl(c_k^\phi\bigl)_{k\ge1,\phi\in\Phi}\Bigr). $$

Now, it remains to prove~\eqref{eq:equalityPinsker}. By Lemma~\ref{lemma:pinskerlimit} together with an application of the martingale convergence theorem, we can pass to the limit as $k\to\infty$ in~\eqref{eq:equalityPinskerInPhi} to get the validity of~\eqref{eq:equalityPinsker} for $f\in\Phi$. Let us consider the case $f\in L^\infty(\mu)$: then there exists a sequence $(\phi_\ell)$ of elements of $\Phi$ converging in $L^1$ to $f$, and it can be arranged so that $\|\phi_\ell\|_\infty\leq \|f\|_\infty+1$ for all $\ell$. Then, we also have
$$
\EE_\mu[\phi_\ell\,|\,\cc] \tend{\ell}{\infty} \EE_\mu[f\,|\,\cc] \quad\text{ in }L^1(\mu).
$$
Passing to a subsequence if necessary, we can also assume that the convergence of $\phi_\ell$ to $f$ and the convergence of $\EE_\mu[\phi_\ell\,|\,\cc]$ to $\EE_\mu[f\,|\,\cc]$ holds pointwise, $\mu$-almost-surely. But since we already know that~\eqref{eq:equalityPinsker} is valid in $\Phi$, we get that for $\mu$-almost every $x$,
$$
\EE_{\mu_x}[\phi_\ell\,|\,\cc](x) \tend{\ell}{\infty} \EE_\mu[f\,|\,\cc](x).
$$
Then, for $\mu$-almost every $y\in X$, we have $\mu_y(B_y)=1$, where $B_y$ is the set of $x\in X$ satisfying:
\begin{itemize}
 \item $\mu_x=\mu_y$,
 \item $\phi_\ell(x)  \tend{\ell}{\infty} f(x)$,
 \item $\EE_{\mu_x}[\phi_\ell\,|\,\cc](x) \tend{\ell}{\infty} \EE_\mu[f\,|\,\cc](x)$.
\end{itemize}
But, for any bounded $\cc$-measurable function $g$, we have
$$
\EE_{\mu_y}[\phi_\ell g] = \EE_{\mu_y}\Bigl[\EE_{\mu_y}[\phi_\ell\,|\,\cc]\, g\Bigr].
$$
If $\mu_y(B_y)=1$, we can pass to the limit as $\ell\to\infty$ and we get by the dominated convergence theorem that
$$ \EE_{\mu_y}[f g] = \EE_{\mu_y}\Bigl[\EE_{\mu}[f\,|\,\cc]\, g\Bigr],$$
which ensures that for $\mu_y$-almost every $x$, we have
$$
\EE_{\mu_x}[f\,|\,\cc](x) = \EE_{\mu_y}[f\,|\,\cc](x) = \EE_{\mu}[f\,|\,\cc] (x).
$$
This proves the validity of~\eqref{eq:equalityPinsker} for $f\in L^\infty(\mu)$. For a general $f\in L^1(\mu)$, we can assume w.l.o.g. that $f$ takes its values in $\RR_+$, and then $f$ is the pointwise limit of the non-decreasing sequence $(f_n)_{n\ge 1}$ where for all $x$, $f_n(x):=\min\{n,f(x)\}$. Then, since each $f_n$ satisfy~\eqref{eq:equalityPinsker}, we get as an easy consequence of the monotone convergence theorem that, for $\mu$-almost every $x$, we have
$$
\EE_{\mu_x}[f\,|\,\cc](x) = \lim_n \EE_{\mu_x}[f_n\,|\,\cc](x) = \lim_n \EE_{\mu}[f_n\,|\,\cc] (x) =  \EE_{\mu}[f\,|\,\cc] (x).
$$
This concludes the proof.
\end{proof}

{\bf Acknowledgments}: Our special thanks go to Tim Austin for his interest and contribution at different stages of writing this article  (Section~\ref{s:RelErg} is entirely due to him). We also thank Adam Kanigowski for discussions on the subject and his contribution (the counterexample in Section~\ref{s:aoma} has been obtained jointly with him). We also thank Florian Richter for some discussions concerning the class Erg$^\perp$. We want to express our gratitude to the referee for a contribution to the enhancement of the present work.

\vspace{2ex}

\noindent
Faculty of Mathematics and Computer Science\\
Nicolaus Copernicus University, \\
Chopin street 12/18, 87-100 Toru\'n, Poland\\
gorska@mat.umk.pl, mlem@mat.umk.pl

\vspace{2ex}

\noindent
Laboratoire de Math\'ematiques Raphael Salem, CNRS – Universit\'e de Rouen Normandie\\
Avenue de l’Universit\'e – 76801 Saint \'Etienne du Rouvray, France\\
Thierry.de-la-Rue@univ-rouen.fr

\end{document}